\documentclass[11pt]{amsart}
\usepackage{amsmath}
\usepackage{bbding}
\usepackage{tikz}
\usepackage{amsmath,amssymb}
\theoremstyle{plain}
\newtheorem{theorem}{Theorem}[section]
\newtheorem{lemma}[theorem]{Lemma}

\theoremstyle{definition}

\newtheorem{definition}[theorem]{Definition}

\newtheorem{example}[theorem]{Example}
\numberwithin{equation}{section}
\def\be{\begin{equation}}
\def\ee{\end{equation}}

\begin{document}

\title[Uniformly Degenerate Elliptic Equations]
{Uniformly Degenerate Elliptic Equations\\ with Varying Characteristic Exponents}

\author[Han]{Qing Han}
\address{Department of Mathematics\\
University of Notre Dame\\
Notre Dame, IN 46556, USA} \email{qhan@nd.edu}
\author[Xie]{Jiongduo Xie}
\address{Beijing International Center for Mathematical Research\\
%School of Mathematical Sciences\\
Peking University\\
Beijing, 100871, China}
\email{2001110018@stu.pku.edu.cn}

\begin{abstract}
In this paper, we study the regularity of solutions to 
uniformly degenerate elliptic equations in bounded domains 
under the condition that the characteristic polynomials have 
varying characteristic exponents. 
\end{abstract}

%\thanks{The first author acknowledges the support of NSF Grant DMS-2305038
%and the second author acknowledges the support of 
%National Key R\&D Program of China Grant 2020YFA0712800.}
%\date{\today}
\maketitle

\section{Introduction}\label{sec-Intro}

Uniformly degenerate elliptic equations appear or can be uncovered 
in many problems and have been extensively studied. 
These problems include the Loewner-Nirenberg problem and its nonlinear versions \cite{Loewner&Nirenberg1974,AvilesMcOwen1988,Mazzeo1991,ACF1992CMP, Kichenassamy2004JFA, Kichenassamy2005JFA, GonzalezLi2018, LiNguyenXiong2023}, 
complex Monge-Amp\`ere equations
\cite{Fefferman1976, ChengYau1980CPAM, LeeMelrose1982, HanJiang2024}, 
affine hyperbolic spheres 
\cite{ChengYau1977,JianWang2013JDG,JianWangZhao2017JDE,JianLuWang2022China}, 
the Bergman Laplacian \cite{Graham1983-1, Graham1983-2}, 
minimal graphs in the hyperbolic space  
\cite{Lin1989Invent,Tonegawa1996MathZ,Graham&Witten1999,Lin2012Invent,HanShenWang16CalVar,HanJiang2023}, 
singular stochastic control problems 
\cite{LasryLions1989, LeonoriPorretta2011, JianLuWang2022China}, 
conformally compact Einstein metrics
\cite{GrahamLee1991, Lee1995, OB2000, Graham2000, OB2002, Chrusciel2005, Lee2006, Hellimell2008, OB2016}, 
and proper harmonic maps between hyperbolic spaces and other spaces
\cite{LiTam1991,LiTam1993,LiTam1993Ind,Donnelly1994,LiNi2000,LiSimon2007,Yin2007,ChenLi2023}. 
%Similar computations can be preformed for some other geometric problems. 
%In other words, by an appropriate choice of new unknown function, 
%such as $v$ in \eqref{eq-ch1-LN-Relation}, 
%we can uncover a uniformly degenerate elliptic operator in the underlying nonlinear problem. 
After introducing the degenerate elliptic equation for the graphs of the minimal surfaces in the hyperbolic space, 
Lin \cite{Lin1989Invent} wrote, \lq\lq It may suggest a general study of linear or non-linear degenerate elliptic equations. 
For example, one would like to know for which class of such equations the solutions are as smooth as the boundary data. 
In the case that the loss in derivatives does occur, one would like to know exactly how much is lost."
In this paper, we prove general regularity results for uniformly degenerate linear elliptic equations. 
Methods in this paper can be modified to treat nonlinear equations. 

The study of the regularity of the solutions to the uniformly degenerate elliptic equations
consists of two parts. 
In the first part, we study the optimal regularity up to a certain level 
under an appropriate regularity assumption of the boundary. 
Here, the optimality refers to the fact that the regularity of solutions cannot be improved 
beyond a level determined by the equation 
even if the boundary is smooth. 
This is a special feature of the uniformly degenerate elliptic equations. 
In the second part, we study behaviors of solutions near the boundary 
beyond the optimal regularity level in the form of the polyhomogeneous expansions. 
Such expansions were originally established by the method of microlocal analysis 
by Melrose \cite{Melrose1981}, Lee and Melrose \cite{LeeMelrose1982}, 
Mazzeo and Melrose \cite{MazzeoMelrose1987}, and Mazzeo \cite{Mazzeo1991}.

In this paper, we study the higher regularity by a different approach. 
We will decompose solutions of uniformly degenerate elliptic equations into 
a finite sum or an infinite sum, involving regular factors and explicit singular factors. 
We are able to prove the uniform and absolute convergence in the case of the infinite sum 
without any analyticity assumption.

Let $\Omega$ be a bounded domain in $\mathbb R^n$. 
For a uniformly degenerate linear elliptic operator $L$ to be introduced, we consider the Dirichlet problem 
for the equation 
%For some given function $f$ at least continuous in $\Omega$, we consider the equation 
\begin{equation}\label{eq-ch2-basic-equation}
Lu=f\quad\text{in }\Omega.\end{equation}
We will study asymptotic expansions near the boundary and 
the global regularity of the solutions under the condition that 
the characteristic exponents are not constant. 

We first introduce a notion of defining functions. 
Let $\Omega$ be a bounded domain in $\mathbb R^n$ with a $C^{1}$-boundary $\partial\Omega$ 
and $\rho$ be a $C^{1}(\bar\Omega)$-function. 
%, for some integer $k\ge 1$ and some constant $\alpha\in [0,1]$. 
Then, $\rho$ is a {\it defining function} of $\Omega$
if $\rho>0$ in $\Omega$ and $\rho=0$ and $\nabla \rho\neq 0$ on $\partial\Omega$. We always require 
$$|\nabla\rho|=1\quad\text{on }\partial\Omega.$$ Then, $\nabla \rho|_{\partial\Omega}$ 
is the inner unit normal vector along $\partial\Omega$. 
In the following, we will fix a defining function $\rho$, 
with possibly a higher regularity. 

Consider the operator 
\begin{align}\label{eq-ch2-def-operator}
L=\rho^2a_{ij}\partial_{ij}+\rho b_i\partial_i+c\quad\text{in }\Omega,\end{align}
where $a_{ij}$, $b_i$, and $c$ are continuous functions in $\bar\Omega$ satisfying $a_{ij}=a_{ji}$ and, 
for any $x\in \Omega$
and $\xi\in\mathbb R^n$, 
\begin{align}\label{eq-ch2-ellipticity}
\lambda|\xi|^2\le a_{ij}(x)\xi_i\xi_j\le\Lambda|\xi|^2,\end{align}
%and 
%\begin{align}\label{eq-bound-c}c\le 0\quad\text{in }\Omega,\end{align}
for some positive constants $\lambda$ and $\Lambda$. 
We note that the operator $L$ is not uniformly elliptic in $\bar \Omega$ and, in fact, is 
degenerate along the boundary $\partial\Omega$, due to the presence of the factor $\rho^2$ in the 
second-order terms. The operator $L$ is called to be {\it uniformly degenerate elliptic}. 
%We need to introduce additional assumptions to offset the impact of degeneracy on the boundary. 

%In view of Examples \ref{exa-ch2-solutions-not-regular}
%and \ref{exa-ch2-solutions-not-regular-non-constant}, 
We now introduce two important concepts, 
characteristic polynomials and characteristic exponents. 
%In view of Example \ref{exa-ch2-solutions-not-regular}, we introduce two important concepts, 
%characteristic polynomials and characteristic exponents. 
%%We now introduce an important function defined on boundary. 
Take an arbitrary $\mu\in\mathbb R$. A straightforward computation yields 
$$\rho^{-\mu}L\rho^\mu=\mu(\mu-1)a_{ij}\partial_i\rho\partial_j\rho+\mu b_i\partial_i\rho+c
+\mu\rho a_{ij}\partial_{ij}\rho.$$
Note that $\nabla \rho|_{\partial\Omega}$ is the inner unit normal 
vector along $\partial\Omega$. 
Under the extra condition that $\rho D^2\rho\in C(\bar\Omega)$ with 
$\rho D^2\rho=0$ on $\partial\Omega$,  we have 
$$\rho^{-\mu}L\rho^\mu\big|_{\partial\Omega}=\mu(\mu-1)a_{ij}\nu_i\nu_j
+\mu b_i\nu_i+c,$$
where $\nu=(\nu_1, \cdots, \nu_n)$ is the inner unit normal along $\partial\Omega$. 
We now define 
\begin{align}\label{eq-ch2-definition-P-mu}
P(\mu)=\mu(\mu-1)a_{ij}\nu_i\nu_j
+\mu b_i\nu_i+c
\quad\text{on }\partial\Omega.\end{align}
For each fixed point on $\partial\Omega$, $P(\mu)$ is a quadratic polynomial in $\mu$ 
with a positive leading coefficient. The polynomial $P$ is referred to as the 
{\it characteristic polynomial} of $L$ and its roots as {\it characteristic exponents}
or {\it indicial roots}.

Next, we introduce some conditions on the boundary to offset the effect of the degeneracy of the operator. 
First, 
we assume 
$$P(0)=c<0\quad\text{on }\partial\Omega.$$ 
Hence, the characteristic polynomial $P(\mu)$ has two real roots, a positive one and a negative one,
and any $\mu$ with a negative value of $P(\mu)$ has to be between these two roots. 
Well-known examples demonstrate that characteristic exponents provide obstructions to the higher regularity of solutions
through singular factors.
Refer to Examples \ref{exa-ch2-solutions-not-regular} 
and \ref{exa-ch2-solutions-not-regular-non-constant} for details.
For some integer $k\ge 0$ and constant $\alpha\in (0,1)$, a general regularity 
of $C^{k,\alpha}$ for solutions is possible only under the condition 
$P({k+\alpha})\le 0$ on $\partial\Omega$.
%In other words, $k+\alpha$ is between two characteristic exponents. 
If $k+\alpha$ is bigger than the positive characteristic exponent, 
solutions will not have desired regularity. 
%This is clear from Example \ref{exa-ch2-solutions-not-regular}.
In this paper, we will study decompositions of solutions near the boundary 
and identify factors that prohibit solutions from the higher regularity.  

\iffalse%%%%%%%%
We assume that 
\begin{align}\label{eq-ch4-assumption-characteristic-roots-1}
&\text{the characteristic exponents are given by %{\it constants} 
$\underline{m}<0$ 
and {\color{red} $\overline{m}>1$}.}
\end{align}
%In other words, we have $P(\underline{m})=P(\overline{m})=0$ on $\partial\Omega$. 
We emphasize that $\underline{m}$ and $\overline{m}$ are functions on $\partial\Omega$. 
The positive characteristic exponent $\overline{m}$ provides an obstacle to the higher regularity. 
For our main results, we assume that 
\begin{align}\label{eq-ch4-assumption-characteristic-roots-2}
&\text{$\overline{m}$ does not take integer values on $\partial\Omega$.}
\end{align}
Equivalently, $\overline{m}$ is strictly between two consecutive integers on $\partial\Omega$. 
Set 
$$\gamma=\overline{m}-[\overline{m}].$$
Then, $0<\gamma<1$ on $\partial\Omega$. 
For the $C^{k,\alpha}$ regularity of solutions near the boundary, we need to discuss 
separately the cases that 
the H\"older index $\alpha$ is larger or smaller than $\gamma$. 
\fi%%%%%%%%%

Let $f$ be a continuous function in $\bar\Omega$. Consider the Dirichlet problem given by 
\eqref{eq-ch2-basic-equation} and 
\begin{align}
\label{eq-ch2-Dirichlet}
%\begin{split}Lu&=f\quad\text{in }\Omega,\\
u=\frac{f}{c}\quad\text{on }\partial\Omega. %\end{split}
\end{align}
The boundary value $f/c$ in \eqref{eq-ch2-Dirichlet} is not arbitrarily prescribed and 
is determined by the equation itself, specifically by the ratio of $f$ and $c$. 
This is due to the degeneracy along the boundary. 
In fact, normal derivatives of solutions on the boundary up to a certain order are also determined by the equation. 

We will prove two types of results concerning the regularity with a degree 
beyond the positive characteristic exponent. 
The first type is in the category of finite differentiability, while the second one is 
in the smooth category. 
In the category of finite differentiability, we need to discuss 
separately the cases that 
the H\"older index is larger or smaller than the decimal part of the positive characteristic exponent. 
%In the following results, $\nu$ is the inner unit normal vector to $\partial\Omega$. 

This paper is organized as follows. 
In Section \ref{sec-Constant-Exponents}, we present our regularity results for constant characteristic exponents. 
In Section \ref{sec-Varying-Exponents}, we present our regularity results for varying characteristic exponents. 
Theorem \ref{Thm-ch4-Linear-MainThm-curved-combined}
and Theorem \ref{Thm-ch4-Linear-MainThm-curved-combined-smooth} are the main results in this paper. 
In Section \ref{sec-flat-boundary}, we reformulate the boundary value problem 
\eqref{eq-ch2-basic-equation} and \eqref{eq-ch2-Dirichlet} 
near a flat portion of the boundary and present several results on the optimal regularity. In Section \ref{sec-Large-Holder-Indices},
we study the case that the H\"older index is larger than the decimal part of 
the positive characteristic exponent. 
In Section \ref{sec-Small-Holder-Indices},
we study the case that the H\"older index is smaller than the decimal part of 
the positive characteristic exponent. 
In Section \ref{sec-Appen-CalculusL},
we prove several results concerning the integrability of integrals involving parameters.

{\it Acknowledgements:} The first author acknowledges the support of NSF Grant DMS-2305038.
The second author acknowledges the support of 
National Key R\&D Program of China Grant 2020YFA0712800.

\section{Constant Characteristic Exponents}\label{sec-Constant-Exponents} 

In this section, we present our regularity results for constant characteristic exponents. 
We first consider an example.

\begin{example}\label{exa-ch2-solutions-not-regular} 
Let $\Omega$ be a bounded domain in $\mathbb R^n$, 
with a $C^\infty$-boundary $\partial\Omega$ and a $C^\infty(\bar\Omega)$-defining function $\rho$. 
For some constants $a$, $b$, and $c$ with $a>0$, set 
\begin{align}\label{eq-ch2-def-operator-example}L=a\rho^2\Delta+b\rho\nabla\rho\cdot\nabla+c.\end{align}
In other words, we take $a_{ij}=a\delta_{ij}$ and $b_i=b\partial_i\rho$ in \eqref{eq-ch2-def-operator}. 
%For the operator $L$ in \eqref{eq-ch2-def-operator-example}, 
%we have $a_{ij}=a\delta_{ij}$ and $b_i=b\partial_i\rho$ if we put it in the form \eqref{eq-ch2-def-operator}. 
Let $\nu$ be the interior unit normal vector along $\partial\Omega$. 
Then, $a_{ij}\nu_i\nu_j=a$ and $b_i\nu_i=b$ on $\partial\Omega$. 
Hence, the characteristic polynomial of $L$ is given by 
$$P(\mu)=a\mu(\mu-1)+b\mu +c.$$
Let $s$ be a positive root of $P(\mu)$; namely 
\begin{equation}\label{eq-ch2-root-condition}
as(s-1)+bs +c=0.
\end{equation}
We consider two cases. 

%\smallskip 
{\it Case 1. We assume that $s$ is not an integer.}  
Take a function $\psi\in C^\infty(\bar\Omega)$. 
A straightforward calculation implies
\begin{align*} L(\psi\rho^s)=f,\end{align*}
where 
\begin{align}\label{eq-ch2-example-expression-f}\begin{split}
f&=\rho^s\big\{\big[as(s-1)+bs+c\big]\psi+\big[as(s-1)+bs\big](|\nabla \rho|^2-1)\psi\\
&\qquad+\big[as\psi\Delta\rho+(2as+b)\nabla\rho\cdot\nabla\psi\big]\rho+a\Delta\psi\rho^{2}\big\}.
\end{split}\end{align}
All functions in the right-hand side are smooth except $\rho^s$. 
Note that first term in the right-hand side is zero  by 
\eqref{eq-ch2-root-condition}. 
By $|\nabla \rho|^2=1$ on $\partial\Omega$, 
we can write $f=\eta\rho^{s+1}$, 
for some function $\eta\in C^\infty(\bar\Omega)$. 
In summary, $u=\psi\rho^s$ is a solution of \eqref{eq-ch2-basic-equation}, for some $f=\eta\rho^{s+1}$. Write 
$s=k+\gamma$, for some integer $k\ge 0$ and constant $\gamma\in (0,1)$. 
For such $u$ and $f$, we have $f\in C^{k+1, \gamma}(\bar\Omega)$ and  
$u\in C^{k,\gamma}(\bar\Omega)$, but $u\notin C^{k,\beta}(\bar\Omega)$, 
for any $\beta\in (\gamma,1)$. 

In fact, by choosing $\psi$ appropriately, we can improve the regularity 
of $f$ in \eqref{eq-ch2-example-expression-f}. For a function 
$\psi_0\in C^\infty(\bar\Omega)$ and a positive integer $m$, consider 
$$\psi=\sum_{i=0}^m\psi_i\rho^i.$$
We now assume in addition that, for $\mu=s+1, \cdots, s+m$, 
$$a\mu(\mu-1)+b\mu +c\neq 0.$$
%$$a\mu(\mu-1)|\nabla \rho|^2+b\mu |\nabla\rho|^2+c\neq 0\quad\text{in }\bar\Omega.$$
Substituting $\psi$ in \eqref{eq-ch2-example-expression-f} and choosing $\psi_1, \cdots, \psi_m$ successively, 
we can write $f=\eta \rho^{s+m+1}$ for some  $\eta\in C^\infty(\bar\Omega)$. Hence, 
$f\in C^{k+m+1, \gamma}(\bar\Omega)$ and  
$u\in C^{k,\gamma}(\bar\Omega)$, but $u\notin C^{k,\beta}(\bar\Omega)$, 
for any $\beta\in (\gamma,1)$. 

{\it Case 2. We assume that $s$ is an integer.} 
Take a function $\psi\in C^\infty(\bar\Omega)$. 
A straightforward calculation implies
\begin{align*} L(\psi\rho^s\log\rho)=f,\end{align*}
where 
\begin{align*}%\label{eq-ch2-example-expression-f-integer}\begin{split}
f&=\rho^s\log \rho\big\{\big[as(s-1)+bs+c\big]\psi+\big[as(s-1)+bs\big](|\nabla \rho|^2-1)\psi\\
&\qquad+\big[as\psi\Delta\rho+(2as+b)\nabla\psi\cdot\nabla\rho\big]\rho
+a\Delta\psi\rho^{2}\big\}\\
&\qquad +\rho^{s}\big\{\big[a(2s-1)+b\big]\psi|\nabla\rho|^2
+\big[a\psi\Delta\rho+2a\nabla\psi\cdot\nabla\rho\big]\rho\big\}.%\end{split}
\end{align*}
All functions in the right-hand side are smooth except $\log\rho$. 
Note that the first term in the right-hand side %of \eqref{eq-ch2-example-expression-f} 
is zero 
by \eqref{eq-ch2-root-condition}.
By $|\nabla \rho|^2=1$ on $\partial\Omega$, 
we can write $f=\eta_1\rho^{s+1}\log\rho+\eta_2$, 
for some functions $\eta_1, \eta_2\in C^\infty(\bar\Omega)$. 
In summary, $u=\psi\rho^s\log \rho$ is a solution of \eqref{eq-ch2-basic-equation}, 
for the above $f$. For such $u$ and $f$, we have 
$f\in C^{s, \alpha}(\bar\Omega)$ and  
$u\in C^{s-1,\alpha}(\bar\Omega)$ for any $\alpha\in (0,1)$, but $u\notin C^{s}(\bar\Omega)$.
Proceeding similarly as in Case 1, we can raise the regularity of $f$ but maintain the regularity of $u$. 
\end{example} 

Example \ref{exa-ch2-solutions-not-regular} demonstrates that the positive characteristic exponent provides an obstruction 
to the boundary regularity of solutions $u$, regardless of the regularity of $f$. 
Different types of singularity arise depending on whether the positive characteristic exponents are integer. 
%We now compare Example \ref{exa-ch2-solutions-not-regular-non-constant} with 
%Example \ref{exa-ch2-solutions-not-regular}.  
In Case 1, 
the positive characteristic exponent $s$ is a non-integer constant
and the only singular factor is given by $\rho^\gamma$ for the decimal part $\gamma\in (0,1)$ of $s$. 
In Case 2, 
the positive characteristic exponent $s$ is an integer
and the only singular factor is given by $\log\rho$. 
The solution $u$ constructed in these cases are given by 
$$u=\sum_{i=0}^m\psi_i\rho^{s+i},$$
if $s$ is not an integer, and 
$$u=\sum_{i=0}^m\psi_{i}\rho^{s+i}\log \rho,$$
if $s$ is an integer. 
The aim of the paper is to demonstrate that $\rho^{s+i}$, for non-integer $s$, and $\rho^{s+i}\log \rho$, 
for integer $s$,  
provide sole obstructions to the higher regularity of solutions.

%If they are non-integer, solutions have monomials of non-integer powers. 
%If they are integers, solutions have logarithmic factors. 

In the following, we denote by $\underline{m}$ and $\overline{m}$ the characteristic exponents 
with $\underline{m}<0<\overline{m}$ and by $\nu$ the inner unit normal vector to $\partial\Omega$. 
We will prove results concerning the regularity with a degree beyond the positive exponent $\overline{m}$. 
In view of Example \ref{exa-ch2-solutions-not-regular}, 
we need to distinguish the case that  $\overline{m}$ is not an integer  and the
case that $\overline{m}$ is an integer.
%In the following results, $\nu$ is the inner unit normal vector to $\partial\Omega$. 

We first consider the case that the positive characteristic exponent 
$\overline{m}$ is a non-integer constant and write 
$\overline{m}=[\overline{m}]+\gamma$, 
for some constant $\gamma\in (0,1)$. 
For a nonnegative integer $k$ and a constant $\alpha\in(0,1)\setminus\{\gamma\}$, we set
\begin{align*}
k_{\ast}&=k,\  \epsilon_{\ast}=\alpha-\gamma\quad\text{if }\gamma<\alpha,\\
k_{\ast}&=k-1,\  \epsilon_{\ast}=\alpha+1-\gamma\quad\text{if }\alpha<\gamma.
\end{align*}
We always consider the case $k+\alpha>\overline{m}$.

\begin{theorem}\label{Thm-ch4-Linear-MainThm-curved-combined-constant}
Let $\Omega$ be a bounded domain in $\mathbb R^n$, with  
a $C^1$-boundary $\partial\Omega$ and a defining function $\rho\in C^1(\bar\Omega)$, 
and $u\in C(\bar \Omega)\cap C^{2}(\Omega)$ be 
a solution of the Dirichlet problem \eqref{eq-ch2-basic-equation} and \eqref{eq-ch2-Dirichlet}.  
For some $x_0\in\partial\Omega$, 
suppose that the characteristic polynomial $P$ has a negative constant root  $\underline{m}$
and a positive constant root $\overline{m}$ on $\partial\Omega\cap B_R(x_0)$, 
with $\overline{m}=[\overline{m}]+\gamma$ for some constant $\gamma\in (0,1)$, 
and that $k$ is an integer and $\alpha\in(0,1)$ is a constant such that 
either $k\geq [\overline{m}]\ge 1$, $\gamma<\alpha$  
or $k\geq [\overline{m}]+1$, $\alpha<\gamma$. 
Assume $\partial\Omega\cap B_R(x_0)\in C^{k+2,\alpha}$, 
$\rho\in C^{k+2,\alpha}(\bar\Omega\cap B_R(x_0))$, 
with $\rho\nabla^{k+3}\rho\in C^{\alpha}(\bar\Omega\cap B_R(x_0))$ 
and $\rho\nabla^{k+3}\rho=0$ on $\partial\Omega \cap B_R(x_0)$, 
$a_{ij}, b_i, c\in C^{k+1, \alpha}(\bar \Omega\cap B_R(x_0))$ satisfying \eqref{eq-ch2-ellipticity}, and 
$f\in C^{k,\alpha}(\bar \Omega\cap B_R(x_0))$. Then,
\begin{equation}\label{eq-ch4-decomposition-curved-non-int-constant}
u=v+w\rho^{\gamma} \quad\text{in } \Omega\cap B_R(x_0),
\end{equation}
for some functions $v\in C^{k, \alpha}(\bar\Omega\cap B_{R}(x_0))$  
and $w\in C^{k_*, \epsilon}(\bar\Omega\cap B_{R}(x_0))$,
for any constant $\epsilon\in (0,\epsilon_*)$, with
$\partial_\nu^iw=0$ on $\partial\Omega\cap B_R(x_0)$, for $i=0, \cdots, [\overline{m}]-1$. 
Moreover, for any $\epsilon\in (0,\epsilon_*)$, 
\begin{align*}
&|v|_{C^{k, \alpha}(\bar\Omega\cap B_{R/2}(x_0))}
+|w|_{C^{k_{\ast}, \epsilon}(\bar\Omega\cap B_{R/2}(x_0))}\\
&\qquad\leq C\big\{|u|_{L^{\infty}(\Omega\cap B_R(x_0))}
+|f|_{C^{k, \alpha}(\bar\Omega\cap B_R(x_0))}\big\},
\end{align*}
where $C$ is a positive constant depending only on $n$, $\lambda$, $k$, $\alpha$, $\epsilon$, 
$\underline{m}$, $\overline{m}$, $\Omega\cap B_R(x_0)$, 
the $C^{k+2, \alpha}$-norm of $\rho$ in $\bar\Omega\cap B_R(x_0)$, 
the $C^{\alpha}$-norm of $\rho\nabla^{k+3}\rho$ in $\bar\Omega\cap B_R(x_0)$,
and the $C^{k+1, \alpha}$-norms of $a_{ij}, b_i, c$ in $\bar\Omega\cap B_R(x_0)$.
If, in addition, $\partial_\nu^{[\overline{m}]} w=0$ on $\partial \Omega\cap B_R(x_0)$, 
then $u \in C^{k,\alpha}(\bar\Omega\cap B_R(x_0))$.
\end{theorem}

As demonstrated by \eqref{eq-ch4-decomposition-curved-non-int-constant}, 
an obstacle to the higher regularity of $u$ is provided 
by the factor $\rho^{\gamma}$ under the assumption that $\overline{m}$ is not an integer. 
In this case,  the function $v$ 
in \eqref{eq-ch4-decomposition-curved-non-int-constant} has an optimal regularity 
and $w$ has a loss in the H\"older regularity, from $C^{k,\alpha}$
in the assumption 
to $C^{k,\epsilon}$ for an arbitrary constant $\epsilon\in (0,\alpha-\gamma)$ if $\alpha>\gamma$
and to $C^{k-1,\epsilon}$ for an arbitrary constant $\epsilon\in (0,\alpha+1-\gamma)$ if $\alpha<\gamma$. 
We point out that the function $w$ in 
\eqref{eq-ch4-decomposition-curved-non-int-constant}  
has a different regularity, depending on whether $\alpha\in (0,\gamma)$ or $\alpha\in (\gamma,1)$. 
This is natural due to the presence of the factor $\rho^\gamma$ in  
\eqref{eq-ch4-decomposition-curved-non-int-constant}. 
The reason is as follows. We will factorize $\rho^\gamma$ from $\rho^\alpha$. 
If $\alpha>\gamma$, we are able to do this  since 
$\alpha=\gamma+(\alpha-\gamma)$. If $\alpha<\gamma$,  we need to borrow 
a factor of $\rho$ somewhere and use $1+\alpha=\gamma+(1+\alpha-\gamma)$. 
The condition $k+\alpha>\overline{m}$ also put different requirements on $k$. 
If $\alpha>\gamma$, we require $k\ge [\overline{m}]$; 
but if $\alpha<\gamma$, we have to require $k\ge [\overline{m}]+1$.

\begin{theorem}\label{Thm-ch4-Linear-MainThm-curved-combined-smooth-constant}
Let $\Omega$ be a bounded domain in $\mathbb R^n$, with  
a $C^1$-boundary $\partial\Omega$ and a defining function $\rho\in C^1(\bar\Omega)$, 
and $u\in C(\bar \Omega)\cap C^{2}(\Omega)$ be 
a solution of the Dirichlet problem \eqref{eq-ch2-basic-equation} and \eqref{eq-ch2-Dirichlet}.  
For some $x_0\in\partial\Omega$, 
suppose that the characteristic polynomial $P$ has a negative constant root  $\underline{m}$
and a positive constant root $\overline{m}$ on $\partial\Omega\cap B_R(x_0)$, 
with $\overline{m}=[\overline{m}]+\gamma$ for some constant $\gamma\in (0,1)$.
Assume $\partial\Omega\cap B_R(x_0)\in C^{\infty}$, 
$\rho\in C^{\infty}(\bar\Omega\cap B_R(x_0))$, 
$a_{ij}, b_i, c\in C^{\infty}(\bar \Omega\cap B_R(x_0))$ satisfying \eqref{eq-ch2-ellipticity}, and 
$f\in C^{\infty}(\bar \Omega\cap B_R(x_0))$. Then, 
\begin{equation}\label{eq-ch4-decomposition-curved-int-non-small-alpha-smooth-constant}
u=v+w\rho^{\overline{m}}\quad\text{in }\Omega\cap B_{R}(x_0),\end{equation}
for some functions $v, w\in C^{\infty}(\bar\Omega\cap B_{R}(x_0))$. 
%, with $\partial_\nu^iw=0$ on $\partial\Omega\cap B_R$, for $i=0, \cdots, [\overline{m}]-1$. 
Moreover, if %$\partial_\nu^{[\overline{m}]}w=0$ 
$w=0$ on $\partial\Omega\cap B_R(x_0)$, then 
$u\in C^{\infty}(\bar\Omega\cap B_{R}(x_0))$. 
\end{theorem}

Theorem \ref{Thm-ch4-Linear-MainThm-curved-combined-smooth-constant} does not follow from 
Theorem \ref{Thm-ch4-Linear-MainThm-curved-combined-constant} trivially. 
We point out that the functions $v$ and $w$ in 
\eqref{eq-ch4-decomposition-curved-non-int-constant}
depend on $k$. When $k$ increases, these functions change. 
Under the assumption that all coefficients and nonhomogeneous terms are smooth, 
we actually obtain sequences of functions 
$\{v_k\}_{k=[\overline{m}]}^{\infty}$ and $\{w_k\}_{k=[\overline{m}]}^{\infty}$ such that 
each pair $v_k$ and $w_k$ satisfy 
\eqref{eq-ch4-decomposition-curved-non-int-constant}
and the stated properties. 
Although the regularity of $v_k$ and $w_k$ improves as $k$ increases, 
it is not clear whether these sequences $\{v_k\}_{k=[\overline{m}]}^{\infty}$ and $\{w_k\}_{k=[\overline{m}]}^{\infty}$
converge to some smooth functions $v$ and $w$, respectively. 
Extra work is needed to prove Theorem \ref{Thm-ch4-Linear-MainThm-curved-combined-smooth-constant} 
from Theorem \ref{Thm-ch4-Linear-MainThm-curved-combined-constant}. 

We also point out that $w$ in \eqref{eq-ch4-decomposition-curved-non-int-constant} 
means differently from that in 
\eqref{eq-ch4-decomposition-curved-int-non-small-alpha-smooth-constant}. 
If we factorize $\rho^{[\overline{m}]}$ from $w$ in \eqref{eq-ch4-decomposition-curved-non-int-constant}, 
the new $w$ as in 
\eqref{eq-ch4-decomposition-curved-non-int-constant}
will have a lower regularity in the context of the finite regularity. 

\smallskip 

We next consider the case that the positive characteristic exponent $\overline{m}$ is an integer.

\begin{theorem}\label{Thm-ch4-Linear-MainThm-curved-integer}
Let $\Omega$ be a bounded domain in $\mathbb R^n$, with  
a $C^1$-boundary $\partial\Omega$ and a defining function $\rho\in C^1(\bar\Omega)$, 
and $u\in C(\bar \Omega)\cap C^{2}(\Omega)$ be 
a solution of the Dirichlet problem \eqref{eq-ch2-basic-equation} and \eqref{eq-ch2-Dirichlet}.  
For some $x_0\in\partial\Omega$, 
suppose that the characteristic polynomial $P$ has a negative constant root  $\underline{m}$
and a positive integer constant root $\overline{m}$ on $\partial\Omega\cap B_R(x_0)$. 
For some nonnegative integer $k\ge \overline{m}$
and some $\alpha\in (0,1)$, %with $k+\alpha<\overline{m}+1$, 
assume $\partial\Omega\cap B_R(x_0)\in C^{k+2,\alpha}$, 
$\rho\in C^{k+2,\alpha}(\bar\Omega\cap B_R(x_0))$, 
with $\rho\nabla^{k+3}\rho\in C^{\alpha}(\bar\Omega\cap B_R(x_0))$ 
and $\rho\nabla^{k+3}\rho=0$ on $\partial\Omega \cap B_R(x_0)$, 
$a_{ij}, b_i, c\in C^{k+1, \alpha}(\bar \Omega\cap B_R(x_0))$ satisfying \eqref{eq-ch2-ellipticity}, and 
$f\in C^{k,\alpha}(\bar \Omega\cap B_R(x_0))$.
Then, 
\begin{equation}\label{eq-ch4-decomposition-curved-int}
u=v+w\log\rho\quad\text{in }\Omega\cap B_{R}(x_0),\end{equation}
for some functions  %$v\in C^{k,\epsilon}(\bar\Omega\cap B_{R/2})$ and 
$v\in C^{k, \epsilon}(\bar\Omega\cap B_{r}(x_0))$ and 
$w\in C^{k, \alpha}(\bar\Omega\cap B_{r}(x_0))$, 
for any $\epsilon\in (0,\alpha)$ and any $r\in (0,R)$, with
$\partial_\nu^iw=0$ on $\partial\Omega\cap B_R(x_0)$, for $i=0, \cdots, \overline{m}-1$. 
Moreover, for any $\epsilon\in (0,\alpha)$, 
\begin{align*}
&|v|_{C^{k, \epsilon}(\bar\Omega\cap B_{R/2}(x_0))}
+|w|_{C^{k, \alpha}(\bar\Omega\cap B_{R/2}(x_0))}\\
&\qquad\leq C\big\{|u|_{L^{\infty}(\Omega\cap B_R(x_0))}
+|f|_{C^{k, \alpha}(\bar\Omega\cap B_R(x_0))}\big\},
\end{align*}
where $C$ is a positive constant depending only on $n$, $\lambda$, $k$, $\alpha$, $\epsilon$, 
$\underline{m}$, $\overline{m}$, $\Omega\cap B_R(x_0)$, 
the $C^{k+2, \alpha}$-norm of $\rho$ in $\bar\Omega\cap B_R(x_0)$, 
the $C^{\alpha}$-norm of $\rho\nabla^{k+3}\rho$ in $\bar\Omega\cap B_R(x_0)$,
and the $C^{k+1, \alpha}$-norms of $a_{ij}, b_i, c$ in $\bar\Omega\cap B_R(x_0)$.
If, in addition,  $\partial_\nu^{\overline{m}}w=0$ on $\partial\Omega\cap B_R(x_0)$, then 
$u\in C^{k,\epsilon}(\bar\Omega\cap B_{r}(x_0))$, 
for any $\epsilon\in (0,\alpha)$ and any $r\in (0,R)$. 
\end{theorem}

\begin{theorem}\label{Thm-ch4-Linear-MainThm-curved-combined-smooth-integer}
Let $\Omega$ be a bounded domain in $\mathbb R^n$, with  
a $C^1$-boundary $\partial\Omega$ and a defining function $\rho\in C^1(\bar\Omega)$, 
and $u\in C(\bar \Omega)\cap C^{2}(\Omega)$ be 
a solution of the Dirichlet problem \eqref{eq-ch2-basic-equation} and \eqref{eq-ch2-Dirichlet}.  
For some $x_0\in\partial\Omega$, 
suppose that the characteristic polynomial $P$ has a negative constant root  $\underline{m}$
and a positive integer constant root $\overline{m}$ on $\partial\Omega\cap B_R(x_0)$.
Assume $\partial\Omega\cap B_R(x_0)\in C^{\infty}$, 
$\rho\in C^{\infty}(\bar\Omega\cap B_R(x_0))$, 
$a_{ij}, b_i, c\in C^{\infty}(\bar \Omega\cap B_R(x_0))$ satisfying \eqref{eq-ch2-ellipticity}, and 
$f\in C^{\infty}(\bar \Omega\cap B_R(x_0))$. 
Then, 
\begin{equation}\label{eq-ch4-decomposition-curved-int-smooth}
u=v+w\rho^{\overline{m}}\log\rho\quad\text{in }\Omega\cap B_{R}(x_0),\end{equation}
for some functions  %$v\in C^{k,\epsilon}(\bar\Omega\cap B_{R/2})$ and 
$v, w\in C^{\infty}(\bar\Omega\cap B_{r}(x_0))$, 
for any $r\in (0,R)$. 
%, with $\partial_\nu^iw=0$ on $\partial\Omega\cap B_R$, for $i=0, \cdots, \overline{m}-1$. 
Moreover, if %$\partial_\nu^{\overline{m}}w=0$ 
$w=0$ on $\partial\Omega\cap B_R(x_0)$, then 
$u\in C^{\infty}(\bar\Omega\cap B_{r}(x_0))$, 
for any $r\in (0,R)$. 
\end{theorem}

\section{Varying Characteristic Exponents}\label{sec-Varying-Exponents} 

In this section, we present our regularity results for varying characteristic exponents. 
We first consider an example. 

\begin{example}\label{exa-ch2-solutions-not-regular-non-constant} 
We reexamine the operator $L$ in \eqref{eq-ch2-def-operator-example}
but assume that $a, b, c$ are smooth functions in $\bar\Omega$ with $a>0$ in $\bar\Omega$. 
Let $s$ be a smooth function in $\bar\Omega$ 
such that \eqref{eq-ch2-root-condition} holds on $\partial\Omega$. 
Take a function $\psi\in C(\bar\Omega)\cap C^\infty(\Omega)$. 
A straightforward calculation implies
\begin{align*} L(\psi\rho^s)=f,\end{align*}
where 
\begin{align}\label{eq-ch2-example-expression-f-non-constant}\begin{split}
f&=\rho^s\big\{\big[as(s-1)+bs+c\big]\psi+(2as+b)\psi\nabla\rho\cdot\nabla s\,\rho\log\rho\\
&\qquad+\big[as(s-1)+bs\big](|\nabla \rho|^2-1)\psi\\
&\qquad+\big[as\psi\Delta\rho+(2as+b)\nabla\rho\cdot\nabla\psi+2a\psi\nabla\rho\cdot\nabla s\big]\rho\\
&\qquad 
+a\psi|\nabla s|^2\rho^2(\log\rho)^2
+a(\psi\Delta s+2\nabla\psi\cdot\nabla s)\rho^2\log\rho
+a\Delta\psi\rho^{2}\big\}.
\end{split}\end{align}
Comparing with $f$ in \eqref{eq-ch2-example-expression-f}, 
we note that there are additional terms in the right-hand side, involving 
$\rho\log\rho$, $\rho^2(\log\rho)^2$, and $\rho^2\log\rho$. These terms are not smooth in $\bar\Omega$. 
Similarly, the first term in the right-hand side of \eqref{eq-ch2-example-expression-f-non-constant} 
vanishes on $\partial\Omega$. 

To examine the regularity of $f$, we consider a special case that 
$s$ is {\it between two consecutive integers} on $\partial\Omega$.  
Write 
$s=k+\gamma$, for some integer $k\ge 0$ 
and some function $\gamma$ with $0<\gamma<1$ on $\partial\Omega$. 
Let $\gamma_0$ be the minimum of $\gamma$ on $\partial\Omega$. 
For $u=\psi\rho^s$ and $f$ as in \eqref{eq-ch2-example-expression-f-non-constant} with $\psi\in C^\infty(\bar\Omega)$, 
we have $f\in C^{k+1, \epsilon}(\bar\Omega)$ for any $\epsilon\in (0,\gamma_0)$, 
$f\notin C^{k+1, \gamma_0}(\bar\Omega)$, and  
$u\in C^{k,\gamma_0}(\bar\Omega)$, but $u\notin C^{k,\beta}(\bar\Omega)$, 
for any $\beta\in (\gamma_0,1)$. 
A slight loss of the H\"older index in the regularity of $f$ is caused by $\log \rho$. 

In order to improve the regularity 
of $f$ in \eqref{eq-ch2-example-expression-f-non-constant}, we need to modify $\psi$ 
to eliminate the logarithmic factors in $f$. It is natural to introduce 
logarithmic factors in $\psi$, which will create more logarithmic factors in $f$. 
For a function $\psi_{0,0}\in C^\infty(\bar\Omega)$ and a positive integer $m$, consider 
$$\psi=\sum_{i=0}^m\sum_{j=0}^i\psi_{i,j}\rho^i(\log \rho)^j.$$
We now assume in addition that, for $\mu=s+1, \cdots, s+m$, 
$$a\mu(\mu-1)+b\mu +c\neq 0\quad\text{on }\partial\Omega.$$
%$$a\mu(\mu-1)|\nabla \rho|^2+b\mu |\nabla\rho|^2+c\neq 0\quad\text{in }\bar\Omega.$$
Substituting $\psi$ in \eqref{eq-ch2-example-expression-f-non-constant} 
and choosing $\psi_{1,1}, \psi_{1,0}, \cdots, \psi_{m,m}, \cdots, \psi_{m,0}$ successively, 
we have $f\in C^{k+m+1, \epsilon}(\bar\Omega)$ for any $\epsilon\in (0,\gamma_0)$, 
$f\notin C^{k+m+1, \gamma_0}(\bar\Omega)$, and  
$u\in C^{k,\gamma_0}(\bar\Omega)$, but $u\notin C^{k,\beta}(\bar\Omega)$, 
for any $\beta\in (\gamma_0,1)$. 
\end{example} 

We now compare Example \ref{exa-ch2-solutions-not-regular-non-constant} with 
Example \ref{exa-ch2-solutions-not-regular}.  
In Case 1 in Example \ref{exa-ch2-solutions-not-regular}, 
$s$ satisfying \eqref{eq-ch2-root-condition} is assumed to be a non-integer constant. 
The only singular factor is given by $\rho^\gamma$ for some $\gamma\in (0,1)$. 
In Example \ref{exa-ch2-solutions-not-regular-non-constant}, 
$s$ is a function on $\partial\Omega$. 
In addition to $\rho^\gamma$, there are also logarithmic factors given by integer powers of $\log\rho$.  
The solution $u$ constructed in Example \ref{exa-ch2-solutions-not-regular-non-constant} is given by 
$$u=\sum_{i=0}^m\sum_{j=0}^i\psi_{i,j}\rho^{s+i}(\log \rho)^j.$$
The aim of the paper is to demonstrate that $\rho^{s+i}(\log \rho)^j$ 
provide sole obstructions to the higher regularity of solutions.

We next present our regularity results for varying characteristic exponents. 
Let $\overline{m}$ be the positive characteristic exponent. 
We always assume that $\overline{m}$ is {\it between two consecutive integers on $\partial\Omega$} and 
write 
$\overline{m}=[\overline{m}]+\gamma.$
Then, $[\overline{m}]$ is an integer and $0<\gamma<1$ on $\partial\Omega$.
%We next consider the general case that the characteristic exponents $\underline{m}$ and $\overline{m}$
%in \eqref{eq-ch4-assumption-characteristic-roots-1} are {\it not} constants. 
Consider an arbitrary point $x_{0}\in \partial\Omega$. 
For a nonnegative integer $k$ and a constant $\alpha\in(0,1)\setminus\{\gamma(x_{0})\}$, we set
\begin{align*}
k_{\ast}&=k,\  \epsilon_{\ast}(x_{0})=\alpha-\gamma(x_{0})\quad\text{if }\gamma(x_{0})<\alpha,\\
k_{\ast}&=k-1,\  \epsilon_{\ast}(x_{0})=\alpha+1-\gamma(x_{0})\quad\text{if }\alpha<\gamma(x_{0}).
\end{align*}
%The next result is the generalization of Theorem \ref{Thm-ch4-Linear-MainThm-curved-combined-constant}. 

\begin{theorem}\label{Thm-ch4-Linear-MainThm-curved-combined}
Let $\Omega$ be a bounded domain in $\mathbb R^n$, with  
a $C^1$-boundary $\partial\Omega$ and a defining function $\rho\in C^1(\bar\Omega)$, 
and $u\in C(\bar \Omega)\cap C^{2}(\Omega)$ be 
a solution of the Dirichlet problem \eqref{eq-ch2-basic-equation} and \eqref{eq-ch2-Dirichlet}.  
For some $x_0\in\partial\Omega$, 
suppose that the characteristic polynomial $P$ has a negative root  $\underline{m}$
and a positive root $\overline{m}$ on $\partial\Omega\cap B_R(x_0)$, 
with $\overline{m}=[\overline{m}]+\gamma$
and $0<\gamma<1$ on $\partial\Omega\cap B_R(x_0)$, 
and that $k$ is an integer and $\alpha\in(0,1)$ is a constant such that 
either $k\geq [\overline{m}]\ge 1$, $\gamma<\alpha$ on $\partial\Omega\cap B_R(x_0)$  
or $k\geq [\overline{m}]+1$, $\alpha<\gamma$ on $\partial\Omega\cap B_R(x_0)$. 
%For some nonnegative integer $k$
%and some $\alpha\in (0,1)$, %with $k+\alpha<\overline{m}+1$, 
Assume $\partial\Omega\cap B_R(x_0)\in C^{k+2,\alpha}$, 
$\rho\in C^{k+2,\alpha}(\bar\Omega\cap B_R(x_0))$, 
with $\rho\nabla^{k+3}\rho\in C^{\alpha}(\bar\Omega\cap B_R(x_0))$ 
and $\rho\nabla^{k+3}\rho=0$ on $\partial\Omega \cap B_R(x_0)$, 
$a_{ij}, b_i, c\in C^{k+1, \alpha}(\bar \Omega\cap B_R(x_0))$ satisfying \eqref{eq-ch2-ellipticity}, and 
$f\in C^{k,\alpha}(\bar \Omega\cap B_R(x_0))$.
Then,
\begin{equation}\label{eq-ch4-decomposition-curved-int-non}
u=v+\sum^{k_{\ast}-[\overline{m}]}_{j=0}w_j\rho^{\gamma}(\log \rho)^{j} \quad\text{in } \Omega\cap B_R(x_0),
\end{equation}
for some functions $v\in C^{k, \alpha}(\bar\Omega\cap B_R(x_0))$ 
and $w_0,\cdots, w_{k_{\ast}-[\overline{m}]}\in C^{k_{\ast}, \epsilon}(\bar\Omega\cap B_R(x_0))$
for any constant $\epsilon$ with $0<\epsilon<\epsilon_*$ on $\partial \Omega\cap B_R(x_0)$, 
with $\partial_\nu^i w_j=0$ on $\partial \Omega\cap B_R(x_0)$, 
for $j=0,\cdots,k_{\ast}-[\overline{m}]$ and $i=0,\cdots,[\overline{m}]+j-1$. Moreover,  
\begin{align*}
&|v|_{C^{k, \alpha}(\bar\Omega\cap B_{R/2}(x_0))}
+\sum^{k_{\ast}-[\overline{m}]}_{j=0}|w_j|_{C^{k_{\ast}, \epsilon}(\bar\Omega\cap B_{R/2}(x_0))}\\
&\qquad\leq C\big\{|u|_{L^{\infty}(\Omega\cap B_R(x_0))}
+|f|_{C^{k, \alpha}(\bar\Omega\cap B_R(x_0))}\big\},
\end{align*}
where $C$ is a positive constant depending only on $n$, $\lambda$, $k$, $\alpha$, $\epsilon$, 
$\underline{m}$, $\overline{m}$, $\Omega\cap B_R(x_0)$, 
the $C^{k+2, \alpha}$-norm of $\rho$ in $\bar\Omega\cap B_R(x_0)$, 
the $C^{\alpha}$-norm of $\rho\nabla^{k+3}\rho$ in $\bar\Omega\cap B_R(x_0)$,
and the $C^{k+1, \alpha}$-norms of $a_{ij}, b_i, c$ in $\bar\Omega\cap B_R(x_0)$.
If, in addition, $\partial_\nu^{[\overline{m}]} w_0=0$ on $\partial \Omega\cap B_R(x_0)$, 
then $u \in C^{k,\alpha}(\bar\Omega\cap B_R(x_0))$. 
%for any constant $\epsilon$ with $0<\epsilon<\epsilon_*$ on $\partial \Omega\cap B_R(x_0)$.
\end{theorem}

%The next result is the generalization of 
%Theorem \ref{Thm-ch4-Linear-MainThm-curved-combined-smooth-constant}. 

\begin{theorem}\label{Thm-ch4-Linear-MainThm-curved-combined-smooth}
Let $\Omega$ be a bounded domain in $\mathbb R^n$, with  
a $C^1$-boundary $\partial\Omega$ and a defining function $\rho\in C^1(\bar\Omega)$, 
and $u\in C(\bar \Omega)\cap C^{2}(\Omega)$ be 
a solution of the Dirichlet problem \eqref{eq-ch2-basic-equation} and \eqref{eq-ch2-Dirichlet}.  
For some $x_0\in\partial\Omega$, 
suppose that the characteristic polynomial $P$ has a negative root  $\underline{m}$
and a positive root $\overline{m}$ on $\partial\Omega\cap B_R(x_0)$, 
with $\overline{m}=[\overline{m}]+\gamma$
and $0<\gamma<1$ on $\partial\Omega\cap B_R(x_0)$. 
Assume $\partial\Omega\cap B_R(x_0)\in C^{\infty}$, 
$\rho\in C^{\infty}(\bar\Omega\cap B_R(x_0))$, 
and $a_{ij}, b_i, c, f\in C^{\infty}(\bar \Omega\cap B_R(x_0))$ satisfying \eqref{eq-ch2-ellipticity}.
Then, for any $r\in (0,R)$, 
$$
u=v+\sum^{\infty}_{j=0}w_j\rho^{\overline{m}+j}(\log \rho)^{j} 
\quad\text{absolutely and uniformly in } \bar\Omega\cap B_r(x_0),
$$
for some functions $v,w_0,w_1,\cdots \in C^{\infty}(\bar\Omega\cap B_R(x_0))$. 
%, with $\partial_t^i w^{(j)}=0$ on $\partial\Omega\cap B_R(x_0)$, 
%for $j=0,1,\cdots$ and $i=0,\cdots, [\overline{m}]+j-1$. 
Moreover, for any integer $k\geq [\overline{m}]$,
\begin{equation}\label{eq-ch4-Decomposition-infinite-higher-non-int-intro}
D^{k}\Big[u-\sum^{k-[\overline{m}]}_{j=0}w_j\rho^{\overline{m}+j}(\log \rho)^{j}\Big]
=D^{k}v+\sum^{\infty}_{j=k-[\overline{m}]+1}D^{k}[w_j\rho^{\overline{m}+j}(\log \rho)^{j}],
\end{equation}
where the series in the right-hand side 
converges absolutely and uniformly in $\bar\Omega\cap B_r(x_0)$, for any $r\in(0,R)$.
If, in addition, $w_0=0$ on $\partial \Omega\cap B_R(x_0)$, 
then $u \in C^{\infty}(\bar\Omega\cap B_R(x_0))$.
\end{theorem}

Theorem \ref{Thm-ch4-Linear-MainThm-curved-combined}
and Theorem \ref{Thm-ch4-Linear-MainThm-curved-combined-smooth}
are the main results in this paper; 
while Theorem \ref{Thm-ch4-Linear-MainThm-curved-combined-constant}
and Theorem \ref{Thm-ch4-Linear-MainThm-curved-combined-smooth-constant} 
are their special cases, respectively. 
Remarks after Theorem \ref{Thm-ch4-Linear-MainThm-curved-combined-constant}
and Theorem \ref{Thm-ch4-Linear-MainThm-curved-combined-smooth-constant} 
also hold for Theorem \ref{Thm-ch4-Linear-MainThm-curved-combined}
and Theorem \ref{Thm-ch4-Linear-MainThm-curved-combined-smooth}.
In the rest of this paper, we will prove 
the general Theorem \ref{Thm-ch4-Linear-MainThm-curved-combined}
and Theorem \ref{Thm-ch4-Linear-MainThm-curved-combined-smooth}
and indicate how to modify the proof to get 
Theorem \ref{Thm-ch4-Linear-MainThm-curved-combined-constant}
and Theorem \ref{Thm-ch4-Linear-MainThm-curved-combined-smooth-constant} 
if the characteristic exponents 
$\underline{m}$ and $\overline{m}$
are constants. 
The proof of Theorem \ref{Thm-ch4-Linear-MainThm-curved-integer}
and Theorem \ref{Thm-ch4-Linear-MainThm-curved-combined-smooth-integer}
is similar.

\section{A Portion of the Flat Boundary}\label{sec-flat-boundary}

We start to prove 
Theorem \ref{Thm-ch4-Linear-MainThm-curved-combined}
and Theorem \ref{Thm-ch4-Linear-MainThm-curved-combined-smooth}. 
By an appropriate change of coordinates, we assume 
that the domain $\Omega$ has a portion of flat boundary and 
that the defining function is given by the distance function to the boundary. 
%In this setting, the uniformly degenerate elliptic equations have a simple form. 

%Throughout this chapter, summations over Latin letters are from 1 to $n$ 
%and those over Greek letters are from 1 to $n-1$. 

Denote by $x=(x',t)$ points in $\mathbb{R}^{n}$ and also write $x_{n}=t$. 
Set, for any $x_0'\in\mathbb R^{n-1}$ and any $r>0$, 
\begin{align*}
B'_{r}(x_0')=\{x'\in\mathbb{R}^{n-1}:|x'-x_0'|<r\},
\end{align*}
and 
\begin{align*}
&G_{r}(x_0')=\{(x',t)\in\mathbb{R}^{n}:|x'-x_0'|<r,\ 0<t<r\},\\
&\Sigma_{r}(x_0')=\{(x',0)\in\mathbb{R}^{n}:|x'-x_0'|<r\}. 
\end{align*}
We also set $B'_{r}=B'_{r}(0)$, $G_{r}=G_{r}(0)$, and $\Sigma_{r}=\Sigma_{r}(0)$.
In the remaining part of the paper, summations over Latin letters are from 1 to $n$ 
and those over Greek letters are from 1 to $n-1$.
However, an unrepeated $\alpha, \beta\in (0,1)$ are reserved for 
the H\"older indices in the study of regularity for elliptic equations.

Consider the operator
\begin{align}\label{eq-ch3-LinearOperator-t}
L=t^2a_{ij}\partial_{ij}+tb_i\partial_i+c\quad\text{in }G_1,\end{align}
with $a_{ij}, b_i, c\in C(\bar G_1)$ 
satisfying $a_{ij}=a_{ji}$ and, 
for any $x\in G_1$ and $\xi\in\mathbb R^n$, 
\begin{align}\label{eq-ch3-ellipticity-t}
\lambda|\xi|^2\le a_{ij}(x)\xi_i\xi_j\le\Lambda|\xi|^2,\end{align}
for some positive constants $\lambda$ and $\Lambda$. 
Similarly as in \eqref{eq-ch2-definition-P-mu}, define $P(\mu)$ by 
$$P(\mu)=\mu(\mu-1)a_{nn}+\mu b_n+c\quad\text{in }G_1.$$
We point out that $P(\mu)$ here is defined in the entire domain $G_1$, instead of only on the boundary. 
Note that $P(0)=c$. 

We first assume
\begin{align}\label{eq-bound-c-t}c\le -c_0\quad\text{in }G_1,\end{align}
where $c_0$ is a positive constant. 
We consider 
\begin{align}\label{eq-ch3-Equ} Lu= f\quad\text{in }G_1,
\end{align}
and 
\begin{align}\label{eq-ch3-Dirichlet}
u(\cdot, 0)=\frac{f}{c}(\cdot, 0)\quad\text{on }B'_1.\end{align}
%where we always assume 
%\begin{align}\label{eq-ch3-boundary-value}u_0=\frac{f}{c}(\cdot, 0)
%\quad\text{on }B'_1.\end{align}
We point out that the assumption $c<0$ on $\Sigma_1$ is essential in order to relate 
the boundary value of $u$ to values of the nonhomogeneous term $f$. 

We start with the following result concerning the optimal regularity of solutions near the boundary.
Refer to \cite{HanXie2024} for details. 

\begin{theorem}\label{thrm-ch3-Linear-NormalEstimate-general}
For some integers $\ell\ge k\ge 0$ and constant 
$\alpha\in (0,1)$, assume 
$\partial^\nu_tD^\tau_{x'}a_{ij}$, $\partial^\nu_tD^\tau_{x'}b_i$, $\partial^\nu_tD^\tau_{x'}c\in C^{\alpha}(\bar G_1)$ 
for any $\nu\le k$ and $\tau+\nu\le \ell$,  
with \eqref{eq-ch3-ellipticity-t}, and $c\le -c_0$ 
and $Q({k+\alpha})\le -c_{k+\alpha}$  in $G_1$, for some positive constants 
$c_0$ and $c_{k+\alpha}$.  
For some $f$ with $\partial^\nu_tD^\tau_{x'}f\in C^{\alpha}(\bar G_1)$ 
for any $\nu\le k$ and $\tau+\nu\le \ell$, let $u\in C(\bar G_1)\cap C^2(G_1)$ be 
a solution of \eqref{eq-ch3-Equ} and \eqref{eq-ch3-Dirichlet}.  
Then, for any nonnegative integers $\nu$ and $\tau$ with $\nu\le k$ and 
$\nu+\tau\le \ell$, and any $r\in (0,1)$, 
\begin{equation*}%\label{eq-regularity-normal-k}
\partial_t^\nu D^\tau_{x'}u, tD\partial_t^\nu D^\tau_{x'}u, t^2D^2\partial_t^\nu D^\tau_{x'}u
\in C^\alpha(\bar G_r),\end{equation*}
with $tD\partial_t^\nu D^\tau_{x'}u=0$ and $t^2D^2\partial_t^\nu D^\tau_{x'}u=0$ on $\Sigma_1$,
and 
\begin{align*}
&|\partial_t^\nu D^\tau_{x'}u|_{C^\alpha(\bar G_{1/2})}+|tD\partial_t^\nu D^\tau_{x'}u|_{C^\alpha(\bar G_{1/2})}
+|t^2D^2\partial_t^\nu D^\tau_{x'}u|_{C^\alpha(\bar G_{1/2})}\\
&\qquad\le C\Big\{|u|_{L^\infty(G_1)}+\sum_{j\le k,i+j\le \ell}|\partial^j_tD^i_{x'}f|_{C^{\alpha}(\bar G_1)}\Big\},
\end{align*}
where 
$C$ is a positive constant depending only on $n$, $\ell$, $\lambda$, 
$\alpha$, $c_0$, $c_{k+\alpha}$, and
the $C^{\alpha}$-norms of 
$\partial^\nu_tD^\tau_{x'}a_{ij}, \partial^\nu_tD^\tau_{x'}b_i, \partial^\nu_tD^\tau_{x'}c$  in $\bar G_1$
for any $\nu\le k$ and $\tau+\nu\le \ell$. 
In particular, 
$D_{x'}^\tau u\in C^{k,\alpha}(\bar G_r)$, for any $\tau\le \ell-k$ and $r\in (0,1)$. 
\end{theorem}

We now point out one important feature of solutions to the uniformly degenerate elliptic equations. 
In Theorem \ref{thrm-ch3-Linear-NormalEstimate-general}, it suffices to assume $u\in L^\infty(G_1)\cap C^2(G_1)$
and \eqref{eq-ch3-Equ} holds. 
Then, it follows that $u$ is continuous up to $\Sigma_1$ and satisfies \eqref{eq-ch3-Dirichlet}.

Solutions of uniformly degenerate elliptic equations usually have different H\"{o}lder indices 
along the tangential directions and the normal direction.
We will introduce the following notations and terminologies for convenience. We fix an $r>0$.

\begin{definition}\label{def-Holder-double-index} 
Let $\alpha, \beta\in (0,1)$ be constants. 

(i)  Denote by $C^{\beta}_{x'}(\bar G_r)$ the collection of functions $f$ in $\bar G_r$ such that, 
for any $(x_1', t)$, $(x_2',t)\in \bar G_r$, 
$$|f(x_1', t)-f(x_2',t)|\le C|x_1'-x_2'|^\beta,$$ 
for some nonnegative constant $C$. 
%Set 
%\begin{align*}[f]_{C^{\beta}_{y'}(\bar G_r)}
%=\sup_{t\in [0,r]}\sup_{y_1', y_2'\in \bar B_r}\frac{|f(y_1', t)-f(y_2',t)|}{|y_1'-y_2'|^\beta}.
%\end{align*}

(ii) Denote by $C^{\alpha}_{t}(\bar G_r)$ the collection of functions $f$ in $\bar G_r$ such that, 
for any $(x', t_1)$, $(x', t_2)\in \bar G_r$, 
$$|f(x', t_1)-f(x', t_2)|\le C|t_1-t_2|^\alpha,$$
for some nonnegative constant $C$. 
%Set 
%\begin{align*}[f]_{C^{\alpha}_{t}(\bar G_r)}
%=\sup_{y'\in \bar B_r}\sup_{t_1, t_2\in [0,r]}\frac{|f(y', t_1)-f(y', t_2)|}{|t_1-t_2|^\alpha}.
%\end{align*}

(iii) Set $C^{\beta, \alpha}_{x',t}(\bar G_r)=C^{\beta}_{x'}(\bar G_r)\cap C^{\alpha}_{t}(\bar G_r)$. 
\end{definition} 

In other words, $f$ is in $C^{\beta}_{x'}(\bar G_r)$
if $f(x', t)$ is a $C^\beta$-function of $x'\in \bar B_r'$, uniformly in $t\in [0,r]$, 
and $f$ is in $C^{\alpha}_{t}(\bar G_r)$ if $f(x', t)$ is  
a $C^\alpha$-function of $t\in [0,r]$, uniformly in $x'\in \bar B_r'$. 
If $\alpha=\beta$, $C^{\beta, \alpha}_{x',t}(\bar G_r)$ is the usual H\"older space $C^{\alpha}(\bar G_r)$.
For convenience, we introduce 
\begin{align*}[f]_{C^{\beta}_{x'}(\bar G_r)}
&=\sup_{t\in [0,r]}[f(\cdot, t)]_{C^\beta(\bar B_r')},\\
[f]_{C^{\alpha}_{t}(\bar G_r)}
&=\sup_{x'\in \bar B_r'}[f(x', \cdot)]_{C^\alpha([0,r])},
\end{align*}
and 
\begin{align*}[f]_{C^{\beta, \alpha}_{x',t}(\bar G_r)}
=[f]_{C^{\beta}_{x'}(\bar G_r)}+[f]_{C^{\alpha}_{t}(\bar G_r)}.
\end{align*}
We can define corresponding norms similarly. 
%{\color{red} (This sentence is enough. There is no need to define these norms explicitly.)}

We have the following more general result. 

\begin{theorem}\label{thrm-ch3-Linear-NormalEstimate-general-beta}
For some integers $\ell\ge k\ge 0$ and constants 
$\alpha, \beta\in (0,1)$ with $\beta\le \alpha$, assume 
$\partial^\nu_tD^\tau_{x'}a_{ij}, \partial^\nu_tD^\tau_{x'}b_i, \partial^\nu_tD^\tau_{x'}c\in C^{\alpha}(\bar G_1)$ 
for any $\nu\le k$ and $\tau+\nu\le \ell$,  
with \eqref{eq-ch3-ellipticity-t}, and $c\le -c_0$ 
and $Q({k+\beta})\le -c_{k+\beta}$  in $G_1$, for some positive constants 
$c_0$ and $c_{k+\beta}$.  
For some $f$ with $\partial^\nu_tD^\tau_{x'}f\in C^{\alpha}(\bar G_1)$ 
for any $\nu\le k$ and $\tau+\nu\le \ell$, let $u\in C(\bar G_1)\cap C^2(G_1)$ be 
a solution of \eqref{eq-ch3-Equ} and \eqref{eq-ch3-Dirichlet}.  
Then, for any nonnegative integers $\nu$ and $\tau$ with $\nu\le k$ and 
$\nu+\tau\le \ell$, and any $r\in (0,1)$, 
\begin{equation*}%\label{eq-regularity-normal-k}
\partial_t^\nu D^\tau_{x'}u, tD\partial_t^\nu D^\tau_{x'}u, t^2D^2\partial_t^\nu D^\tau_{x'}u
\in C^{\alpha,\beta}_{x',t}(\bar G_r),\end{equation*}
with $tD\partial_t^\nu D^\tau_{x'}u=0$ and $t^2D^2\partial_t^\nu D^\tau_{x'}u=0$ on $\Sigma_1$,
and 
\begin{align*}
&|\partial_t^\nu D^\tau_{x'}u|_{C^{\alpha,\beta}_{x',t}(\bar G_{1/2})}
+|tD\partial_t^\nu D^\tau_{x'}u|_{C^{\alpha,\beta}_{x',t}(\bar G_{1/2})}
+|t^2D^2\partial_t^\nu D^\tau_{x'}u|_{C^{\alpha,\beta}_{x',t}(\bar G_{1/2})}\\
&\qquad\le C\Big\{|u|_{L^\infty(G_1)}+\sum_{j\le k,i+j\le \ell}|\partial^j_tD^i_{x'}f|_{C^{\alpha}(\bar G_1)}\Big\},
\end{align*}
where 
$C$ is a positive constant depending only on $n$, $\ell$, $\lambda$, 
$\alpha$, $\beta$, $c_0$, $c_{k+\beta}$, and
the $C^{\alpha}$-norms of 
$\partial^\nu_tD^\tau_{x'}a_{ij}, \partial^\nu_tD^\tau_{x'}b_i, \partial^\nu_tD^\tau_{x'}c$  in $\bar G_1$
for any $\nu\le k$ and $\tau+\nu\le \ell$. 
%In particular, $D_{x'}^\tau u\in C^{k,\alpha}(\bar G_r)$, for any $\tau\le \ell-k$ and $r\in (0,1)$. 
\end{theorem}

\begin{proof} 
If $\beta=\alpha$, Theorem \ref{thrm-ch3-Linear-NormalEstimate-general-beta} is 
simply Theorem \ref{thrm-ch3-Linear-NormalEstimate-general}. We consider the case $0<\beta<\alpha<1$. 
For simplicity, we consider $\ell=k=0$ only. In the proof below, we will use Lemma 5.6(1) in \cite{Caffarelli-Cabre1995} repeatedly. 

By $a_{ij}, b_i, c, f\in C^{\beta}(\bar G_1)$ and Theorem \ref{thrm-ch3-Linear-NormalEstimate-general}
with $\beta$ in the role of $\alpha$ there, we have, for any $r\in (0,1)$,  
\begin{equation}\label{a} %\label{eq-ch3-regularity-beta}
u, tDu, t^2D^2u\in C^{\beta}(\bar G_r),\end{equation}
with $u=u_0$, $tDu=0$, and $t^2D^2u=0$ on $\Sigma_1$, and 
\begin{align}\label{b} %\label{eq-ch3-estimate-alpha-beta}
\begin{split}
&|u|_{C^{\beta}(\bar G_{r})}+|tDu|_{C^{\beta}(\bar G_{r})}
+|t^2D^2u|_{C^{\beta}(\bar G_{r})}\\
&\qquad\le C\big\{|u|_{L^\infty(G_1)}+|f|_{C^{\beta}(\bar G_1)}\big\}.\end{split}
\end{align}
We now improve the H\"older regularity in $x'$. We claim, for any $r\in (0,1)$, 
\begin{equation}\label{eq-ch3-regularity-alpha-beta-x}
u, tDu, t^2D^2u\in C^{\alpha}_{x'}(\bar G_r),\end{equation}
and 
\begin{align}\label{eq-ch3-estimate-alpha-alpha-beta-x}\begin{split}
&[u]_{C^{\alpha}_{x'}(\bar G_{r})}+[tDu]_{C^{\alpha}_{x'}(\bar G_{r})}
+[t^2D^2u]_{C^{\alpha}_{x'}(\bar G_{r})}\\
&\qquad\le C\big\{|u|_{L^\infty(G_1)}+|f|_{C^{\alpha}(\bar G_1)}\big\}.\end{split}
\end{align}

%To prove \eqref{eq-ch3-regularity-alpha-beta-x} and \eqref{eq-ch3-estimate-alpha-alpha-beta-x}, 
We take a positive integer $k$ such that $k\beta<\alpha\le (k+1)\beta$. 
Set $\beta_0=\alpha-k\beta$. Then, $\alpha=k\beta+\beta_0$ and $\beta_0\in (0, \beta]$. 
We fix an arbitrary unit vector $e\in\mathbb{R}^{n-1}\times\{0\}$ and an arbitrary constant $r_1\in (3/4,1)$. 
Set, for any $x\in \bar G_{r_1}$ and any $s\in\mathbb R$ with $0<|s|\le (1-r_1)/2$, 
$$
u_{\beta,s}(x)=\frac{1}{|s|^{\beta}}[u(x+se)-u(x)],
$$
and similarly $a_{ij,\beta,s}$, $b_{i,\beta,s}$, $c_{\beta,s}$, and $f_{\beta,s}$. 
Evaluate the equation $Lu=f$ at $x+se$ and $x$, take the difference, and then divide by $|s|^{\beta}$. Hence,
\begin{equation}\label{c}
Lu_{\beta,s}=h_1,
\end{equation}
where 
$$
h_1=f_{\beta,s}-t^2a_{ij,\beta,s}\partial_{ij}u(\cdot+se)-t b_{i,\beta,s}\partial_{i}u(\cdot+se)-c_{\beta,s}u(\cdot+se).
$$
By \eqref{a} and \eqref{b}, we have $u_{\beta,s}\in C(\bar G_{r_1})\cap C^2(G_{r_1})$ and 
\begin{align*}
|u_{\beta,s}|_{L^\infty(G_{r_1})}\le C\big\{|u|_{L^\infty(G_1)}+|f|_{C^{\beta}(\bar G_1)}\big\},\end{align*}
where $C$ is a positive constant independent of $e\in\mathbb{R}^{n-1}\times\{0\}$ and
$s\in\mathbb R$ with $0<|s|\le (1-r_1)/2$. 
Next, we study the regularity of $h_1$. For an illustration, we discuss $f_{\beta,s}$.
It is obvious that 
$$|f_{\beta,s}|_{L^\infty(G_{r_1})}\le [f]_{C^\beta(\bar G_1)}.$$
For any $x_{1},x_{2}\in \bar G_{r_1}$, we write 
$$f_{\beta,s}(x_1)-f_{\beta,s}(x_2)=\frac{1}{|s|^{\beta}}I,$$
where 
$$I=\big(f(x_{1}+se)-f(x_1)\big)-\big(f(x_{2}+se)-f(x_{2})\big).$$
First, we have 
$$
|I|\leq 2[f]_{C^{\alpha}(\bar G_{1})}|s|^{\alpha}.
$$
Next, by writing 
$$I=\big(f(x_{1}+se)-f(x_{2}+se)\big)-\big(f(x_1)-f(x_{2})\big),$$
we have 
$$
|I|\leq 2[f]_{C^{\alpha}(\bar G_{1})}|x_1-x_2|^{\alpha}.
$$
Hence, 
$$
|I|=|I|^{\frac\beta\alpha}|I|^{\frac{\alpha-\beta}\alpha}
\leq 2[f]_{C^{\alpha}(\bar G_{1})}|s|^{\beta}|x_1-x_2|^{\alpha-\beta}.
$$
This implies, for any $x_{1},x_{2}\in \bar G_{r_1}$, 
$$|f_{\beta,s}(x_1)-f_{\beta,s}(x_2)|\le 2[f]_{C^{\alpha}(\bar G_{1})}|x_1-x_2|^{\alpha-\beta},$$
and thus 
\begin{equation*}
[f_{\beta,s}]_{C^{\alpha-\beta}(\bar{G}_{r_1})}\leq  2[f]_{C^{\alpha}(\bar G_{1})}.
\end{equation*}
Therefore, 
\begin{equation}\label{d}
|f_{\beta,s}|_{C^{\alpha-\beta}(\bar{G}_{r_1})}\leq  3[f]_{C^{\alpha}(\bar G_{1})}.
\end{equation}
Similarly, we have
\begin{equation}\label{e}
|a_{ij,\beta,s}|_{C^{\alpha-\beta}(\bar{G}_{r_1})}+|b_{i,\beta,s}|_{C^{\alpha-\beta}(\bar{G}_{r_1})}
+|c_{\beta,s}|_{C^{\alpha-\beta}(\bar{G}_{r_1})}\leq C.
\end{equation}
By combining \eqref{b}, \eqref{d}, and \eqref{e}, 
we obtain  $h_1\in C^{\min\{\alpha-\beta, \beta\}}(\bar{G}_{r_1})$ and 
$$|h_1|_{C^{\min\{\alpha-\beta, \beta\}}(\bar{G}_{r_1})}\le 
C\big\{|u|_{L^\infty(G_1)}+|f|_{C^{\alpha}(\bar G_1)}\big\}.$$
We now consider two cases. 

{\it Case 1: $k=1$.} In this case, we have $\alpha=\beta+\beta_0$ with $\beta_0\in (0, \beta]$. 
Hence, $\min\{\alpha-\beta, \beta\}=\beta_0$. 
We also note that $Q(\beta_0)\le -\min\{c_0, c_\beta\}$ in $\bar G_1$.
By applying Theorem \ref{thrm-ch3-Linear-NormalEstimate-general} to \eqref{c} in $G_{r_1}$,
%{\color{red} (I added a paragraph after Theorem \ref{thrm-ch3-Linear-NormalEstimate-general}.
%It is really not necessary since $u_{\beta,s}$ is continuous up to $\Sigma_1$.)}
with $\beta_0$ in the role of $\alpha$ there, we get, for any $r\in (0,r_1)$, 
$$
u_{\beta,s},
t D u_{\beta,s},
t^{2} D^{2} u_{\beta,s} \in C^{\beta_0}(\bar{G}_{r}), 
$$
and 
\begin{align*}
&|u_{\beta,s}|_{C^{\beta_0}(\bar G_{r})}+|tDu_{\beta,s}|_{C^{\beta_0}(\bar G_{r})}
+|t^2D^2u_{\beta,s}|_{C^{\beta_0}(\bar G_{r})}\\
&\qquad\le C\big\{|u|_{L^\infty(G_1)}+|f|_{C^{\alpha}(\bar G_1)}\big\}.
\end{align*}
We focus only on the tangential directions. 
By Lemma 5.6(1) in \cite{Caffarelli-Cabre1995}, 
we obtain \eqref{eq-ch3-regularity-alpha-beta-x}
and \eqref{eq-ch3-estimate-alpha-alpha-beta-x} for any $r\in (0, r_1)$. 

{\it Case 2: $k\ge 2$.} Take arbitrary constants $r_1, \cdots, r_k$ such that $3/4<r_k<\cdots<r_1<1$. 
We first prove inductively, for $i=2, \cdots, k$ and any $r\in (0, r_{i-1})$, 
\begin{equation}\label{eq-ch3-regularity-alpha-i-beta-x}
u, tDu, t^2D^2u\in C^{i\beta}_{x'}(\bar G_r),\end{equation}
and 
\begin{align}\label{eq-ch3-estimate-alpha-alpha-i-beta-x}\begin{split}
&[u]_{C^{i\beta}_{x'}(\bar G_{r})}+[tDu]_{C^{i\beta}_{x'}(\bar G_{r})}
+[t^2D^2u]_{C^{i\beta}_{x'}(\bar G_{r})}\\
&\qquad\le C\big\{|u|_{L^\infty(G_1)}+|f|_{C^{\alpha}(\bar G_1)}\big\}.\end{split}
\end{align}
By $k\ge 2$, we have $\min\{\alpha-\beta, \beta\}=\beta$. 
By applying Theorem \ref{thrm-ch3-Linear-NormalEstimate-general} to \eqref{c} in $G_{r_1}$,
with $\beta$ in the role of $\alpha$ there, and proceeding similarly as in Case 1, we get 
\eqref{eq-ch3-regularity-alpha-i-beta-x}
and \eqref{eq-ch3-estimate-alpha-alpha-i-beta-x} for $i=2$ and any $r\in (0, r_1)$. 
Instead of $u_{\beta,s}$, we now consider, 
for any $x\in \bar G_{r_2}$ and any $s\in\mathbb R$ with $0<|s|\le (r_1-r_2)/2$, 
$$
u_{2\beta,s}(x)=\frac{1}{|s|^{2\beta}}[u(x+se)-u(x)].
$$
Then, $u_{2\beta,s}\in C(\bar G_{r_2})\cap C^2(G_{r_2})$ and 
\begin{align*}
|u_{2\beta,s}|_{L^\infty(G_{r_2})}\le C\big\{|u|_{L^\infty(G_1)}+|f|_{C^{\alpha}(\bar G_1)}\big\},\end{align*}
where $C$ is a positive constant independent of $e\in\mathbb{R}^{n-1}\times\{0\}$ and
$s\in\mathbb R$ with $0<|s|\le (r_1-r_2)/2$. 
Proceeding inductively, we obtain \eqref{eq-ch3-regularity-alpha-i-beta-x}
and \eqref{eq-ch3-estimate-alpha-alpha-i-beta-x} for $i=k$ and any $r\in (0, r_{k-1})$. 
Then, proceeding as in Case 1 for $u_{k\beta,s}$, we obtain 
\eqref{eq-ch3-regularity-alpha-beta-x}
and \eqref{eq-ch3-estimate-alpha-alpha-beta-x} for any $r\in (0, r_k)$.
\end{proof}

In this paper, we discuss obstructions to the higher regularity of solutions along the normal direction. 
We start with the equation $Lu=f$ in \eqref{eq-ch3-Equ}
by rewriting it as an ODE in $t$. 
First, a simple rearrangement of $Lu=f$ yields 
$$t^2u_{tt}+\frac{b_n}{a_{nn}}tu_t+\frac{c}{a_{nn}}u=
\frac{1}{a_{nn}}\big(f-t^2a_{\alpha\beta}\partial_{\alpha\beta}u-2t^2a_{\alpha n}\partial_{\alpha t}u
-t b_\alpha\partial_\alpha u\big).$$
Set 
$$p=\frac{b_n}{a_{nn}}(x',0),\quad q=\frac{c}{a_{nn}}(x',0).$$
Then, 
\begin{align}\label{eq-ch4-LinearODE}
t^2u_{tt}+ptu_t+q u=F\quad\text{in }G_1, 
\end{align}
where
\begin{align}\label{eq-ch4-LinearODE-F}\begin{split}
F&=\frac{1}{a_{nn}}\left(f-t^2a_{\alpha\beta}\partial_{\alpha\beta}u-2t^2a_{\alpha n}\partial_{\alpha t}u
-t b_\alpha\partial_\alpha u\right)\\
&\qquad -\frac1t\Big(\frac{b_n}{a_{nn}}-p\Big)t^2\partial_tu
-\frac1t\Big(\frac{c}{a_{nn}}-q\Big)tu.\end{split}\end{align}
We note that $F$ is a linear combination of 
\begin{equation}\label{eq-ch4-LinearDependence}
f,\, t^2\partial_{\alpha t}u, \, t^2\partial_{\alpha\beta}u, \, t^2\partial_tu,\, t\partial_\alpha u,\, tu.
\end{equation}
The expression of $F$ suggests that we need to assume a better regularity for 
$a_{nn}, b_n$, and $c$ than that for the rest of coefficients. 

We always assume that $t^{\underline{m}}$ and $t^{\overline{m}}$ are 
solutions of the linear homogeneous equation corresponding to \eqref{eq-ch4-LinearODE}, 
for some functions $\underline{m}$ and $\overline{m}$ on $B_1'$; namely, 
\be\label{eq-ch4-Assumption_m1}
p=1-(\underline{m}+\overline{m}), \quad q=\underline{m}\cdot\overline{m}.\ee
We always assume that $\underline{m}$ and $\overline{m}$ satisfy 
\be\label{eq-ch4-Assumption_m2}\underline{m}<0<\overline{m}\quad\text{on }B'_{1}.\ee
Throughout this paper, we assume that $\overline{m}$ avoids integer values; 
in other words, $\overline{m}$ is between two consecutive integers. 
Set 
\begin{align}\label{eq-definition-gamma}
\gamma=\overline{m}-[\overline{m}]\quad\text{on }B'_{1}.\end{align}
We emphasize that $\gamma$ is a function satisfying $0<\gamma<1$ on $B_1'$. 

Next, we solve the equation \eqref{eq-ch4-LinearODE}, regarding $F$ as a function of $x'$ and $t$. 
A standard calculation yields the following result. 

\begin{lemma}\label{lemma-ch4-SolutionODE} Let $u$ be a $C(\bar G_r)\cap C^2(G_r)$-solution of 
\eqref{eq-ch4-LinearODE} in $G_r$.
% satisfying 
%\be\label{eq-ODE-Assumption} t^{-\underline{m}}u\to 0\quad\text{as }t\to 0.\ee
Then,
\begin{align}\label{eq-ch4-LinearODE-Solution-l}\begin{split}
u(x', t)&=\left[u(x',r)r^{-\overline{m}}
+\frac{r^{\underline{m}-\overline{m}}}{\overline{m}-\underline{m}}
\int_0^r s^{-1-\underline{m}}F(x',s)ds\right] t^{\overline{m}}\\
&\qquad-\frac{t^{\underline{m}}}{\overline{m}-\underline{m}} 
\int_0^t s^{-1-\underline{m}}F(x',s) ds\\
&\qquad-\frac{t^{\overline{m}}}{\overline{m}-\underline{m}}
\int_t^r s^{-1-\overline{m}}F(x',s) ds.\end{split}
\end{align}
\end{lemma} 

The expression \eqref{eq-ch4-LinearODE-Solution-l} plays a fundamental role in our study of 
regularity. 
We note that the regularity of $u$ in $x'$ inherits from that of $u(\cdot, r)$ and $F$, 
as long as the integrals in \eqref{eq-ch4-LinearODE-Solution-l} make sense. 
The regularity in $t$ is more complicated. 
The factor $s^{-1-\overline{m}}$ in the last integral 
in \eqref{eq-ch4-LinearODE-Solution-l} is singular at $s=0$ since $\overline{m}$ is positive. 
This is the reason that we integrate from $t$ to $r$ instead of from 0 to $t$.

We also need the following extension lemma. Refer to Gilbarg and H\"ormander \cite{GH1980}. 

\begin{lemma}\label{lemma-ch4-extensions} Let $k$ be a nonnegative integer, $\alpha\in (0,1]$ 
be a constant, and $c_0, c_1, \cdots, c_k$ be given functions on $\bar B'_1$ with 
$c_i\in C^{k-i,\alpha}(\bar B'_1)$ for $0\le i\le k$. Then, there exists 
a function $w\in C^{k,\alpha}(\bar G_1)\cap C^{\infty}(G_1)$ such that, 
for any $i=0, 1, \cdots, k$, %for any $0\le i\le k$,  
\begin{equation}\label{eq-Taylor-degree}
\partial^i_tw(\cdot, 0)=c_{i}\quad\text{on }B_1',\end{equation}
and $t^lD^lD^{k}w\in C^{\alpha}(\bar G_1)$ and $t^lD^lD^{k}w=0$ on $\Sigma_1$, 
for any $l\ge 1$. Moreover, $w$ depends on $c_0, c_1, \cdots, c_k$ linearly and 
$$|w|_{C^{k,\alpha}(\bar G_1)}+|t^lD^lD^{k}w|_{C^{\alpha}(\bar G_1)}\le 
C\sum_{i=0}^k|c_i|_{C^{k-i,\alpha}(\bar B_1')},$$ 
where $C$ is a positive constant depending only on $n$, $k$, $\alpha,$ and $l$.
\end{lemma}

%Throughout this paper, summations over Latin letters are from 1 to $n$ 
%and those over Greek letters are from 1 to $n-1$. 
%However, an unrepeated $\alpha, \beta\in (0,1)$ are reserved for 
%the H\"older indices in the study of regularity for elliptic equations. 

\section{Larger H\"older Indices}\label{sec-Large-Holder-Indices}

In this section, we study the case that the H\"older index is larger than the decimal part of 
the positive characteristic exponent. 
Specifically, let $\gamma$ be the function on $B_1'$ given by \eqref{eq-definition-gamma}. 
Throughout this section, we assume 
\begin{align}\label{eq-large-Holder-index}0<\gamma<\alpha<1\quad\text{on }B_1'.\end{align} 
%We emphasize that $\gamma$ is a function on $B_1'$, instead of a constant. 

Before we go to details, we first describe results in this section. 
In Lemma \ref{lemma-ch4-Linear-MainThm-non-integer-large} and 
Theorem \ref{thrm-ch4-Linear-MainThm-non-integer-large}, we study expansions of solutions in terms of $t$. 
In Lemma \ref{lemma-ch4-Linear-MainThm-non-integer-large}, we prove an expansion up to order $[\overline{m}]$. 
This is the first time that the singular factor $t^\gamma$ appears. 
In Theorem \ref{thrm-ch4-Linear-MainThm-non-integer-large}, 
we generalize such an expansion to an arbitrary order $k> [\overline{m}]$. 
In these results, it is essential to study the regularity of remainders. 
In Theorem \ref{Thm-ch4-Linear-MainThm-decomposition-non-int-large} and 
Theorem \ref{thrm-ch4-Decomposition-infinite-non-int},
we establish a concise decomposition of solutions into a regular part and a singular part, 
%the latter of which is a product of a regular factor and a specific singular factor $t^\gamma$, 
under the assumption of the finite differentiability and infinite differentiability, respectively. 
Both results are based on Theorem \ref{thrm-ch4-Linear-MainThm-non-integer-large} and, in particular, 
the regularity of the remainders. All results in this section are optimal.

We will use frequently the functions with different H\"older indices in the 
tangent directions and the normal direction, as introduced in Definition \ref{def-Holder-double-index}.
We will use results established in Section \ref{sec-Appen-CalculusL} 
concerning the regularity of certain singular integrals.
We discuss the regularity of solutions in the $C^{k,\alpha}$ space inductively by 
increasing $k$, starting with $k=[\overline{m}]$.

\begin{lemma}\label{lemma-ch4-Linear-MainThm-non-integer-large}
Suppose that $\underline{m}$, $\overline{m}$, and $\gamma$ are functions on $B'_{1}$ satisfying 
\eqref{eq-ch4-Assumption_m1}, \eqref{eq-ch4-Assumption_m2}, 
and \eqref{eq-definition-gamma}, 
and that $\ell$ is an integer and $\alpha\in(0,1)$ is a constant such that $\ell\geq [\overline{m}] \ge 1$ and 
\eqref{eq-large-Holder-index} holds. 
Assume $a_{n n}, b_{n}, c \in C^{\ell+1, \alpha}(\bar{G}_{1})$ and 
$a_{i j}, b_{i}, f \in C^{\ell, \alpha}(\bar{G}_{1})$ for $i \neq n$, 
and let $u \in C(\bar{G}_{1}) \cap C^{2}(G_{1})$ be a solution 
of \eqref{eq-ch3-Equ} and \eqref{eq-ch3-Dirichlet}. %, with \eqref{eq-ch3-boundary-value}. 
Then,
\begin{equation*}%\label{eq-ch4-LinearExpansion-non-int-k-large}
u=\sum_{i=0}^{[\overline{m}]} c_{i} t^{i}+c_{[\overline{m}], 0} t^{[\overline{m}]+\gamma}
+R_{[\overline{m}]} \quad \text {in } G_{1},
\end{equation*}
where $c_{0},c_{1},\cdots,c_{[\overline{m}]}$, and $c_{[\overline{m}],0}$ are functions on $B'_{1}$, 
and $R_{[\overline{m}]}$ is a function in $G_{1}$ satisfying, 
for any $x'_{0}\in B'_{1}$, 
any $r\in(0,1-|x'_{0}|)$, 
and any $\epsilon$ with $0<\epsilon 
<\alpha-\gamma$ on $\bar B_{r}'(x_0')$, 
\begin{align}\label{eq-ch4-LinearRegularity-noon-int-m1-large}
\begin{split}
&c_{i} \in C^{\ell-i, \alpha}(\bar B_{r}^{\prime}(x'_{0}))\quad\text{for }i=0,\cdots,[\overline{m}],\\
&c_{[\overline{m}], 0} \in C^{\ell-i, \epsilon}(\bar B_{r}^{\prime}(x'_{0})),
\end{split}    
\end{align}
and, for any nonnegative integers $\nu \leq [\overline{m}]$ and $\tau \leq \ell-[\overline{m}]$, 
%and any $r \in(0,r_{0})$,
\begin{align*}%\label{eq-ch4-LinearRegularity-noon-int-m1-large}\begin{split}
\partial_{t}^{\nu} D_{x^{\prime}}^{\tau} R_{[\overline{m}]}, 
t \partial_{t} \partial_{t}^{\nu} D_{x^{\prime}}^{\tau} R_{[\overline{m}]}, 
t^{2} \partial_{t}^{2} \partial_{t}^{\nu} D_{x^{\prime}}^{\tau} R_{[\overline{m}]} 
&\in C^{\epsilon,\alpha}_{x',t}(\bar{G}_{r}(x_{0}')),\\ 
\partial_{t}^{\nu} D_{x^{\prime}}^{\tau} (t^{-\gamma}R_{[\overline{m}]}), 
t \partial_{t} \partial_{t}^{\nu} D_{x^{\prime}}^{\tau} (t^{-\gamma}R_{[\overline{m}]}), 
t^{2} \partial_{t}^{2} \partial_{t}^{\nu} D_{x^{\prime}}^{\tau}(t^{-\gamma} R_{[\overline{m}]}) 
&\in C^{\epsilon}(\bar{G}_{r}(x_{0}')),  
%\end{split}
\end{align*}
and
\begin{align*}%\label{eq-ch4-LinearRegularity-non-int-m2-large}\begin{split}
&\big|\partial_{t}^{\nu} D_{x^{\prime}}^{\tau} R_{[\overline{m}]}\big|
+\big|t \partial_{t} \partial_{t}^{\nu} D_{x^{\prime}}^{\tau} R_{[\overline{m}]}\big|
+\big|t^{2} \partial_{t}^{2} \partial_{t}^{\nu} D_{x^{\prime}}^{\tau} R_{[\overline{m}]}\big| \\
&\qquad 
\leq C t^{[\overline{m}]-\nu+\alpha}\big\{|u|_{L^{\infty}(G_{1})}+|f|_{C^{\ell, \alpha}(\bar{G}_{1})}\big\} 
\quad \text {in }\bar{G}_{r}(x_{0}'),
%\end{split}    
\end{align*}
for some positive constant $C$ depending only on $n$, $\lambda$, $\ell$, $\alpha$, $r$, 
$\underline{m}$, $\overline{m}$, the $C^{\ell+1, \alpha}$-norms of $a_{n n}, b_{n}, c$ in $\bar{G}_{1}$, 
and the $C^{\ell, \alpha}$-norms of $a_{i j}, b_{i}$ in $\bar{G}_{1}$ for $i \neq n .$ 
\end{lemma}

\begin{proof}
%Our proof is based on the expression of $u$ given by \eqref{eq-ch4-LinearODE-Solution-l}. 
Note that $\overline{m},\underline{m},\gamma\in C^{\ell+1, \alpha}(B'_{1})$ 
by \eqref{eq-ch4-Assumption_m1}, \eqref{eq-ch4-Assumption_m2}, 
and \eqref{eq-definition-gamma}, and that $[\overline{m}]$ is an integer. 
Without loss of generality, we assume $x'_{0}=0$ %, $r_0=1$, 
and take a constant $\epsilon_0$ such that 
$0<\epsilon_0 \le\alpha-\gamma$ on $\bar B_{1}'$. 
Throughout the proof, $\epsilon$ is an arbitrary constant in $(0,\epsilon_{0})$ 
and $r$ is an arbitrary constant in $(0,1)$. 

%{\color{red} We first consider the case $[\overline{m}]\ge1$.} 
By Theorem \ref{thrm-ch3-Linear-NormalEstimate-general}, 
we have, for any $\nu \leq[\overline{m}]-1$ and $\tau \leq \ell-\nu$,
\begin{equation}\label{eq-ch4-regularity-u-non-int-preli1-large}
\partial_{t}^{\nu} D_{x^{\prime}}^{\tau} u, 
t D \partial_{t}^{\nu} D_{x^{\prime}}^{\tau} u, 
t^{2} D^{2} \partial_{t}^{\nu} D_{x^{\prime}}^{\tau} u \in C^{\alpha}(\bar{G}_{r}).    
\end{equation}
We now analyze the regularity of $F$, given by \eqref{eq-ch4-LinearODE-F}. 
We write 
\begin{align}\label{eq-ch4-LinearODE-F1}
F=\Tilde f+2\Tilde a_{\beta n}t^2\partial_{\beta t}u+
\Tilde a_{\alpha\beta}t^2\partial_{\alpha\beta}u
+\Tilde b_nt^2\partial_tu+\Tilde b_\beta t\partial_\beta u
+\Tilde ctu,\end{align}
where $\Tilde a_{\beta j}, \Tilde b_i, \Tilde c, \Tilde f\in C^{\ell,\alpha}(\bar G_1)$ 
by the assumption. 
Here, we treat $F$ as a function of $x'$ and $t$. 
Note that $F$ is  a linear combination of 
$f$,  $t^2\partial_{t}D_{x'}u$, $t^2D_{x'}^2u$, $t^2\partial_tu$, $tD_{x'}u$, and $ tu$. 
We need to calculate $\partial_t^\nu D_{x'}^\tau$ acting on these quantities, 
for $\nu\le [\overline{m}]$ and $\tau\le \ell-\nu$. 
For an illustration, we consider $t^2\partial_{t}D_{x'}u$. 
Note that, for any nonnegative integers $\nu$ and $\tau$, 
\begin{align*}
\partial^\nu_tD_{x'}^\tau(t^2\partial_{t}D_{x'}u)&=t^2\partial_t^2\partial_{t}^{\nu-1}D_{x'}^{\tau+1}u\\
&\qquad+2\nu t\partial_t\partial_{t}^{\nu-1} D_{x'}^{\tau+1}u+\nu(\nu-1)\partial_{t}^{\nu-1} D_{x'}^{\tau+1}u.\end{align*}
We intentionally write derivatives with respect to $t$ as above. 
By \eqref{eq-ch4-regularity-u-non-int-preli1-large}, we have, 
for any nonnegative integers  $\nu\le [\overline{m}]$ and $\tau\le \ell-\nu$, 
$$\partial^\nu_tD_{x'}^\tau(t^2\partial_{t}D_{x'}u)\in C^\alpha(\bar G_{r}).$$ 
A similar result holds for other terms in the expression of $F$ in \eqref{eq-ch4-LinearODE-F1}. 
Therefore, for any nonnegative integers $\nu\le [\overline{m}]$ and $\tau\le \ell-\nu$,  
\begin{equation}\label{eq-ch4-regularity-F1-non-int-large}
\partial^\nu_tD_{x'}^\tau F\in C^\alpha(\bar G_{r}).\end{equation}
Here, we treat $F$ as a function of $x'$ and $t$. 

%\textcolor{red}{This is red.} 

Set 
$$a_i=\frac{1}{i!}\partial_t^iF(\cdot,0)\quad\text{for }i=0, \cdots, [\overline{m}],$$
and write 
\begin{equation}\label{eq-ch4-identity-F-non-int-overline-m-large}F=\sum_{i=0}^{[\overline{m}]} a_it^i
+S_{[\overline{m}]}.\end{equation}
By \eqref{eq-ch4-regularity-F1-non-int-large}, we have 
\begin{equation}\label{eq-regularity-a-non-large} 
a_i\in C^{\ell-i, \alpha}(B_1')\quad\text{for any }i=0, \cdots, [\overline{m}],\end{equation}
and,  
for any $\nu\le [\overline{m}]$ and $\tau\le \ell-[\overline{m}]$, 
\begin{align}\label{eq-regularity-S-m-non-large}\begin{split} 
\partial_t^\nu D_{x'}^\tau S_{[\overline{m}]}&\in C^\alpha(\bar G_r),\\
|\partial_t^\nu D_{x'}^\tau S_{[\overline{m}]}|&\le Ct^{[\overline{m}]-\nu+\alpha}\quad\text{in }\bar G_r.
\end{split}\end{align}

By substituting \eqref{eq-ch4-identity-F-non-int-overline-m-large} in \eqref{eq-ch4-LinearODE-Solution-l} 
and a straightforward computation, we have 
\begin{align}\label{eq-ch4-indentity-u-int-overline-m-non-large}
u=\sum_{i=0}^{[\overline{m}]}c_it^i+c_{[\overline{m}],0}t^{[\overline{m}]+\gamma}
+R_{[\overline{m}]},
\end{align}
where 
\begin{align}\label{eq-ch4-expression-coefficient-m1-int-non-large}\begin{split}
c_{i}&=\frac{a_{i}}{(i-\underline{m})(i-\overline{m})}\quad\text{for }i=0, \cdots, [\overline{m}],\\
%\end{split}\end{align}
%and 
%\begin{align}\label{eq-expression-coefficient-m-int}\begin{split}
c_{[\overline{m}],0}&=u(\cdot,r)r^{-\overline{m}}-\sum_{i=0}^{[\overline{m}]}
\frac{a_ir^{i-\overline{m}}}{(i-\underline{m})(i-\overline{m})}\\
&\qquad 
+\frac{r^{\underline{m}-\overline{m}}}{\overline{m}-\underline{m}}
\int_0^r s^{-1-\underline{m}}S_{[\overline{m}]}ds
-\frac{1}{\overline{m}-\underline{m}}\int_0^rs^{-1-\overline{m}}S_{[\overline{m}]}ds,
\end{split}\end{align}
and
\begin{align}\label{eq-ch4-expression-remainder-m-int-non-large}
R_{[\overline{m}]}=\frac{t^{\overline{m}}}{\overline{m}-\underline{m}}\int_0^ts^{-1-\overline{m}}S_{[\overline{m}]}ds
-\frac{t^{\underline{m}}}{\overline{m}-\underline{m}}\int_0^ts^{-1-\underline{m}}S_{[\overline{m}]}ds.
\end{align} 
By \eqref{eq-regularity-a-non-large}, we have 
$$c_i\in C^{\ell-i,\alpha}(B'_1)\quad\text{for }i=0, \cdots, [\overline{m}].$$
In addition, by \eqref{eq-regularity-S-m-non-large}, 
Lemma \ref{lemma-ch4-BasicHolderRegularity} with $a=-\underline{m}$, 
and Lemma \ref{lemma-ch4-singular-integral-non-int-lower-larger} with $a=\overline{m}$, we have %, 
%for any $\epsilon\in (0,\alpha-\gamma)$,
$$c_{[\overline{m}],0}\in C^{\ell-[\overline{m}],\epsilon}(B'_1).$$ 
Next, by \eqref{eq-ch4-expression-remainder-m-int-non-large}, we write 
$$R_{[\overline{m}]}=\underline{R}_{[\overline{m}]}+\overline{R}_{[\overline{m}]},$$
where
\begin{align*}
\underline{R}_{[\overline{m}]}
&=-\frac{t^{\underline{m}}}{\overline{m}-\underline{m}}\int_0^ts^{-1-\underline{m}}S_{[\overline{m}]}ds,\\\overline{R}_{[\overline{m}]}
&=\frac{t^{\overline{m}}}{\overline{m}-\underline{m}}\int_0^ts^{-1-\overline{m}}S_{[\overline{m}]}ds.
\end{align*}
A simple computation yields
\begin{align*}
t\partial_tR_{[\overline{m}]}= 
\underline{m}\, \underline{R}_{[\overline{m}]}
+\overline{m}\,\overline{R}_{[\overline{m}]},
\end{align*}
and 
\begin{align*}
t^2\partial^2_{t}R_{[\overline{m}]}= S_{[\overline{m}]}
+\underline{m}(\underline{m}-1) \underline{R}_{[\overline{m}]}
+\overline{m}(\overline{m}-1)\overline{R}_{[\overline{m}]}.
\end{align*}
%We skip the rest of the proof. 
By \eqref{eq-regularity-S-m-non-large}, 
Lemma \ref{lemma-ch4-BasicHolderRegularity} with $a=-\underline{m}$, % and $\beta=\alpha$, 
and Lemma \ref{lemma-ch4-singular-integral-non-int-lower-larger}
with $a=\overline{m}$, we have, 
for any $\nu\le [\overline{m}]$ and $\tau\le \ell-[\overline{m}]$, 
%and any $\epsilon\in (0,\alpha-\gamma)$, 
\begin{align*}
\partial_t^\nu D^\tau_{x'}\underline{R}_{[\overline{m}]}\in C^{\alpha}(\bar G_r),
\quad 
\partial_t^\nu D^\tau_{x'}\overline{R}_{[\overline{m}]}\in C^{\epsilon, \alpha}_{x',t}(\bar{G}_r),\end{align*}
and 
\begin{align*}
|\partial_t^\nu D^\tau_{x'}\underline{R}_{[\overline{m}]}|
+|\partial_t^\nu D^\tau_{x'}\overline{R}_{[\overline{m}]}|\le Ct^{[\overline{m}]-\nu+\alpha}
\quad\text{in }G_{1/2}.\end{align*}
As a consequence, we obtain, 
for any $\nu\le [\overline{m}]$ and $\tau\le \ell-[\overline{m}]$, 
\begin{align*}
\partial_t^\nu D_{x'}^\tau R_{[\overline{m}]}, 
t\partial_t\partial_t^\nu D_{x'}^\tau R_{[\overline{m}]}, 
t^2\partial_t^2\partial_t^\nu D_{x'}^\tau R_{[\overline{m}]}\in C^{\epsilon, \alpha}_{x',t}(\bar{G}_r), 
\end{align*}
and
\begin{align*}
|\partial_t^\nu D_{x'}^\tau R_{[\overline{m}]}|
+|t\partial_t\partial_t^\nu D_{x'}^\tau R_{[\overline{m}]}|
+|t^2\partial_t^2\partial_t^\nu D_{x'}^\tau R_{[\overline{m}]}|
\le Ct^{[\overline{m}]-\nu+\alpha}
\quad\text{in }G_{1/2}.\end{align*}
We now analyze the regularity of $t^{-\gamma}R_{[\overline{m}]}$ similarly. 
By \eqref{eq-ch4-expression-remainder-m-int-non-large} again, we write 
\begin{align*}%\label{eq-ch4-expression-remainder-m-int-non-large}
t^{-\gamma}R_{[\overline{m}]}
=t^{-\gamma}\underline{R}_{[\overline{m}]}+t^{-\gamma}\overline{R}_{[\overline{m}]}, 
\end{align*} 
where 
\begin{align*}
t^{-\gamma}\underline{R}_{[\overline{m}]}
&=-\frac{t^{\underline{m}-\gamma}}{\overline{m}-\underline{m}}
\int_0^ts^{-1-(\underline{m}-\gamma)}\,s^{-\gamma}S_{[\overline{m}]}ds,\\
t^{-\gamma}\overline{R}_{[\overline{m}]}
&=\frac{t^{\overline{m}-\gamma}}{\overline{m}-\underline{m}}
\int_0^ts^{-1-(\overline{m}-\gamma)}\,s^{-\gamma}S_{[\overline{m}]}ds.
\end{align*} 
A similar computation yields
\begin{align*}
t\partial_t(t^{-\gamma}R_{[\overline{m}]})= 
(\underline{m}-\gamma)t^{-\gamma}\underline{R}_{[\overline{m}]}
+(\overline{m}-\gamma)t^{-\gamma}\overline{R}_{[\overline{m}]},
\end{align*}
and 
\begin{align*}
t^2\partial^2_{t}(t^{-\gamma}R_{[\overline{m}]})&= t^{-\gamma}S_{[\overline{m}]}
+(\underline{m}-\gamma)(\underline{m}-\gamma-1) t^{-\gamma}\underline{R}_{[\overline{m}]}\\
&\qquad+(\overline{m}-\gamma)(\overline{m}-\gamma-1)t^{-\gamma}\overline{R}_{[\overline{m}]}.
\end{align*}
By \eqref{eq-regularity-S-m-non-large} and Lemma \ref{lemma-ch4-regularity-power1}, we have, 
for any $\nu\le [\overline{m}]$ and $\tau\le \ell-[\overline{m}]$,  
\begin{align*}%\label{eq-regularity-S-m-non-large}\begin{split} 
\partial_t^\nu D_{x'}^\tau (t^{-\gamma}S_{[\overline{m}]})&\in C^{\epsilon}(\bar G_r),\\
|\partial_t^\nu D_{x'}^\tau (t^{-\gamma}S_{[\overline{m}]})|&
\le Ct^{[\overline{m}]-\nu+\epsilon}\quad\text{in }\bar G_r.
%\end{split}
\end{align*}
Note that $\underline{m}-\gamma<0$ and 
$\overline{m}-\gamma=[\overline{m}]$ by the definition of $\gamma$. 
By Lemma \ref{lemma-ch4-BasicHolderRegularity} 
with $a=\gamma-\underline{m}$, % and $\beta=\alpha$, 
and Lemma \ref{lemma-ch4-singular-integral-int-lower} 
with $a=[\overline{m}]$,  % and $\beta=\alpha$, 
we have, 
for any $\nu\le [\overline{m}]$ and $\tau\le \ell-[\overline{m}]$, 
\begin{align*}
\partial_t^\nu D^\tau_{x'}(t^{-\gamma}\underline{R}_{[\overline{m}]}),  
\partial_t^\nu D^\tau_{x'}(t^{-\gamma}\overline{R}_{[\overline{m}]})
\in C^{\epsilon}(\bar{G}_r).
\end{align*}
As a consequence, we obtain, 
for any $\nu\le [\overline{m}]$ and $\tau\le \ell-[\overline{m}]$, 
\begin{align*}%\label{eq-ch4-LinearRegularity-noon-int-m1-large}
\partial_t^\nu D_{x'}^\tau(t^{-\gamma}R_{[\overline{m}]}), 
t\partial_t\partial_t^\nu D_{x'}^\tau (t^{-\gamma}R_{[\overline{m}]}), 
t^2\partial_t^2\partial_t^\nu D_{x'}^\tau (t^{-\gamma}R_{[\overline{m}]})
\in C^{\epsilon}(\bar{G}_r).
\end{align*}
%{\color{blue} A simple computation yields
%$$t^{-\gamma}t\partial_{t}R_{[\overline{m}]}=t\partial_{t}(t^{-\gamma}R_{[\overline{m}]})
%+\gamma t^{-\gamma}R_{[\overline{m}]},$$
%and
%$$t^{-\gamma}t^{2}\partial^{2}_{t}R_{[\overline{m}]}
%=t^{2}\partial^{2}_{t}(t^{-\gamma}R_{[\overline{m}]})-\gamma(\gamma+1)t^{-\gamma}R_{[\overline{m}]}
%+2\gamma t^{-\gamma}t\partial_{t}R_{[\overline{m}]}.$$
%Hence, we obtain, for any $\nu\leq [\overline{m}]$ and $\tau\leq \ell-[\overline{m}]$, 
%$$\partial_{t}^{\nu} D_{x^{\prime}}^{\tau}(t^{-\gamma}R_{[\overline{m}]}),
%t\partial_{t}\partial_{t}^{\nu} D_{x^{\prime}}^{\tau}(t^{-\gamma}R_{[\overline{m}]}),
%t^{2}\partial^{2}_{t}\partial_{t}^{\nu} D_{x^{\prime}}^{\tau}(t^{-\gamma}R_{[\overline{m}]})
%\in C^{\epsilon}(\bar{G}_{r}).$$
%{\color{red} (Question: Why do we need this?)}}
We conclude the desired result.
%{\color{red} We conclude the desired result for the case $[\overline{m}]\ge1$.}
%
%{\color{red} Next, we prove the desired result for the case $[\overline{m}]=0$. 
%(Fill in some details.)}
\end{proof}

There is a loss of regularity of the remainder $R_{[\overline{m}]}$
in Lemma \ref{lemma-ch4-Linear-MainThm-non-integer-large}, 
with the loss only in the $x'$ direction.
This causes different H\"older indices in $t$ and $x'$ 
and is the motivation to introduce the more general H\"older space in 
Definition \ref{def-Holder-double-index}. 

We point out that $R_{[\overline{m}]}$ and $t^{-\gamma}R_{[\overline{m}]}$ 
have the same regularity in $x'$. To achieve this, we proved the regularity of 
$R_{[\overline{m}]}$ and $t^{-\gamma}R_{[\overline{m}]}$ parallelly. 
If we derive the regularity of $t^{-\gamma}R_{[\overline{m}]}$ 
from that of $R_{[\overline{m}]}$ by Lemma \ref{lemma-ch4-regularity-power1}, 
we will further lose a regularity of a H\"older index $\gamma$.

Next, we study the higher regularity. 
In the proof of the next result, we will use repeatedly the following simple identity. 
Let $\gamma$ be the function given by \eqref{eq-definition-gamma}. 
Then, for any $t>0$, 
$$D_{x'}t^\gamma=t^\gamma\log t\, D_{x'}\gamma.$$ 
We point out that the $x'$-derivative of $t^\gamma$ yields 
a factor of $\log t$ if $\gamma$ is not constant. 

\begin{theorem}\label{thrm-ch4-Linear-MainThm-non-integer-large}
Suppose that $\underline{m}$, $\overline{m}$, and $\gamma$ are functions on $B'_{1}$ satisfying 
\eqref{eq-ch4-Assumption_m1}, \eqref{eq-ch4-Assumption_m2}, 
and \eqref{eq-definition-gamma}, 
and that $\ell$ is an integer and $\alpha\in(0,1)$ is a constant such that $\ell\geq [\overline{m}] \ge1$ 
and \eqref{eq-large-Holder-index} holds. 
Assume $a_{n n}, b_{n}, c \in C^{\ell+1, \alpha}(\bar{G}_{1})$ 
and $a_{i j}, b_{i}, f \in C^{\ell, \alpha}(\bar{G}_{1})$ for $i \neq n$, 
and let $u \in C(\bar{G}_{1}) \cap C^{2}(G_{1})$ be a solution 
of \eqref{eq-ch3-Equ} and \eqref{eq-ch3-Dirichlet}. %, with \eqref{eq-ch3-boundary-value}. 
Then, for any $k$ with $[\overline{m}] \leq k \leq \ell$,
\begin{equation}\label{eq-ch4-LinearExpansion-non-int-k-large}
u=\sum_{i=0}^{k} c_{i} t^{i}+\sum^{k}_{i=[\overline{m}]}
\sum^{i-[\overline{m}]}_{j=0} c_{i, j} t^{i+\gamma}(\log t)^{j}+R_{k} \quad \text {in } G_{1},
\end{equation}
where $\{c_{i}\}_{i=0}^{\ell}$ and 
$\{c_{i, j}\}_{ [\overline{m}]\leq i\leq \ell,0\leq j\leq i-[\overline{m}]}$ 
are functions on $B'_{1}$ and $R_{k}$ is a function in $G_{1}$ satisfying, 
for any $x'_{0}\in B'_{1}$, 
any $r\in(0,1-|x'_{0}|)$, 
and any $\epsilon$ with $0<\epsilon <\alpha-\gamma$ on $\bar B_{r}'(x_0')$, 
\begin{align}\label{eq-ch4-regularity-coefficients-non-int-large}\begin{split} 
&c_{i} \in C^{\ell-i, \alpha}(\bar B_{r}^{\prime}(x'_{0}))\quad\text{for }i=0,\cdots,\ell,\\
&c_{i, j} \in C^{\ell-i, \epsilon}(\bar B_{r}^{\prime}(x'_{0}))
\quad\text{for }i=[\overline{m}],\cdots,\ell\text{ and }j=0,\cdots,i-[\overline{m}],
\end{split}    
\end{align}
and, for any nonnegative integers $\nu \leq k$ and $\tau \leq \ell-k$, %and any $r \in(0,r_{\epsilon})$,
\begin{align}\label{eq-ch4-LinearRegularity-noon-int-m111-large}
\partial_{t}^{\nu} D_{x^{\prime}}^{\tau} R_{k}, t \partial_{t} \partial_{t}^{\nu} D_{x^{\prime}}^{\tau} R_{k}, 
t^{2} \partial_{t}^{2} \partial_{t}^{\nu} D_{x^{\prime}}^{\tau} R_{k} 
&\in C^{\epsilon,\alpha}_{x',t}(\bar{G}_{r}(x_{0}')),\\ 
\label{eq-ch4-LinearRegularity-noon-int-m112-large}
\partial_{t}^{\nu} D_{x^{\prime}}^{\tau} (t^{-\gamma}R_{k}), 
t \partial_{t} \partial_{t}^{\nu} D_{x^{\prime}}^{\tau} (t^{-\gamma}R_{k}), 
t^{2} \partial_{t}^{2} \partial_{t}^{\nu} D_{x^{\prime}}^{\tau}(t^{-\gamma} R_{k}) 
&\in C^{\epsilon}(\bar{G}_{r}(x_{0}')),  
\end{align}
and
\begin{align}\label{eq-ch4-LinearRegularity-non-int-m2-large}\begin{split}
&|\partial_{t}^{\nu} D_{x^{\prime}}^{\tau} R_{k}|
+|t \partial_{t} \partial_{t}^{\nu} D_{x^{\prime}}^{\tau} R_{k}|
+|t^{2} \partial_{t}^{2} \partial_{t}^{\nu} D_{x^{\prime}}^{\tau} R_{k}| \\
&\qquad \leq C t^{k-\nu+\alpha}\big\{|u|_{L^{\infty}(G_{1})}
+|f|_{C^{\ell, \alpha}(\bar{G}_{1})}\big\} \quad \text {in }\bar{G}_{r}(x_{0}'),
\end{split}    
\end{align}
for some positive constant $C$ depending only on $n$, $\lambda$, $\ell$, $\alpha$, $r$, 
$\underline{m}$, $\overline{m}$, 
the $C^{\ell+1, \alpha}$-norms of $a_{n n}, b_{n}, c$ in $\bar{G}_{1}$, 
and the $C^{\ell, \alpha}$-norms of $a_{i j}, b_{i}$ in $\bar{G}_{1}$ for $i \neq n .$ 
If, in addition, $c_{[\overline{m}], 0}=0$ on $B'_r(x_0')$, 
for some $x'_{0}\in B'_{1}$ and $r\in(0,1-|x'_{0}|)$, 
then $c_{i, j}=0$ on $B'_r(x_0')$, 
for any $i=[\overline{m}], \cdots, \ell$ and $j=0,\cdots,i-[\overline{m}]$, 
and $u \in C^{\ell, \alpha}(\bar{G}_{r}(x_{0}'))$. 
%for any $\epsilon$ with $0<\epsilon <\alpha-\gamma$ on $\bar B_{r}'(x_0')$. 
\end{theorem}

In the second summation in the right-hand side of 
\eqref{eq-ch4-LinearExpansion-non-int-k-large}, 
the indices $i,j$ in $c_{i,j}$ refer to the integer part of the power of $t$ and the power of $\log t$, 
respectively. 
%Our proof is based on the expression of $u$ given by \eqref{eq-ch4-LinearODE-Solution-l}.

\begin{proof}
We adopt the same notations as in the proof of Lemma \ref{lemma-ch4-Linear-MainThm-non-integer-large}. 
We will prove 
\eqref{eq-ch4-LinearExpansion-non-int-k-large}-\eqref{eq-ch4-LinearRegularity-non-int-m2-large} 
by induction on $k$. We note that 
\eqref{eq-ch4-LinearExpansion-non-int-k-large}-\eqref{eq-ch4-LinearRegularity-non-int-m2-large} 
hold for $k=[\overline{m}]$ 
by Lemma \ref{lemma-ch4-Linear-MainThm-non-integer-large}.
Without loss of generality, we assume $x'_{0}=0$ %, $r_0=1$, 
and take a constant $\epsilon_0$ such that 
$0<\epsilon_0 \le\alpha-\gamma$ on $\bar B_{1}'$. 
Throughout the proof, $\epsilon$ is an arbitrary constant in $(0,\epsilon_{0})$ 
and $r$ is an arbitrary constant in $(0,1)$. 
The proof consists of two steps.

{\it Step 1.} 
We fix an integer $k$ with $[\overline{m}]<k\leq \ell$. 
We assume that 
\eqref{eq-ch4-LinearExpansion-non-int-k-large}-\eqref{eq-ch4-LinearRegularity-non-int-m2-large} 
hold with $k$ replaced by any integer between $[\overline{m}]$ and $k-1$ and then proceed to prove 
\eqref{eq-ch4-LinearExpansion-non-int-k-large}-\eqref{eq-ch4-LinearRegularity-non-int-m2-large} 
for $k$. 
With $c_{i}$ for $0\leq i\leq k-1$ and $c_{i,j}$ for $[\overline{m}]\leq i\leq k-1$ 
and $0\leq j\leq i-[\overline{m}]$ already determined, 
we will find $c_{k}$, $c_{k,0}$, $\cdots$, $c_{k,k-[\overline{m}]}$, and $R_{k}$, 
and prove that they have the stated properties. 
In the following, we will use 
\eqref{eq-ch4-LinearExpansion-non-int-k-large}-\eqref{eq-ch4-LinearRegularity-non-int-m2-large} 
with $k$ replaced by $k-1$ and $k-2$. 

We first analyze $F$, which is given by \eqref{eq-ch4-LinearODE-F1};  
namely, 
\begin{equation*}
F=\Tilde{f}+2 \Tilde{a}_{\alpha n} t^2 \partial_{\alpha t}u
+\Tilde{a}_{\alpha \beta} t^2 \partial_{\alpha \beta}u+\Tilde{b}_n t^2 \partial_t u
+\Tilde{b}_\alpha t \partial_\alpha u+\Tilde{c} t u,
\end{equation*}
where $\Tilde{a}_{\beta j},\Tilde{b}_{i},\Tilde{c},\Tilde{f}
\in C^{\ell,\alpha}(\bar{G}_{1})$ by the assumption.
We will prove
\begin{equation}\label{eq-ch4-identity-F-non-int-k-large}
F=\sum_{i=0}^{k} a_{i} t^{i}+\sum^{k}_{i=[\overline{m}]+1}\sum^{i-[\overline{m}]}_{j=0} 
a_{i,j} t^{i+\gamma}(\log t)^{j}+S_{k},
\end{equation}
where $a_{i}$ and $a_{i,j}$ are functions on $B'_{1}$ and $S_{k}$ is a function in $G_{1}$ such that
\begin{align}\label{eq-ch4-regularity-coefficients-F-non-int-k-large}\begin{split}
&a_{i}\in C^{\ell-i,\alpha}(B_{1}')\quad \text{for }i=0,\cdots,k,\\
&a_{i,j}\in C^{\ell-i,\epsilon}(B_{1}')\quad 
\text{for }i=[\overline{m}]+1,\cdots,k\text{ and }j=0,\cdots,i-[\overline{m}],
\end{split}\end{align}
and, for any nonnegative integers $\nu\leq k$ and $\tau\leq \ell-k$, 
\begin{align}\label{eq-ch4-LinearRegularity-F-m1-non-int-k-large}\begin{split}
&\partial_{t}^{\nu} D_{x^{\prime}}^{\tau} S_{k}\in C^{\epsilon,\alpha}_{x',t}(\bar{G}_{r}),\\
&\partial_{t}^{\nu} D_{x^{\prime}}^{\tau}(t^{-\gamma}S_{k})
\in C^{\epsilon}(\bar{G}_{r}),\\
&|\partial_{t}^{\nu} D_{x^{\prime}}^{\tau} S_{k}|\leq Ct^{k-\nu+\alpha}\quad\text{in }G_{1/2}.
\end{split}\end{align}
We emphasize that the summation in \eqref{eq-ch4-identity-F-non-int-k-large} 
for $t^{i+\gamma}$ starts from $i=[\overline{m}]+1$. 
The regularity of $a_{i}$ for $i=0,\cdots,[\overline{m}]$ is already established 
in \eqref{eq-regularity-a-non-large}. 
The regularity of $a_{i}$ and $a_{i,j}$ for $i=[\overline{m}]+1, \cdots, k-1$ 
and $j=0,\cdots,i-[\overline{m}]$ is the consequence of the induction hypothesis. 
We will focus on $a_k$, $a_{k,j}$ for $j=0,\cdots,k-[\overline{m}]$, and $S_{k}$.

In the expression of $F$ in \eqref{eq-ch4-LinearODE-F1}, 
there are five types of $u$ and its derivatives, given by 
$t^{2}D^{2}_{x'}u$, $t^{2}D_{x'}\partial_{t}u$, $t^{2}\partial_{t}u$, $tD_{x'}u$, and $tu$. 
We will use \eqref{eq-ch4-LinearExpansion-non-int-k-large} for $k-2$ to analyze $t^{2}D^{2}_{x'}u$ 
and use \eqref{eq-ch4-LinearExpansion-non-int-k-large} for $k-1$ to analyze 
$t^{2}D_{x'}\partial_{t}u$, $t^{2}\partial_{t}u$, $tD_{x'}u$, and $tu$. 
If we use \eqref{eq-ch4-LinearExpansion-non-int-k-large} for $k-1$ to study $t^{2}D^{2}_{x'}u$, 
there will be a loss of regularity in $x'$. 
In the following, we set
\begin{align*}
&c_{i,j}=0\quad\text{for }(i,j)\notin\{(p,q)\in \mathbb{Z}^{2}:
[\overline{m}]\leq p\leq \ell,0\leq q\leq p-[\overline{m}]\},\\
&a_{i,j}=0\quad\text{for }(i,j)\notin\{(p,q)\in \mathbb{Z}^{2}:
[\overline{m}]+1\leq p\leq \ell,0\leq q\leq p-[\overline{m}]\}.
\end{align*}
The proof of 
\eqref{eq-ch4-identity-F-non-int-k-large}-\eqref{eq-ch4-LinearRegularity-F-m1-non-int-k-large}
is lengthy. We will provide a detailed analysis 
of the term $\Tilde{a}_{\alpha \beta} t^2 \partial_{\alpha \beta}u$
and indicate how to analyze the rest of terms in \eqref{eq-ch4-LinearODE-F1}.

We first study $\Tilde{a}_{\alpha \beta} t^2 \partial_{\alpha \beta}u$. 
By \eqref{eq-ch4-LinearExpansion-non-int-k-large} with $k$ replaced by $k-2$, we have
$$
u=\sum_{p=0}^{k-2} c_{p} t^{p}+\sum^{k-2}_{p=[\overline{m}]}
\sum^{p-[\overline{m}]}_{j=0} c_{p,j} t^{p+\gamma}(\log t)^{j}+R_{k-2}.
$$
Then,
$$
t^{2}\partial_{\alpha\beta}u=\sum_{p=0}^{k-2} \partial_{\alpha\beta}c_{p} t^{p+2}
+\sum^{k-2}_{p=[\overline{m}]}\sum^{p-[\overline{m}]+2}_{j=0} b_{p,j}t^{p+2+\gamma}(\log t)^{j}
+t^{2}\partial_{\alpha\beta}R_{k-2},
$$
where, for $[\overline{m}]\leq p\leq k-2$ and $0\leq j\leq p-[\overline{m}]+2$,
\begin{align}\label{eq-ch4-expression-b-non-int-k-large}\begin{split}
b_{p,j}&=\partial_{\alpha\beta}c_{p,j}+\partial_{\alpha}c_{p,j-1}\partial_{\beta}\gamma
+\partial_{\beta}c_{p,j-1}\partial_{\alpha}\gamma+c_{p,j-1}\partial_{\alpha\beta}\gamma\\
&\qquad+c_{p,j-2}\partial_{\alpha}\gamma\partial_{\beta}\gamma.
\end{split}\end{align}
We will multiply each term in $t^{2}\partial_{\alpha\beta}u$ by $\tilde{a}_{\alpha\beta}$. 
For each $p=0,\cdots,k-2$, we write
$$
\Tilde{a}_{\alpha\beta}=\sum^{k-p-2}_{q=0}\Tilde{a}_{\alpha\beta,q}t^{q}
+S_{k-p-2}(\Tilde{a}_{\alpha\beta}).
$$
We now examine $t^{2}\Tilde{a}_{\alpha\beta}\partial_{\alpha\beta}u$ in \eqref{eq-ch4-LinearODE-F1}. 
A simple computation shows
$$
t^{2}\Tilde{a}_{\alpha\beta}\partial_{\alpha\beta}u=\sum^{k}_{i=2}a^{(1)}_{i}t^{i}
+\sum^{k}_{i=[\overline{m}]+2}\sum^{i-[\overline{m}]}_{j=0}a^{(1)}_{i,j}t^{i+\gamma}(\log t)^{j}+S^{(1)}_{k},
$$
where
\begin{align*}
a^{(1)}_{i}&=\sum_{p+q=i-2}\tilde{a}_{\alpha\beta,q}\partial_{\alpha\beta}c_{p},\\
a^{(1)}_{i,j}&=\sum_{p+q=i-2}\tilde{a}_{\alpha\beta,q}b_{p,j},
\end{align*}
and
\begin{align}\label{eq-expression-S-1-k-large}
S^{(1)}_{k}=t^{2}\Tilde{a}_{\alpha\beta}\partial_{\alpha\beta}R_{k-2}
+\sum^{k-2}_{p=0}\big[\partial_{\alpha\beta}c_{p}t^{p+2}
+\sum^{p-[\overline{m}]+2}_{j=0}b_{p,j}t^{p+2+\gamma}(\log t)^{j}]S_{k-p-2}(\Tilde{a}_{\alpha\beta}).
\end{align}
%Then,
%\begin{align*}
%t^{-\gamma}S^{(1)}_{k}&=t^{2}\Tilde{a}_{\alpha\beta}t^{-\gamma}\partial_{\alpha\beta}R_{k-2}\\
%&\qquad+\sum^{k-2}_{p=0}\big[\partial_{\alpha\beta}c_{p}t^{p+2-\gamma}
%+\sum^{p-[\overline{m}]+2}_{j=0}b_{p,j}t^{p+2}(\log t)^{j}]S_{k-p-2}(\Tilde{a}_{\alpha\beta}).
%\end{align*}
We first study the regularity of $a^{(1)}_{k}$ and $a^{(1)}_{k,j}$ for $j\leq k-[\overline{m}]$. 
Recall that $\gamma\in C^{\ell+1, \alpha}(B'_{1})$. 
By \eqref{eq-ch4-regularity-coefficients-non-int-large}, we have, for each $p\leq k-2$,
$$
c_{p}\in C^{\ell-p,\alpha}(B'_{1}),\quad c_{p,h}\in C^{\ell-p,\epsilon}(B'_{1})\quad\text{for }h\leq p-[\overline{m}],
$$
and hence, 
$$
\partial_{\alpha\beta}c_{p}\in  C^{\ell-p-2,\alpha}(B'_{1}),
$$
and, by \eqref{eq-ch4-expression-b-non-int-k-large},
$$
b_{p,j}\in C^{\ell-p-2,\epsilon}(B_{1}')\quad\text{for }j\leq p-[\overline{m}]+2. 
$$
By $\Tilde{a}_{\alpha\beta}\in C^{\ell,\alpha}(\bar{G}_{1})$, we have, for each $q\leq k-p-2$,
$$
\Tilde{a}_{\alpha\beta,q}\in C^{\ell-q,\alpha}(B'_{1}).
$$
It is obvious that 
$$
\ell-p-2\geq\ell-k,\quad\ell-q\geq\ell-k,
$$
for each $p\leq k-2$ and $q\leq k-p-2$.
Therefore, we obtain, for each pair $p$ and $q$ with $p+q=k-2$,
$$
\Tilde{a}_{\alpha\beta,q}\partial_{\alpha\beta}c_{p}\in  C^{\ell-k,\alpha}(B'_{1}),
\quad\Tilde{a}_{\alpha\beta,q}b_{p,j}
\in  C^{\ell-k,\epsilon}(B'_{1})\quad\text{for }j\leq p-[\overline{m}]+2,
$$
and hence
$$
a^{(1)}_{k}\in  C^{\ell-k,\alpha}(B'_{1}),\quad a^{(1)}_{k,j}\in  C^{\ell-k,\epsilon}(B'_{1})
\quad\text{for }j\leq k-[\overline{m}].
$$
Next, we analyze
$t^2\partial_{\alpha \beta}R_{k-2}$,
$t^{p+2}S_{k-p-2}(\Tilde a_{\alpha\beta})$, and
$t^{p+2+\gamma}(\log t)^{j}S_{k-p-2}(\Tilde a_{\alpha\beta})$ in $S_k^{(1)}$.
For any nonnegative integers $\nu\le k$ and $\tau\le \ell-k$, a simple computation yields
\begin{align*}
\partial_t^\nu D_{x'}^\tau (t^2\partial_{\alpha \beta}R_{k-2})
&=t^2\partial_t^2\partial_t^{\nu-2} D_{x'}^{\tau+2} R_{k-2}
+2\nu t\partial_t\partial_t^{\nu-2} D_{x'}^{\tau+2} R_{k-2}\\
&\qquad
+\nu(\nu-1)\partial_t^{\nu-2} D_{x'}^{\tau+2}R_{k-2}.
\end{align*}
We intentionally write derivatives with respect to $t$ as above.
By \eqref{eq-ch4-LinearRegularity-noon-int-m111-large} 
and \eqref{eq-ch4-LinearRegularity-non-int-m2-large} with $k$ replaced by $k-2$, 
we have, for any $\nu\le k$ and $\tau\le \ell-k$,
\begin{align}\label{eq-ch4-LinearRegularity-R1-large}\begin{split}
&\partial_t^\nu D_{x'}^\tau (t^2\partial_{\alpha \beta}R_{k-2})
\in C^{\epsilon,\alpha}_{x',t}(\bar{G}_r),\\
&|\partial_t^\nu D_{x'}^\tau (t^2 \partial_{\alpha \beta}R_{k-2})|\leq  Ct^{k-\nu+\alpha}
\quad\text{in }G_{1/2}.
\end{split}\end{align}
We next study $\partial_t^\nu D_{x'}^\tau (t^{p+2} S_{k-p-2}(\Tilde a_{\alpha\beta}))$ and
$\partial_t^\nu D_{x'}^\tau (t^{p+2+\gamma}(\log t)^{j}S_{k-p-2}(\Tilde a_{\alpha\beta}))$
for $p\le k-2$, $j\leq p-[\overline{m}]+2$, $\nu\le k$, and $\tau\le \ell-k$.
For an illustration, we consider $p= k-2$, $j=1$, $\nu= k$, and $\tau= \ell-k$.
Set
$$I= \partial_t^k D_{x'}^{\ell-k} (t^k S_{0}(\Tilde a_{\alpha\beta})), \quad
J=\partial_t^k D_{x'}^{\ell-k} (t^{k+\gamma}\log t\,{S}_{0}(\Tilde a_{\alpha\beta})).$$
A simple computation yields
$$I=\sum_{\sigma=0}^k\Tilde{c}^{(1)}_{k,\sigma}(t^{k})^{(k-\sigma)}\partial_t^\sigma
D_{x'}^{\ell-k}(S_{0}(\Tilde a_{\alpha\beta})),$$
for some constants $\Tilde c^{(1)}_{k,\sigma}$. Recall that
$$S_{0}(\Tilde a_{\alpha\beta})(x',t)=\Tilde a_{\alpha\beta}(x',t)-\Tilde a_{\alpha\beta}(x',0).$$
The $t$-derivative of the second term is 0.
We consider $\sigma=0$ and $1\le \sigma\le k$
in the summation separately. By renaming $\Tilde c^{(1)}_{k,\sigma}$, we write
$$I=\Tilde c^{(1)}_{k,0}[D_{x'}^{\ell-k}\Tilde a_{\alpha\beta}(x',t)-D_{x'}^{\ell-k}\Tilde a_{\alpha\beta}(x',0)]
+\sum_{\sigma=1}^k\Tilde c^{(1)}_{k,\sigma}t^{\sigma}\partial_t^\sigma
D_{x'}^{\ell-k}\Tilde a_{\alpha\beta}.$$
By $k=p+2\ge 2$, we have $\ell-k\le \ell-2$.
With $\sigma=k$, we have $\partial_t^kD_{x'}^{\ell-k}\Tilde a_{\alpha\beta}\in C^{\alpha}(\bar{G}_r)$.
This is the function in $I$ with the worst regularity. Moreover, each term in $I$ has a factor of $t$.
Hence,
\begin{align*}
I\in C^{\alpha}(\bar{G}_r)\quad\text{and}\quad
|I|\le  Ct
\quad\text{in }G_{1/2}.\end{align*}
Similarly, a simple computation yields
$$J=\sum_{\sigma=0}^k \sum_{\mu=0}^{\ell-k}\Tilde{b}^{(1)}_{k,\ell-k,\sigma,\mu}\partial_t^{k-\sigma}
D_{x'}^{\ell-k-\mu}(t^{k+\gamma}\log t)\partial_t^\sigma
D_{x'}^{\mu}(S_{0}(\Tilde a_{\alpha\beta})).$$
We consider $\sigma=0$ and $1\le \sigma\le k$
in the summation separately, and write
$$
J=J_{1}+J_{2},
$$
where
$$
J_{1}=\sum_{\mu=0}^{\ell-k}\Tilde{b}^{(1)}_{k,\ell-k,0,\mu}\partial_t^{k}
D_{x'}^{\ell-k-\mu}(t^{k+\gamma}\log t)D_{x'}^{\mu}(S_{0}(\Tilde a_{\alpha\beta})),
$$
and
$$
J_{2}=\sum_{\sigma=1}^k \sum_{\mu=0}^{\ell-k}\Tilde{b}^{(1)}_{k,\ell-k,\sigma,\mu}\partial_t^{k-\sigma}
D_{x'}^{\ell-k-\mu}(t^{k+\gamma}\log t)\partial_t^\sigma
D_{x'}^{\mu}(S_{0}(\Tilde a_{\alpha\beta})).
$$
Note that, for $\mu\leq \ell-k\leq \ell-2$,
$$D_{x'}^{\mu}\big(S_{0}(\Tilde a_{\alpha\beta})\big)(x',t)=D_{x'}^{\mu}\Tilde a_{\alpha\beta}(x',t)
-D_{x'}^{\mu}\Tilde a_{\alpha\beta}(x',0)=t\int^{1}_{0}\partial_{t}D_{x'}^{\mu}\Tilde a_{\alpha\beta}(x',t\rho)d\rho.$$
Then, 
\begin{align*}
J_{1}&=\sum_{\mu=0}^{\ell-k}\sum^{\ell-k-\mu+1}_{\pi=0}b^{(1)}_{k,\ell-k,0,\mu,\pi}
t^{\gamma}(\log t)^{\pi}D_{x'}^{\mu}(S_{0}(\Tilde a_{\alpha\beta}))\\
&=\sum_{\mu=0}^{\ell-k}\sum^{\ell-k-\mu+1}_{\pi=0}b^{(1)}_{k,\ell-k,0,\mu,\pi}
t^{1+\gamma}(\log t)^{\pi}\int^{1}_{0}\partial_{t}D_{x'}^{\mu}\Tilde a_{\alpha\beta}(x',t\rho)d\rho,
\end{align*}
where $b^{(1)}_{k,\ell-k,0,\mu,\pi}$ is a polynomial of $D^{s}_{x'}\gamma$, for $s\leq \ell-k-\mu$. Hence,
\begin{align*}
J_1\in C^{\alpha}(\bar{G}_r)\quad\text{and}\quad
|J_1|\le  Ct^{1+\gamma}|\log t|^{\ell-k+1}
\quad\text{in }G_{1/2}.
\end{align*}
Next, with $\sigma=k$ and $\mu=\ell-k$, we have 
$\partial_t^kD_{x'}^{\ell-k}\Tilde a_{\alpha\beta}\in C^{\alpha}(\bar{G}_r)$. 
This is the function in $J_2$ with the worst regularity. 
Moreover, each term in $J_2$ has a factor of $t^{1+\gamma}$.
Hence,
\begin{align*}
J_2\in C^{\alpha}(\bar{G}_r)\quad\text{and}\quad
|J_2|\le  Ct^{1+\gamma}|\log t|^{\ell-k+1}
\quad\text{in }G_{1/2}.
\end{align*}
In summary, we have
\begin{align*}
J\in C^{\alpha}(\bar{G}_r)\quad\text{and}\quad
|J|\le  Ct^{1+\gamma}|\log t|^{\ell-k+1}
\quad\text{in }G_{1/2}.
\end{align*}
In general, we have, for any $p\le k-2$, any $j\leq p-[\overline{m}]+2$, any $\nu\le k$, and $\tau\le \ell-k$,
\begin{align}\label{eq-ch4-LinearRegularity-R2-large}\begin{split}
&\partial_t^\nu D_{x'}^\tau (t^{p+2}S_{k-p-2}(\Tilde a_{\alpha\beta}))
\in C^{\alpha}(\bar{G}_r),\\
&|\partial_t^\nu D_{x'}^\tau (t^{p+2}S_{k-p-2}(\Tilde a_{\alpha\beta}))|
\le  Ct^{k-\nu+1}
\quad\text{in }G_{1/2},
\end{split}\end{align}
and
\begin{align}\label{eq-ch4-LinearRegularity-R3-large}\begin{split}
&\partial_t^\nu D_{x'}^\tau (t^{p+2+\gamma}(\log t)^{j}S_{k-p-2}(\Tilde a_{\alpha\beta}))
\in C^{\alpha}(\bar{G}_r),\\
&|\partial_t^\nu D_{x'}^\tau (t^{p+2+\gamma}(\log t)^{j}S_{k-p-2}(\Tilde a_{\alpha\beta}))|
\le  Ct^{k+\gamma-\nu+1}|\log t|^{j+\tau}
\quad\text{in }G_{1/2}.
\end{split}\end{align}
By combining \eqref{eq-ch4-LinearRegularity-R1-large}, \eqref{eq-ch4-LinearRegularity-R2-large}, 
and \eqref{eq-ch4-LinearRegularity-R3-large}, 
we obtain, for any nonnegative integers $\nu\leq k$ and $\tau\leq\ell-k$, 
\begin{align*}
&\partial_{t}^{\nu} D_{x^{\prime}}^{\tau} S^{(1)}_{k}\in C^{\epsilon,\alpha}_{x',t}(\bar{G}_{r}),\\
&|\partial_{t}^{\nu} D_{x^{\prime}}^{\tau} S^{(1)}_{k}|\leq Ct^{k-\nu+\alpha}\quad\text{in }G_{1/2}.
\end{align*}
Next, we analyze each term in $t^{-\gamma}S^{(1)}_{k}$. 
By \eqref{eq-expression-S-1-k-large}, we have
\begin{align*}
t^{-\gamma}S^{(1)}_{k}&=t^{2}\Tilde{a}_{\alpha\beta}t^{-\gamma}\partial_{\alpha\beta}R_{k-2}\\
&\qquad+\sum^{k-2}_{p=0}\big[\partial_{\alpha\beta}c_{p}t^{p+2-\gamma}
+\sum^{p-[\overline{m}]+2}_{j=0}b_{p,j}t^{p+2}(\log t)^{j}]S_{k-p-2}(\Tilde{a}_{\alpha\beta}).
\end{align*}
A simple computation yields
$$
t^{-\gamma}\partial_{\alpha}R_{k-2}=\partial_{\alpha}(t^{-\gamma}R_{k-2})
+\partial_{\alpha}\gamma \log t (t^{-\gamma}R_{k-2}),
$$
and
\begin{align*}
t^{-\gamma}\partial_{\alpha\beta}R_{k-2}
&=\partial_{\alpha}(t^{-\gamma}\partial_{\beta}R_{k-2})
+\partial_{\alpha}\gamma \log t (t^{-\gamma}\partial_{\beta}R_{k-2})\\
&=\partial_{\alpha\beta}(t^{-\gamma}R_{k-2})
+\big[\partial_{\beta}\gamma \partial_{\alpha}(t^{-\gamma}R_{k-2})
+\partial_{\alpha}\gamma\partial_{\beta}(t^{-\gamma}R_{k-2})\big]\log t\\
&\qquad +\partial_{\alpha\beta}\gamma\,t^{-\gamma}R_{k-2}\log t
+\partial_{\alpha}\gamma \partial_{\beta}\gamma  \,t^{-\gamma}R_{k-2}(\log t)^2.
\end{align*}
By \eqref{eq-ch4-LinearRegularity-noon-int-m112-large} with $k$ replaced by $k-2$, 
and Lemma \ref{lemma-ch4-regularity-log}(i), we have, for any $\nu\le k$ and $\tau\le \ell-k$,
$$
\partial_t^\nu D_{x'}^\tau (t^2t^{-\gamma}\partial_{\alpha \beta}R_{k-2})
\in C^{\epsilon}(\bar{G}_r).
$$
By \eqref{eq-ch4-LinearRegularity-R2-large}, 
\eqref{eq-ch4-LinearRegularity-R3-large}, and Lemma \ref{lemma-ch4-regularity-power1}, we have, 
for any $p\le k-2$, any $j\leq p-[\overline{m}]+2$, any $\nu\le k$, and $\tau\le \ell-k$,
$$
\partial_t^\nu D_{x'}^\tau (t^{p+2-\gamma}S_{k-p-2}(\Tilde a_{\alpha\beta})),
\partial_t^\nu D_{x'}^\tau (t^{p+2}(\log t)^{j}S_{k-p-2}(\Tilde a_{\alpha\beta}))
\in C^{\epsilon}(\bar{G}_r).
$$
Hence, we obtain, for any nonnegative integers $\nu\leq k$ and $\tau\leq\ell-k$, 
$$
\partial_{t}^{\nu} D_{x^{\prime}}^{\tau}(t^{-\gamma}S^{(1)}_{k})\in 
C^{\epsilon}(\bar{G}_{r}).
$$

Next, we analyze the rest of the terms in $F$. 
By \eqref{eq-ch4-LinearExpansion-non-int-k-large} with $k$ replaced by $k-1$, we have 
$$
u=\sum_{p=0}^{k-1} c_{p} t^{p}
+\sum^{k-1}_{p=[\overline{m}]}\sum^{p-[\overline{m}]}_{j=0} c_{p,j} t^{p+\gamma}(\log t)^{j}+R_{k-1}.
$$
Then,
$$
t^{2}\partial_{\alpha t}u=\sum_{p=0}^{k-1} p\partial_{\alpha}c_{p} t^{p+1}+\sum^{k-1}_{p=[\overline{m}]}\sum^{p-[\overline{m}]+1}_{j=0} e_{p,j}t^{p+1+\gamma}(\log t)^{j}+t^{2}\partial_{\alpha t}R_{k-1},
$$
where, for $[\overline{m}]\leq p\leq k-1$ and $0\leq j\leq p-[\overline{m}]+1$, 
\begin{align*}
e_{p,j}&=(p+\gamma)(\partial_{\alpha}c_{p,j}+c_{p,j-1}\partial_{\alpha}\gamma)
+(j+1)(\partial_{\alpha}c_{p,j+1}+c_{p,j}\partial_{\alpha}\gamma).
\end{align*}
For each $p=0,\cdots,k-1$, we write
$$
\Tilde{a}_{\alpha n}=\sum^{k-p-1}_{q=0}\tilde{a}_{\alpha n,q}t^{q}+S_{k-p-1}(\Tilde{a}_{\alpha n}).
$$
By a similar computation, we have
$$
t^{2}\Tilde{a}_{\alpha n}\partial_{\alpha t}u=\sum^{k}_{i=1}a^{(2)}_{i}t^{i}
+\sum^{k}_{i=[\overline{m}]+1}\sum^{i-[\overline{m}]}_{j=0}a^{(2)}_{i,j}t^{i+\gamma}(\log t)^{j}+S^{(2)}_{k},
$$
where 
\begin{align*}
a^{(2)}_{i}&=\sum_{p+q=i-1}p\Tilde{a}_{\alpha n,q}\partial_{\alpha}c_{p},\\
a^{(2)}_{i,j}&=\sum_{p+q=i-1}\Tilde{a}_{\alpha n,q}e_{p,j},\\
\end{align*}
and
\begin{align*}
S^{(2)}_{k}&=t^{2}\Tilde{a}_{\alpha n}\partial_{\alpha t}R_{k-1}\\
&\qquad+\sum^{k-1}_{p=0}\big[p\partial_{\alpha}c_{p}t^{p+1}
+\sum^{p-[\overline{m}]+1}_{j=0}e_{p,j}t^{p+1+\gamma}(\log t)^{j}\big]S_{k-p-1}(\Tilde{a}_{\alpha n}).
\end{align*}
Then,
\begin{align*}
t^{-\gamma}S^{(2)}_{k}&=t^{2}\tilde{a}_{\alpha n}t^{-\gamma}\partial_{\alpha t}R_{k-1}\\
&\qquad+\sum^{k-1}_{p=0}\big[p\partial_{\alpha}c_{p}t^{p+1-\gamma}
+\sum^{p-[\overline{m}]+1}_{j=0}e_{p,j}t^{p+1}(\log t)^{j}\big]S_{k-p-1}(\Tilde{a}_{\alpha n}).
\end{align*}
A simple computation yields
$$
tt^{-\gamma}\partial_{t}R_{k-1}=t\partial_{t}(t^{-\gamma}R_{k-1})+\gamma t^{-\gamma}R_{k-1},
$$
and
\begin{align*}
tt^{-\gamma}\partial_{\alpha t}R_{k-1}&=t\big[\partial_{\alpha}(t^{-\gamma}\partial_{t}R_{k-1})
+\partial_{\alpha}\gamma \log t (t^{-\gamma}\partial_{t}R_{k-1})\big]\\
&=t\partial_{\alpha t}(t^{-\gamma}R_{k-1})+\partial_{\alpha}\gamma t^{-\gamma}R_{k-1}
+\gamma\partial_{\alpha}( t^{-\gamma}R_{k-1})\\
&\qquad+\partial_{\alpha}\gamma  \big[t\partial_{t}(t^{-\gamma}R_{k-1})
+\gamma t^{-\gamma}R_{k-1}\big]\log t.
\end{align*}
Similarly, by \eqref{eq-ch4-regularity-coefficients-non-int-large} for $i\leq k-1$ and $j\leq i-[\overline{m}]$, 
and \eqref{eq-ch4-LinearRegularity-noon-int-m111-large},
\eqref{eq-ch4-LinearRegularity-noon-int-m112-large},
and \eqref{eq-ch4-LinearRegularity-non-int-m2-large} with $k$ replaced by $k-1$, we obtain, 
$$
a^{(2)}_{k}\in  C^{\ell-k,\alpha}(B'_{1}),\quad a^{(2)}_{k,j}
\in  C^{\ell-k,\epsilon}(B'_{1})\quad\text{for }j\leq k-[\overline{m}],
$$
and, for any nonnegative integers $\nu\leq k$ and $\tau\leq\ell-k$, 
$$
\begin{aligned}
&\partial_{t}^{\nu} D_{x^{\prime}}^{\tau} S^{(2)}_{k}\in C^{\epsilon,\alpha}_{x',t}(\bar{G}_{r}),\\
&\partial_{t}^{\nu} D_{x^{\prime}}^{\tau}(t^{-\gamma}S^{(2)}_{k})
\in C^{\epsilon}(\bar{G}_{r}),\\
&|\partial_{t}^{\nu} D_{x^{\prime}}^{\tau} S^{(2)}_{k}|\leq Ct^{k-\nu+\alpha}\quad\text{in }G_{1/2}.
\end{aligned}
$$
We have similar results for $t^{2}\Tilde{b}_{n}\partial_{t}u$, 
$t\Tilde{b}_{\alpha}\partial_{\alpha}u$, and $t\Tilde{c}u$. 

In conclusion, we can write $F$ in \eqref{eq-ch4-LinearODE-F1} as in \eqref{eq-ch4-identity-F-non-int-k-large} 
and obtain \eqref{eq-ch4-regularity-coefficients-F-non-int-k-large} and 
\eqref{eq-ch4-LinearRegularity-F-m1-non-int-k-large}. For future reference, we write
\begin{align}\label{eq-ch4-expression-a-non-int-k-large}\begin{split}
a_{i,j}&\in\mathrm{span}\{c_{p,q},D_{x'}c_{p,q}(p\leq i-1,q\leq p-[\overline{m}]),\\
&\qquad\qquad D^{2}_{x'}c_{p,q}(p\leq i-2,q\leq p-[\overline{m}])\},   
\end{split}\end{align}
for $i=[\overline{m}]+1,\cdots,k$ and $j=0,\cdots,i-[\overline{m}].$ 
We emphasize that $a_{i,j}$ is determined only by $c_{p,q}$ and its derivatives for $p\leq i-1$ and $q\leq p-[\overline{m}]$.

By a simple computation, for a constant $a>0$ and a positive integer $p$, we have, for any $0<t<1$,
$$
\int^{t}_{0}s^{a-1}(\log s)^{p}ds=\sum^{p}_{j=0}\frac{(-1)^{j}}{a^{j+1}}\frac{p!}{(p-j)!}t^{a}(\log t)^{p-j}.
$$
Next, by comparing \eqref{eq-ch4-identity-F-non-int-k-large} 
with \eqref{eq-ch4-identity-F-non-int-overline-m-large}, we have
$$
S_{[\overline{m}]}=\sum_{i=[\overline{m}]+1}^{k} a_{i} t^{i}
+\sum^{k}_{i=[\overline{m}]+1}\sum^{i-[\overline{m}]}_{j=0} a_{i,j} t^{i+\gamma}(\log t)^{j}+S_{k}.
$$
Substituting such an expression in \eqref{eq-ch4-expression-remainder-m-int-non-large} 
and combining with \eqref{eq-ch4-indentity-u-int-overline-m-non-large}, we have
\begin{align}\label{eq-ch4-identity-u-k-intermediate-non-int-overline-m-large}
u=\sum_{i=0}^{k} c_{i} t^{i}+\sum^{k}_{i=[\overline{m}]}\sum^{i-[\overline{m}]}_{j=0} 
c_{i,j} t^{i+\gamma}(\log t)^{j}+R_{k},
\end{align}
where, for $i=[\overline{m}]+1,\cdots,k,$ and $j=0,\cdots,i-[\overline{m}]$,
\begin{align}\label{eq-ch4-expression-coefficient-k-k1-non-int-k-large}\begin{split} 
c_{i} &=\frac{a_{i}}{(i-\underline{m})(i-\overline{m})}, \\
c_{i,j}&\in\mathrm{span}\{a_{i,h} \,(j\leq h\leq i-[\overline{m}])\},
\end{split}\end{align}
and
\begin{align}\label{eq-ch4-expression-R-k-k1-non-int-k-large}
R_{k}=\frac{t^{\overline{m}}}{\overline{m}-\underline{m}} 
\int_{0}^{t} s^{-1-\overline{m}} S_{k} d s-\frac{t^{\underline{m}}}{\overline{m}-\underline{m}} 
\int_{0}^{t} s^{-1-\underline{m}} S_{k} d s.
\end{align}
By \eqref{eq-ch4-regularity-coefficients-F-non-int-k-large} for $i=k$ and $j=0,\cdots,k-[\overline{m}]$, 
we obtain \eqref{eq-ch4-regularity-coefficients-non-int-large} for $i=k$ and $j=0,\cdots,k-[\overline{m}]$. 
Next, by \eqref{eq-ch4-expression-R-k-k1-non-int-k-large}, we write
\begin{equation*}%\label{eq-ch4-expression-R-k-large}
R_{k}=\underline{R}_{k}+\overline{R}_{k}, 
\end{equation*}
where
\begin{align*}
\underline{R}_{k}
&=-\frac{t^{\underline{m}}}{\overline{m}-\underline{m}} \int_{0}^{t} s^{-1-\underline{m}} S_{k}d s,\\
\overline{R}_{k}&=\frac{t^{\overline{m}}}{\overline{m}-\underline{m}} 
\int_{0}^{t} s^{-1-\overline{m}} S_{k} d s.
\end{align*}
A simple computation yields
\begin{equation*}%\label{eq-ch4-expression-R-k-derivative-large}
t\partial_{t}R_{k}=\underline{m}\, \underline{R}_{k}+\overline{m}\,\overline{R}_{k},
\end{equation*}
and
\begin{equation*}%\label{eq-ch4-expression-R-k-double-derivative-large}
t^{2}\partial^{2}_{t}R_{k}=S_{k}+\underline{m}(\underline{m}-1) \underline{R}_{k}
+\overline{m}(\overline{m}-1)\overline{R}_{k}.
\end{equation*}
Note that $k\geq [\overline{m}]+1$. By \eqref{eq-ch4-LinearRegularity-F-m1-non-int-k-large}, 
Lemma \ref{lemma-ch4-BasicHolderRegularity} with $a=-\underline{m}$, 
and Lemma \ref{lemma-ch4-singular-integral-non-int-higher} with $a=\overline{m}$, 
we have, for any $\nu \leq k$ and $\tau \leq \ell-k$, 
$$
\partial_{t}^{\nu} D_{x^{\prime}}^{\tau} \underline{R}_{k},
\partial_{t}^{\nu} D_{x^{\prime}}^{\tau} \overline{R}_{k}\in C^{\epsilon,\alpha}_{x',t}(\bar{G}_{r}),
$$
and 
$$
\big|\partial_{t}^{\nu} D_{x^{\prime}}^{\tau} \underline{R}_{k}\big|
+\big|\partial_{t}^{\nu} D_{x^{\prime}}^{\tau} \overline{R}_{k}\big|
\leq C t^{k-\nu+\alpha} \quad\text {in } G_{1 / 2}.
$$
As a consequence, we obtain, for any $\nu \leq k$ and $\tau \leq \ell-k$,
$$
\partial_{t}^{\nu} D_{x^{\prime}}^{\tau} R_{k}, 
t \partial_{t} \partial_{t}^{\nu} D_{x^{\prime}}^{\tau} R_{k}, 
t^{2} \partial_{t}^{2} \partial_{t}^{\nu} D_{x^{\prime}}^{\tau} R_{k} 
\in C^{\epsilon,\alpha}_{x',t}(\bar{G}_{r}),
$$
and
$$
\big|\partial_{t}^{\nu} D_{x^{\prime}}^{\tau} R_{k}\big|
+\big|t \partial_{t} \partial_{t}^{\nu} D_{x^{\prime}}^{\tau} R_{k}\big|
+\big|t^{2} \partial_{t}^{2} \partial_{t}^{\nu} D_{x^{\prime}}^{\tau} R_{k}\big| 
\leq C t^{k-\nu+\alpha} \quad \text {in } G_{1 / 2}.
$$
By \eqref{eq-ch4-expression-R-k-k1-non-int-k-large} again, we write
$$
t^{-\gamma}R_{k}=t^{-\gamma}\underline{R}_{k}+t^{-\gamma}\overline{R}_{k},
$$
where
\begin{align*}
t^{-\gamma}\underline{R}_{k}&=-\frac{t^{\underline{m}-\gamma}}{\overline{m}-\underline{m}} 
\int_{0}^{t} s^{-1-(\underline{m}-\gamma)} s^{-\gamma}S_{k} d s,\\
t^{-\gamma}\overline{R}_{k}&=\frac{t^{\overline{m}-\gamma}}{\overline{m}-\underline{m}} 
\int_{0}^{t} s^{-1-(\overline{m}-\gamma)} s^{-\gamma}S_{k} d s.
\end{align*}
A similar computation yields
\begin{equation*}%\label{eq-ch4-expression-R-k-derivative-large}
t\partial_{t}(t^{-\gamma}R_{k})=(\underline{m}-\gamma)t^{-\gamma}\underline{R}_{k}
+(\overline{m}-\gamma)t^{-\gamma}\overline{R}_{k},
\end{equation*}
and
\begin{align*}%\label{eq-ch4-expression-R-k-double-derivative-large}
t^{2}\partial^{2}_{t}(t^{-\gamma}R_{k})&=t^{-\gamma}S_{k}
+(\underline{m}-\gamma)(\underline{m}-\gamma-1) t^{-\gamma}\underline{R}_{k}\\
&\qquad+(\overline{m}-\gamma)(\overline{m}-\gamma-1)t^{-\gamma}\overline{R}_{k}.
\end{align*}
Note that $\underline{m}-\gamma<0$ and $\overline{m}-\gamma=[\overline{m}]$ 
by the definition of $\gamma$. By \eqref{eq-ch4-LinearRegularity-F-m1-non-int-k-large}, 
%\eqref{eq-ch4-expression-R-k-large}, \eqref{eq-ch4-expression-R-k-derivative-large}, 
%\eqref{eq-ch4-expression-R-k-double-derivative-large}, 
and Lemma \ref{lemma-ch4-BasicHolderRegularity} 
with $a=\gamma-\underline{m}$ 
and Lemma \ref{lemma-ch4-singular-integral-int-higher} %\ref{lemma-ch4-singular-integral-int-lower}
with $a=[\overline{m}]$, 
we have, for any $\nu \leq k$ and $\tau \leq \ell-k$, 
$$
\partial_{t}^{\nu} D_{x^{\prime}}^{\tau}(t^{-\gamma} \underline{R}_{k}),
\partial_{t}^{\nu} D_{x^{\prime}}^{\tau}(t^{-\gamma} \overline{R}_{k})
\in C^{\epsilon}(\bar{G}_{r}).
$$
As a consequence, we obtain, for any $\nu \leq k$ and $\tau \leq \ell-k$,
$$
\partial_{t}^{\nu} D_{x^{\prime}}^{\tau} (t^{-\gamma}R_{k}), 
t \partial_{t} \partial_{t}^{\nu} D_{x^{\prime}}^{\tau} (t^{-\gamma}R_{k}), 
t^{2} \partial_{t}^{2} \partial_{t}^{\nu} D_{x^{\prime}}^{\tau}(t^{-\gamma} R_{k}) 
\in C^{\epsilon}(\bar{G}_{r}).
$$
Therefore, we conclude 
\eqref{eq-ch4-regularity-coefficients-non-int-large}-\eqref{eq-ch4-LinearRegularity-non-int-m2-large} 
for the chosen $k$. This finishes the proof by induction.

{\it Step 2. We prove \eqref{eq-ch4-LinearRegularity-solution-non-int-m2-large}.}
We now assume $c_{[\overline{m}], 0}=0$ on $B'_1$.  
%Last, if  $c_{[\overline{m}], 1}=0$, we obtain 
Then, $a_{i,j}=0$ and $c_{i,j}=0$ by \eqref{eq-ch4-expression-a-non-int-k-large} and 
\eqref{eq-ch4-expression-coefficient-k-k1-non-int-k-large} inductively for 
$i=[\overline{m}]+1, \cdots, k$ and $j=0, \cdots, i-[\overline{m}]$. 
Next, we prove, 
for any nonnegative integers $[\overline{m}]\le\nu\le k$ 
and $\tau+\nu\le \ell$,  
\begin{align}\label{eq-ch4-LinearRegularity-solution-non-int-m2-large}\partial_t^\nu D_{x'}^\tau u, 
t\partial_t\partial_t^\nu D_{x'}^\tau u, 
t^2\partial_t^2\partial_t^\nu D_{x'}^\tau u
\in %{\color{red} C^{\epsilon}(\bar{G}_r)},  
C^{\alpha}(\bar{G}_r). 
\end{align}

For each $\nu$ with $[\overline{m}]\le \nu\le k$, 
by \eqref{eq-ch4-LinearExpansion-non-int-k-large} with $k$ replaced by $\nu$ 
and $c_{i,j}=0$ for $i=[\overline{m}], \cdots, k$ and $j=0, \cdots, i-[\overline{m}]$, we have 
\begin{align*} %\label{eq-ch4-LinearExpansion-non-int-k-large}
u=\sum_{i=0}^{\nu}c_i t^i
+R_{\nu}\quad\text{in }G_1,\end{align*}
and hence 
$$\partial_t^\nu u=\nu!c_\nu +\partial_t^\nu R_{\nu}\quad\text{in }G_1.$$
By \eqref{eq-ch4-regularity-coefficients-non-int-large} and \eqref{eq-ch4-LinearRegularity-noon-int-m111-large}, we have, 
for any $\tau\le \ell-\nu$,  
$$\partial_t^\nu D_{x'}^\tau u, 
t\partial_t\partial_t^\nu D_{x'}^\tau u, 
t^2\partial_t^2\partial_t^\nu D_{x'}^\tau u
\in  C^{\epsilon, \alpha}_{x',t}(\bar{G}_r).$$
Next, we improve the regularity in $x'$. 

We take a $\beta\in (0,\gamma)$. 
By Theorem \ref{thrm-ch3-Linear-NormalEstimate-general-beta}, we have, 
for any $\nu$ and $\tau$ with $\nu\le [\overline{m}]$ and 
$\nu+\tau\le \ell$, 
\begin{equation*}%\label{eq-regularity-normal-k}
\partial_t^\nu D^\tau_{x'}u, tD\partial_t^\nu D^\tau_{x'}u, t^2D^2\partial_t^\nu D^\tau_{x'}u
\in C^{\alpha,\beta}_{x',t}(\bar G_r).\end{equation*}
Hence, \eqref{eq-ch4-LinearRegularity-solution-non-int-m2-large} holds for 
$\nu= [\overline{m}]$ and any $\tau\le \ell-[\overline{m}]$. 

%We now fix an integer $k$ with $[\overline{m}]<k\leq \ell$. 
We assume that \eqref{eq-ch4-LinearRegularity-solution-non-int-m2-large} 
holds for $[\overline{m}]\le \nu\le k-1$ and $\tau+\nu\le\ell$  
and then proceed to prove \eqref{eq-ch4-LinearRegularity-solution-non-int-m2-large} 
for $\nu=k$ and $\tau+\nu\le\ell$. 
The proof is similar as that of the regularity of $R_k$ in Step 1. 
%Lemma \ref{lemma-ch4-Linear-MainThm-non-integer-large}. 

By \eqref{eq-ch4-LinearRegularity-solution-non-int-m2-large} 
for any $\nu\le k-1$ and $\tau+\nu\le \ell$ 
and a similar argument as in the proof of Lemma \ref{lemma-ch4-Linear-MainThm-non-integer-large}, 
we have, for any $\nu \leq k$ and $\tau \leq \ell-\nu$,
\begin{equation}\label{i}
\partial_{t}^{\nu} D_{x^{\prime}}^{\tau} F \in C^{\alpha}(\bar{G}_{r}).
\end{equation}
Then, we can write 
\begin{equation}\label{eq-equality-F-a-non-large-full-Holder}
F=\sum_{i=0}^{k} a_it^i
+S_{k},\end{equation}
where $a_i$ are functions on $B_1'$ 
and $S_k$ is a function in $G_1$  
such that   
\begin{equation}\label{eq-regularity-a-non-large-full-Holder}
a_i\in C^{\ell-i, \alpha}(B_1')\quad\text{for any }i=0, \cdots, k,\end{equation}
and,
for any $\nu\le k$ and $\tau\le \ell-k$,
\begin{align}\label{eq-Regularity-S-non-large-full-Holder}\begin{split}
\partial_t^\nu D_{x'}^\tau S_{k}&\in C^\alpha(\bar G_r),\\
|\partial_t^\nu D_{x'}^\tau S_{k}|&\le Ct^{k-\nu+\alpha}\quad\text{in }\bar G_r.
\end{split}\end{align}
In fact, $a_i$ is given by 
$$a_i=\frac{1}{i!}\partial_t^iF(\cdot,0)\quad\text{for }i=0, \cdots, k.$$
Compare with \eqref{eq-ch4-identity-F-non-int-k-large}, 
\eqref{eq-ch4-regularity-coefficients-F-non-int-k-large}, 
and \eqref{eq-ch4-LinearRegularity-F-m1-non-int-k-large}. 

With $F$ given by \eqref{eq-equality-F-a-non-large-full-Holder}, by proceeding as in Step 1, 
we get $u$ as in \eqref{eq-ch4-identity-u-k-intermediate-non-int-overline-m-large}, 
where $c_i$ is given by \eqref{eq-ch4-expression-coefficient-k-k1-non-int-k-large} 
and $c_{i,j}=0$ for $i=[\overline{m}], \cdots, k$ and $j=0, \cdots, i-[\overline{m}]$, 
and $R_k$ is given by \eqref{eq-ch4-expression-R-k-k1-non-int-k-large}. 
%\eqref{eq-ch4-identity-R-k-intermediate-non-int-overline-m-large}. 
Similarly, by \eqref{eq-regularity-a-non-large-full-Holder}
and \eqref{eq-Regularity-S-non-large-full-Holder}, we obtain 
\begin{equation*}%\label{m'}
c_i\in C^{\ell-i,\alpha}(B'_1)\quad\text{for }i=0, \cdots, k,
\end{equation*}
and, for any $\nu\le k$ and $\tau\le \ell-k$,
\begin{equation*}%\label{n}
\partial_t^\nu D_{x'}^\tau R_{k},
t\partial_t\partial_t^\nu D_{x'}^\tau R_{k},
t^2\partial_t^2\partial_t^\nu D_{x'}^\tau R_{k}\in C^{\alpha}(\bar{G}_r),
\end{equation*}
and
\begin{align*}
|\partial_t^\nu D_{x'}^\tau R_{k}|
+|t\partial_t\partial_t^\nu D_{x'}^\tau R_{k}|
+|t^2\partial_t^2\partial_t^\nu D_{x'}^\tau R_{k}|
\le Ct^{k-\nu+\alpha}
\quad\text{in }G_{1/2}.\end{align*}
Therefore, we conclude \eqref{eq-ch4-LinearRegularity-solution-non-int-m2-large} 
for $\nu=k$ and $\tau+\nu\le\ell$.
\end{proof}

In Theorem \ref{thrm-ch4-Linear-MainThm-non-integer-large}, 
we proved the regularity of $R_{k}$ and $t^{-\gamma}R_k$. 
The regularity of $R_k$ is already used in the same theorem 
for the regularity of $u$ when $c_{[\overline{m}], 0}$ is identically zero. 
We will use the regularity of $t^{-\gamma}R_k$ in the proof of 
Theorem \ref{Thm-ch4-Linear-MainThm-decomposition-non-int-large} below.  
We need to point out that we proved the regularity of $R_{k}$ and $t^{-\gamma}R_k$ parallelly.  
In this way, both $R_{k}$ and $t^{-\gamma}R_k$ have the same regularity in $x'$. 
If we derive the regularity of $t^{-\gamma}R_k$ from that of $R_{k}$ by 
Lemma \ref{lemma-ch4-regularity-power1}, 
we will further lose a regularity of a H\"older index $\gamma$. 
Such a loss will accumulate in the induction process. 

The expansion in \eqref{eq-ch4-LinearExpansion-non-int-k-large} with $k=\ell$ 
has the following pattern: 
\begin{align*}&t^0, t^1, \cdots, t^{[\overline{m}]},t^{[\overline{m}]+\gamma},
t^{[\overline{m}]+1},t^{[\overline{m}]+1+\gamma}\log t,t^{[\overline{m}]+1+\gamma},\cdots,\\ 
&\qquad t^\ell, t^{\ell+\gamma}(\log t)^{\ell-[\overline{m}]}, \cdots, 
t^{\ell+\gamma}\log t, t^{\ell+\gamma}.\end{align*}
The coefficients $c_0, \cdots, c_{\ell}$ have an optimal regularity 
determined by the regularity of $a_{ij}, b_i, c$, and $f$, 
and there is a loss of regularity for $c_{i,j}$
for $[\overline{m}]\le i\le \ell$ and $0\le j\le i-[\overline{m}]$. 
%with $\ell_*=\ell$ in  \eqref{eq-ch4-LinearExpansion-non-int-k-large}
%and $\ell_*=\ell-1$ in  \eqref{eq-ch4-LinearExpansion-non-int-k-small}. 
The first term with non-integer power is given by $t^{[\overline{m}]+\gamma}$. 
If the first non-integer power does not appear, 
then there are no non-integer power or logarithmic terms in the expansion and hence 
the solutions are as regular as the 
coefficients and nonhomogeneous terms allow.

%\begin{remark}
%If we use Lemma 4.6 and (2.20) to obtain the regularity of $t^{-\gamma}R_{k}$, 
%we will not obtain the optimal regularity of $w$.
%\end{remark}

\begin{theorem}\label{Thm-ch4-Linear-MainThm-decomposition-non-int-large}
Suppose that $\underline{m}$, $\overline{m}$, and $\gamma$ are functions on $B'_{1}$ satisfying 
\eqref{eq-ch4-Assumption_m1}, \eqref{eq-ch4-Assumption_m2}, 
and \eqref{eq-definition-gamma}, 
and that $\ell$ is an integer and $\alpha\in(0,1)$ is a constant such that $\ell\geq [\overline{m}] \ge 1$ 
and \eqref{eq-large-Holder-index} holds. 
Assume $a_{n n}, b_{n}, c \in C^{\ell+1, \alpha}(\bar{G}_{1})$ 
and $a_{i j}, b_{i}, f \in C^{\ell, \alpha}(\bar{G}_{1})$ for $i \neq n$, 
and let $u \in C(\bar{G}_{1}) \cap C^{2}(G_{1})$ be a solution 
of \eqref{eq-ch3-Equ} and \eqref{eq-ch3-Dirichlet}. %, with \eqref{eq-ch3-boundary-value}. 
Then,
\begin{align}\label{eq-ch4-LinearExpansion-decomposition-non-large}
u=v+\sum^{\ell-[\overline{m}]}_{j=0}w^{(j)}t^{\gamma}(\log t)^{j} \quad\text{in } G_{1},
\end{align}
where $v$ and $w^{(0)},\cdots,w^{(\ell-[\overline{m}])}$ are functions in $G_{1}$ satisfying, 
for any $x'_{0}\in B'_{1}$, 
any $r\in(0,1-|x'_{0}|)$, 
and any $\epsilon$ with $0<\epsilon <\alpha-\gamma$ on $\bar B_{r}'(x_0')$, 
\begin{align*}
&v \in C^{\ell, \alpha}(\bar{G}_{r}(x_{0}')),\\
&w^{(j)} \in C^{\ell, \epsilon}(\bar{G}_{r}(x_{0}'))\quad\text{for }j=0,\cdots,\ell-[\overline{m}],
\end{align*}
with $\partial_t^i w^{(j)}=0$ on $\Sigma_{1}$, for $j=0,\cdots,\ell-[\overline{m}]$ 
and $i=0,\cdots,[\overline{m}]+j-1$, and 
$$
|v|_{C^{\ell, \alpha}(\bar{G}_{r}(x_{0}'))}
+\sum^{\ell-[\overline{m}]}_{j=0}|w^{(j)}|_{C^{\ell, \epsilon}(\bar{G}_{r}(x_{0}'))} 
\leq C\big\{|u|_{L^{\infty}(G_1)}+|f|_{C^{\ell, \alpha}(\bar{G}_1)}\big\},
$$
for some positive constant $C$ depending only on $n$, $\lambda$, $\ell$, $\alpha$, $\epsilon$, $r$
$\underline{m}$, $\overline{m}$, the $C^{\ell+1, \alpha}$-norms of $a_{n n}, b_n, c$ in $\bar{G}_1$, 
and the $C^{\ell, \alpha}$-norms of $a_{i j}, b_i$ in $\bar{G}_1$ for $i \neq n$.
If, in addition, $\partial_t^{[\overline{m}]} w^{(0)}=0$ on $\Sigma_r(x_0')$, 
for some $x'_{0}\in B'_{1}$ and $r\in(0,1-|x'_{0}|)$, 
then $u \in C^{\ell,\alpha}(\bar{G}_{r}(x_{0}'))$. 
%for any $\epsilon$ with $0<\epsilon <\alpha-\gamma$ on $\bar B_{r}'(x_0')$.
\end{theorem}

\begin{proof}
We will apply Theorem \ref{thrm-ch4-Linear-MainThm-non-integer-large} 
and Lemma \ref{lemma-ch4-extensions} to construct functions 
$v$ and $w^{(j)}$, for $j=0, \cdots, \ell-[\overline{m}]-1$, 
and prove that there is a function $w^{(\ell-[\overline{m}])}$ with the desired regularity 
such that \eqref{eq-ch4-LinearExpansion-decomposition-non-large} holds.

We adopt notations in the proof of Theorem \ref{thrm-ch4-Linear-MainThm-non-integer-large}. 
Without loss of generality, we assume $x'_{0}=0$ %, $r_0=1$, 
and take a constant $\epsilon_0$ such that 
$0<\epsilon_0 \le\alpha-\gamma$ on $\bar B_{1}'$. 
Throughout the proof, $\epsilon$ is an arbitrary constant in $(0,\epsilon_{0})$ 
and $r$ is an arbitrary constant in $(0,1)$. 

By Theorem \ref{thrm-ch4-Linear-MainThm-non-integer-large}, 
there exist functions $c_{i} \in C^{\ell-i, \alpha}(B'_{1})$, for $0 \leq i \leq$ $\ell$, 
and $c_{i, j} \in C^{\ell-i, \epsilon}(B'_{1})$, 
for $[\overline{m}] \leq i \leq \ell$ and $0\leq j\leq i-[\overline{m}]$, such that
$$
u=\sum_{i=0}^{\ell} c_{i} t^{i}+\sum_{i=[\overline{m}]}^{\ell}\sum^{i-[\overline{m}]}_{j=0} 
c_{i, j} t^{i+\gamma}(\log t)^{j}+R_{\ell},
$$
where $R_{\ell}$ is a function in $G_{1}$ with the stated properties in 
Theorem \ref{thrm-ch4-Linear-MainThm-non-integer-large} with $k=\ell$. We write 
$$
u=\sum_{i=0}^{\ell} c_{i} t^{i}+\sum_{j=0}^{\ell-[\overline{m}]}\sum^{\ell}_{i=[\overline{m}]+j} 
c_{i, j} t^{i}t^{\gamma}(\log t)^{j}+R_{\ell}.
$$
By Lemma \ref{lemma-ch4-extensions}, there exist functions $v \in C^{\ell, \alpha}(\bar{G}_{r})$ 
and $w^{(j)}\in  C^{\ell, \epsilon}(\bar{G}_{r}) $ for $j=0,\cdots,\ell-[\overline{m}]-1$, such that
\begin{equation}\label{eq-ch4-choice-v-large}
\partial_{t}^{i} v|_{t=0}=i ! c_{i}\quad\text {for } i=0, \cdots, \ell,
\end{equation}
and
\begin{equation}\label{eq-ch4-choice-w-j-large}
\partial_{t}^{i} w^{(j)}|_{t=0}=
\begin{cases}
0\quad&\text{for }i=0,\cdots,[\overline{m}]+j-1,\\
i ! c_{i,j}\quad&\text {for } i=[\overline{m}]+j, \cdots, \ell. 
\end{cases}    
\end{equation}
Introduce a function $w=w^{(\ell-[\overline{m}])}$ in $G_{1}$ such that
\begin{equation}\label{eq-identity-w-v-non-int-large}
u=v+\sum^{\ell-[\overline{m}]}_{j=0}w^{(j)}t^{\gamma}(\log t)^{j}.    
\end{equation}
We will prove $w \in C^{\ell, \epsilon}(\bar{G}_{r})$.

We first consider a nonnegative integer $k$ with $k\leq [\overline{m}]-1$. 
We will prove $\partial^{\nu}_{t}D^{\tau}_{x'}w\in C^{\epsilon}(\bar{G}_{r})$ 
for any nonnegative integers $\nu\leq k$ and $\tau\leq \ell-k$. 
By Theorem \ref{thrm-ch3-Linear-NormalEstimate-general} and $k+\alpha<\overline{m}$, 
we get $\partial^{\nu}_{t}D^{\tau}_{x'}u\in C^{\alpha}(\bar{G}_{r})$ 
for any $\nu\leq k$ and $\tau\leq \ell-\nu$. 
By $v\in C^{\ell,\alpha}(\bar{G}_{r}),$ \eqref{eq-ch4-choice-v-large}, 
Lemma \ref{lemma-ch4-regularity-log}(ii), and Lemma \ref{lemma-ch4-regularity-power1}, 
we have, for any $\nu\leq k$ and $\tau\leq \ell-k$,
$$
\partial^{\nu}_{t}D^{\tau}_{x'}\big[t^{-\gamma}(\log t)^{-(\ell-[\overline{m}])}(u-v)\big]\in C^{\epsilon}(\bar{G}_{r}).
$$
Similarly, by $w^{(j)}\in  C^{\ell, \epsilon}(\bar{G}_{r}) $ for $j=0,\cdots,\ell-[\overline{m}]-1$, 
\eqref{eq-ch4-choice-w-j-large}, and Lemma \ref{lemma-ch4-regularity-log}(ii), we have, 
for any $\nu\leq k$ and $\tau\leq \ell-k$, and $j=0,\cdots,\ell-[\overline{m}]-1$,
$$
\partial^{\nu}_{t}D^{\tau}_{x'}\big[(\log t)^{j-(\ell-[\overline{m}])}w^{(j)}\big]\in C^{\epsilon}(\bar{G}_{r}).
$$
By \eqref{eq-identity-w-v-non-int-large}, we obtain
$$
w=t^{-\gamma}(\log t)^{-(\ell-[\overline{m}])}(u-v)-\sum^{\ell-[\overline{m}]-1}_{j=0}(\log t)^{j-(\ell-[\overline{m}])}w^{(j)},
$$
and hence $\partial^{\nu}_{t}D^{\tau}_{x'}w\in C^{\epsilon}(\bar{G}_{r})$ for any $\nu\leq k$ and $\tau\leq \ell-k$.

We next fix an integer $k$ with $[\overline{m}] \leq k \leq \ell$. 
We will prove $\partial_{t}^{\nu} D_{x'}^{\tau} w \in C^{\epsilon}(\bar{G}_{r})$
 for any nonnegative integers $\nu\leq k$ and $\tau\leq \ell-k$. 
By Theorem \ref{thrm-ch4-Linear-MainThm-non-integer-large}, we have
\begin{align}\label{eq-ch4-definition-v-k-large}\begin{split}
u&=\sum_{i=0}^{k} c_{i} t^{i}+\sum_{i=[\overline{m}]}^{k}\sum^{i-[\overline{m}]}_{j=0} 
c_{i, j} t^{i+\gamma}(\log t)^{j}+R_{k}\\
&=\sum_{i=0}^{k} c_{i} t^{i}+\sum_{j=0}^{k-[\overline{m}]}\sum^{k}_{i=[\overline{m}]+j} 
c_{i, j} t^{i+\gamma}(\log t)^{j}+R_{k},
\end{split}\end{align}
where $R_{k}$ is a function in $G_{1}$ with the stated properties in 
Theorem \ref{thrm-ch4-Linear-MainThm-non-integer-large}. 
By \eqref{eq-ch4-choice-v-large}, we can write
\begin{equation}\label{eq-ch4-definition-wj-k-large}
v=\sum_{i=0}^{k} c_{i} t^{i}+S_{k}(v), 
\end{equation}
where $S_{k}(v)$ is a function in $\bar{G}_{r}$ such that, for any $\nu \leq k$ and $ \tau \leq \ell-k$,
\begin{equation}\label{eq-ch4-property-wj-k1-large}
\partial_{t}^{\nu}D_{x'}^{\tau}  S_{k}(v) \in C^{\alpha}(\bar{G}_{r}).
\end{equation}
Similarly, for $j=0,\cdots,\ell-[\overline{m}]-1$, we can write
\begin{equation}\label{eq-ch4-property-wj-k2-large}
w^{(j)}=\sum_{i=[\overline{m}]+j}^{k} c_{i,j} t^{i}+S_{k}(w^{(j)}),
\end{equation}
where $S_{k}(w^{(j)})$ is a function in $\bar{G}_{r}$ such that, for any $\nu \leq k$ and $ \tau \leq \ell-k$,
\begin{equation}\label{eq-ch4-property-wj-k3-large}
\partial_{t}^{\nu}D_{x'}^{\tau}  S_{k}(w^{(j)}) \in C^{\epsilon}(\bar{G}_{r}).
\end{equation}
Note that $w^{(j)}=S_{k}(w^{(j)})$ for $j=k-[\overline{m}]+1,\cdots,\ell-[\overline{m}]-1$. 
By substituting \eqref{eq-ch4-definition-v-k-large}, 
\eqref{eq-ch4-definition-wj-k-large}, and \eqref{eq-ch4-property-wj-k2-large}
 in \eqref{eq-identity-w-v-non-int-large}, we get, for $[\overline{m}]\leq k\leq \ell-1$, 
$$
w=(R_{k}-S_{k}(v))t^{-\gamma}(\log t)^{-(\ell-[\overline{m}])}
-\sum^{\ell-[\overline{m}]-1}_{j=0}S_{k}(w^{(j)})(\log t)^{j-(\ell-[\overline{m}])},
$$
and, for $k=\ell$,
$$
w=c_{\ell,\ell-[\overline{m}]}t^{\ell}+(R_{\ell}-S_{\ell}(v))t^{-\gamma}(\log t)^{-(\ell-[\overline{m}])}
-\sum^{\ell-[\overline{m}]-1}_{j=0}S_{\ell}(w^{(j)})(\log t)^{j-(\ell-[\overline{m}])}.
$$
Note that $\partial^i_t w=0$  on $\Sigma_1$ for $i=0,\cdots,\ell-1$. 
Take any $\nu \leq k$ and $\tau \leq \ell-k$. 
By \eqref{eq-ch4-property-wj-k1-large}, \eqref{eq-ch4-property-wj-k3-large}, 
Lemma \ref{lemma-ch4-regularity-log}(ii), and Lemma \ref{lemma-ch4-regularity-power1}, we have, 
$$
\partial_{t}^{\nu}D_{x'}^{\tau}\big[S_{k}(v)t^{-\gamma}(\log t)^{-(\ell-[\overline{m}])}\big]
\in C^{\epsilon}(\bar{G}_{r}),
$$
and for $j=0,\cdots,\ell-[\overline{m}]-1$,
$$
\partial_{t}^{\nu}D_{x'}^{\tau}\big[S_{k}(w^{(j)})(\log t)^{j-(\ell-[\overline{m}])}\big]\in C^{\epsilon}(\bar{G}_{r}).
$$
By \eqref{eq-ch4-regularity-coefficients-non-int-large}, we have, with $k=\ell$, 
$$\partial_{t}^{\nu}D_{x'}^{\tau}(c_{\ell,\ell-[\overline{m}]}t^{\ell})\in C^{\epsilon}(\bar{G}_{r}).$$ 
By \eqref{eq-ch4-LinearRegularity-noon-int-m112-large} and Lemma \ref{lemma-ch4-regularity-log}(ii), we have 
$$
\partial_{t}^{\nu}D_{x'}^{\tau}\big[R_{k}t^{-\gamma}(\log t)^{-(\ell-[\overline{m}])}\big]\in C^{\epsilon}(\bar{G}_{r}).
$$
In summary, we obtain  $\partial_{t}^{\nu}D_{x'}^{\tau}w\in C^{\epsilon}(\bar{G}_{r})$. 
This holds for any $\nu\leq k$ and $\tau\leq \ell-k$, 
and hence for any $\nu$ and $\tau$ with $\nu+\tau\leq \ell$. 
Therefore, $w \in C^{\ell, \epsilon}(\bar{G}_{r})$.

If, in addition, $\partial_t^{[\overline{m}]} w^{(0)}=0$ on $\Sigma_{1}$, 
then $c_{[\overline{m}],0}=0$  on $B'_{1}$. 
By Theorem \ref{thrm-ch4-Linear-MainThm-non-integer-large}, 
we have $u \in C^{\ell, \alpha}(\bar{G}_{r})$.
\end{proof}

We point out that the functions
$v$ and $w^{(0)}, \cdots,w^{(\ell-[\overline{m}])}$ in \eqref{eq-ch4-LinearExpansion-decomposition-non-large} are not unique. 
In fact, take any $\eta\in C^\infty(\bar G_1)$ which vanishes up to infinite order on $\Sigma_1$. 
We can write 
\begin{align*}
u=v-\eta t^{\gamma}\sum^{\ell-[\overline{m}]}_{j=0}(\log t)^{j}
+\sum^{\ell-[\overline{m}]}_{j=0}(w^{(j)}+\eta)t^{\gamma}(\log t)^{j}.\end{align*}
In other words, if $v$ and $w^{(0)}, \cdots,w^{(\ell-[\overline{m}])}$ 
satisfy \eqref{eq-ch4-LinearExpansion-decomposition-non-large}, 
then $v-\eta t^{\gamma}\sum^{\ell-[\overline{m}]}_{j=0}(\log t)^{j}$ 
and $w^{(0)}+\eta, \cdots, w^{(\ell-[\overline{m}])}+\eta$ 
also satisfy \eqref{eq-ch4-LinearExpansion-decomposition-non-large}. 
In this sense, $v$ and $w^{(0)}, \cdots,w^{(\ell-[\overline{m}])}$ 
created in the proof above is only one collection of functions satisfying 
\eqref{eq-ch4-LinearExpansion-decomposition-non-large}. 
However, they satisfy additional properties. 
First, they satisfy the estimates already stated 
in Theorem \ref{Thm-ch4-Linear-MainThm-decomposition-non-int-large}. 
Second, they depend on $u$ linearly. 
This follows from the linear dependence of $c_i$,
for $0 \leq i \leq$ $\ell$, and $c_{i,j}$, 
for $[\overline{m}] \leq i \leq \ell$ and $0\leq j\leq i-[\overline{m}]$, on $u$
as in Theorem \ref{thrm-ch4-Linear-MainThm-non-integer-large} 
and the linear dependence of the functions $v$ and 
$w^{(0)}, \cdots,w^{(\ell-[\overline{m}]-1)}$ on $c_i$, for $0 \leq i \leq$ $\ell$, and $c_{i,j}$, 
for $[\overline{m}] \leq i \leq \ell$ and $0\leq j\leq i-[\overline{m}]$, 
according to Lemma \ref{lemma-ch4-extensions}.

We note that Theorem \ref{Thm-ch4-Linear-MainThm-curved-combined} for the case 
$\gamma<\alpha$ 
follows from Theorem \ref{Thm-ch4-Linear-MainThm-decomposition-non-int-large}. 

\begin{theorem}\label{thrm-ch4-Decomposition-infinite-non-int}
Suppose that $\underline{m}$, $\overline{m}$, and $\gamma$ are functions on $B'_{1}$ satisfying 
\eqref{eq-ch4-Assumption_m1}, \eqref{eq-ch4-Assumption_m2}, 
and \eqref{eq-definition-gamma}, with  $\overline{m} > 1$ and $0<\gamma<1$ on $\bar B_1'$. 
Assume $a_{i j}, b_i, c,$ $f\in C^{\infty}(\bar{G}_1)$
and let $u \in C(\bar{G}_{1}) \cap C^{2}(G_{1})$ be a solution 
of \eqref{eq-ch3-Equ} and \eqref{eq-ch3-Dirichlet}. %, with \eqref{eq-ch3-boundary-value}. 
Then,
$$
u=v+\sum^{\infty}_{j=0}w^{(j)}t^{\gamma}(\log t)^{j} \quad\text{absolutely and uniformly in } \bar{G}_{r},
$$
for some functions $v,w^{(0)},w^{(1)},\cdots \in C^{\infty}(\bar{G}_{r})$, for any $r\in(0,1)$, 
with $\partial_t^i w^{(j)}=0$ on $\Sigma_{1}$, for $j=0,1,\cdots$ and $i=0,\cdots, [\overline{m}]+j-1$. 
Moreover, for any integer $k\geq [\overline{m}]$,
\begin{equation}\label{eq-ch4-Decomposition-infinite-higher-non-int}
D^{k}\Big[u-\sum^{k-[\overline{m}]}_{j=0}w^{(j)}t^{\gamma}(\log t)^{j}\Big]
=D^{k}v+\sum^{\infty}_{j=k-[\overline{m}]+1}D^{k}[w^{(j)}t^{\gamma}(\log t)^{j}],
\end{equation}
where the series converges absolutely and uniformly in $\bar{G}_{r}$, for any $r\in(0,1)$.
If, in addition, $\partial_t^{[\overline{m}]} w^{(0)}=0$ on $\Sigma_{1}$, 
then $u \in C^{\infty}(\bar{G}_{r})$, for any $r\in(0,1)$.
\end{theorem}

\begin{proof}
Note that $\overline{m},\underline{m},\gamma\in C^{\infty}(B'_{1})$ 
by \eqref{eq-ch4-Assumption_m1}, \eqref{eq-ch4-Assumption_m2}, 
and \eqref{eq-definition-gamma}. 
Fix a constant $\alpha\in (0,1)$ such that $\alpha>\gamma$ on $\bar B_1'$. 
By applying Theorem \ref{thrm-ch4-Linear-MainThm-non-integer-large} successively 
for each $k\geq [\overline{m}]$ and $\ell=2k$, we obtain sequences of smooth functions 
$\{c_{i}\}^{\infty}_{i=0}$ and $\{c_{i,j}\}_{i\geq [\overline{m}],0\leq j\leq i-[\overline{m}]}$ on $B_{1}'$ 
and a sequence of functions $\{R_{k}\}^{\infty}_{k=[\overline{m}]}$ in $G_{1}$ such that, 
for any $k\geq [\overline{m}]$,
\begin{equation}\label{eq-ch4-expression-any-ell}
u=\sum_{i=0}^{k} c_{i} t^{i}+\sum^{k}_{i=[\overline{m}]}\sum^{i-[\overline{m}]}_{j=0} 
c_{i, j} t^{i+\gamma}(\log t)^{j}+R_{k} \quad \text {in } G_{1},
\end{equation}
and $R_{k}\in C^{k}(\bar{G}_{r})$, for any $r\in(0,1)$. Take a sequence $\{r_{i,j}\}_{i,j\geq 0}
=\{1-2^{-i-j-1}\}_{i,j\geq 0}$ in $(0,1)$ such that $r_{i,j}\rightarrow 1$ as $i+j\rightarrow \infty$. 

Take a cutoff function $\eta\in C^{\infty}_{0}(-1,1)$ with $\eta=1$ on $(-{1}/{2},{1}/{2})$. 
For any constant $\lambda_{i,j}\geq 1$, we observe that 
$$
\eta(\lambda_{i,j}t)=
\begin{cases}
\begin{aligned}
0&\quad\text{if }|t|\geq 1/\lambda_{i,j},\\
1&\quad\text{if }|t|\leq 1/(2\lambda_{i,j}).
\end{aligned}
\end{cases}
$$
It is straightforward to verify that, for any $j\geq 0$, $i\geq [\overline{m}]+j$, 
and $k<i$,
$$
\big|D^k_{(x',t)}[c_{i,j}\eta(\lambda_{i,j}t)t^{i}]\big|\leq 
\frac{C_{i,j,k}}{\lambda_{i,j}^{i-k}}\quad\text{in }G_{r_{i,j}},
$$
for some positive constant $C_{i,j,k}$ depending only on $i$, $j$, $k$,  and $c_{i,j}$. Note that $\overline{m}>1$.
By choosing $\lambda_{i,j}$ sufficiently large, we get, for any $j\geq 0$ and $i\geq [\overline{m}]+j$, 
\begin{equation}\label{eq-ch4-decay-w}
|c_{i,j}\eta(\lambda_{i,j}t)t^{i}|_{C^{i-1}(\bar{G}_{r_{i,j}})}\leq 2^{-i}.
\end{equation}
%The only exceptional term comes from $i=j=0$ when $[\overline{m}]=0$ and we take $\lambda_{0,0}=1$. 
For each $j\geq 0$, set 
\begin{equation}\label{eq-ch4-def-w-int}
w^{(j)}=\sum^{\infty}_{i=[\overline{m}]+j}c_{i,j}\eta(\lambda_{i,j}t)t^{i}\quad\text{in }G_{1}.
\end{equation}
We note that each term in the summation is smooth in $G_{1}$. Fix a $j\geq 0$. 
We claim that, for any $k\geq [\overline{m}]+j$ and any $r\in(0,1)$,
\begin{equation}\label{eq-ch4-summation-Ck-finite}
\sum^{\infty}_{i=k+1}|c_{i,j}\eta(\lambda_{i,j}t)t^{i}|_{C^{i-1}(\bar G_{r})}\leq C_{r,j,k},
\end{equation}
where $C_{r,j,k}$ is a positive constant depending on $r$, $j$, and $k$. If $r_{i,j}\geq r$ for any $i\geq k+1$, 
by \eqref{eq-ch4-decay-w}, we have 
$$
\sum^{\infty}_{i=k+1}|c_{i,j}\eta(\lambda_{i,j}t)t^{i}|_{C^{i-1}(\bar{G}_{r})}\leq \sum^{\infty}_{i=k+1}2^{-i}=2^{-k}.
$$
Note that $r_{i,j}<r$ can hold only for finitely many $i$'s. 
Hence, in the estimate above, we only need to adjust finitely many terms for a general $r$. 
This proves \eqref{eq-ch4-summation-Ck-finite}. 
As a consequence, we obtain $w^{(j)}\in C^{k}(\bar{G}_{r})$,
for any $r\in (0,1)$ and any $k\geq [\overline{m}]+j$. 
Therefore, $w^{(j)}\in C^{\infty}(\bar{G}_{r})$, for any $r\in (0,1)$. 

Similarly, by choosing $\lambda_{i,j}$ larger if necessary, we have, 
for each $j\geq 0$ and $i\geq [\overline{m}]+j$,
$$
|c_{i,j}\eta(\lambda_{i,j}t)t^{i+\gamma}(\log t)^{j}|_{C^{i-1}(\bar{G}_{r_{i,j}})}\leq 2^{-i}.
$$
Hence, for any $j\geq 0$ and $k\geq [\overline{m}]+j$, and any $r\in(0,1)$,
\begin{equation}\label{eq-ch4-summation-regularity1-finite}
\sum^{\infty}_{i=k+1}c_{i,j}\eta(\lambda_{i,j}t)t^{i+\gamma}(\log t)^{j}\in C^{k}(\bar{G}_{r}), 
\end{equation}
and, for any $k\geq [\overline{m}]$ and any $r\in(0,1)$,
\begin{equation}\label{eq-ch4-summation-regularity2-finite}
\sum^{\infty}_{j=k-[\overline{m}]+1}
\sum^{\infty}_{i=[\overline{m}]+j}c_{i,j}\eta(\lambda_{i,j}t)t^{i+\gamma}(\log t)^{j}\in C^{k}(\bar{G}_{r}).
\end{equation}
In fact, to prove \eqref{eq-ch4-summation-regularity2-finite}, 
if $r_{i,j}\geq r$ for any $j\geq k-[\overline{m}]+1$ and $i\geq[\overline{m}]+j$, we get
\begin{align*}
&\sum^{\infty}_{j=k-[\overline{m}]+1}\sum^{\infty}_{i=[\overline{m}]+j}
|c_{i,j}\eta(\lambda_{i,j}t)t^{i+\gamma}(\log t)^{j}|_{C^{k}(\bar{G}_{r})}\\
&\qquad\leq \sum^{\infty}_{j=k-[\overline{m}]+1}\sum^{\infty}_{i=[\overline{m}]+j}
|c_{i,j}\eta(\lambda_{i,j}t)t^{i+\gamma}(\log t)^{j}|_{C^{i-1}(\bar{G}_{r_{i,j}})}\\
&\qquad\leq \sum^{\infty}_{j=k-[\overline{m}]+1}\sum^{\infty}_{i=[\overline{m}]+j}2^{-i}
=\sum^{\infty}_{j=k-[\overline{m}]+1}2^{-[\overline{m}]+1-j}=2^{-k+1}.
\end{align*}
Note that $r_{i,j}<r$ can hold only for finitely many $i$'s and $j$'s.

Now, take an arbitrarily fixed $k\geq [\overline{m}]$ and $r\in (0,1)$. 
By \eqref{eq-ch4-expression-any-ell}, we write
\begin{align*}
u&=\sum_{i=0}^{k} c_{i} t^{i}+\sum^{k-[\overline{m}]}_{j=0}\sum^{k}_{i=[\overline{m}]+j} 
c_{i, j} t^{i+\gamma}(\log t)^{j}+R_{k}\\
&=\sum_{i=0}^{k} c_{i} t^{i}+\sum^{k-[\overline{m}]}_{j=0}\sum^{k}_{i=[\overline{m}]+j} 
c_{i, j}\eta(\lambda_{i,j}t) t^{i+\gamma}(\log t)^{j}\\
&\qquad+\sum^{k-[\overline{m}]}_{j=0}\sum^{k}_{i=[\overline{m}]+j} c_{i, j}[1-\eta(\lambda_{i,j}t)] 
t^{i+\gamma}(\log t)^{j}+R_{k}.
\end{align*}
By combining with \eqref{eq-ch4-def-w-int}, we have 
$$
\begin{aligned}
u-\sum^{\infty}_{j=0}w^{(j)}t^{\gamma}(\log t)^{j}&=\sum_{i=0}^{k} c_{i} t^{i}
-\sum^{k-[\overline{m}]}_{j=0}\sum^{\infty}_{i=k+1} c_{i, j}\eta(\lambda_{i,j}t) t^{i+\gamma}(\log t)^{j}\\
&\qquad-\sum^{\infty}_{j=k-[\overline{m}]+1}\sum^{\infty}_{i=[\overline{m}]+j} 
c_{i, j}\eta(\lambda_{i,j}t) t^{i+\gamma}(\log t)^{j}\\
&\qquad+\sum^{k-[\overline{m}]}_{j=0}\sum^{k}_{i=[\overline{m}]+j} c_{i, j}[1-\eta(\lambda_{i,j}t)] 
t^{i+\gamma}(\log t)^{j}+R_{k}.
\end{aligned}
$$
In the first summation, each term is smooth in $\bar{G}_{r}$ by $c_{i}\in C^{\infty}(B_{1}')$. 
The entire second and third summations are functions in $C^{k}(\bar{G}_{r})$ by
\eqref{eq-ch4-summation-regularity1-finite} and \eqref{eq-ch4-summation-regularity2-finite}. 
In the fourth summation, there are only finitely many terms and each term is $0$ for small $t$ 
and hence is smooth in $\bar{G}_{r}$. By $R_{k}\in C^{k}(\bar{G}_{r})$, 
we conclude 
$$u-\sum^{\infty}_{j=0}w^{(j)}t^{\gamma}(\log t)^{j}\in C^{k}(\bar{G}_{r}).$$ 
This holds for any $k\geq [\overline{m}]$. Thus,  for any $r\in (0,1)$, 
$$u-\sum^{\infty}_{j=0}w^{(j)}t^{\gamma}(\log t)^{j}\in C^{\infty}(\bar{G}_{r}).$$ 
Set 
$$v=u-\sum^{\infty}_{j=0}w^{(j)}t^{\gamma}(\log t)^{j}.$$ 
Then, 
$$u=v+\sum^{\infty}_{j=0}w^{(j)}t^{\gamma}(\log t)^{j}\quad\text{absolutely and uniformly in } \bar{G}_{r},$$ 
and $v,w^{(j)}\in C^{\infty}(\bar{G}_{r})$ for any $j\ge 0$, with $\partial_t^i w^{(j)}=0$ on $\Sigma_{1}$, 
for any $j\geq 0$, $0\leq i\leq [\overline{m}]+j-1$, and any $r\in(0,1)$.

Finally, for any $k\geq [\overline{m}]$, we write 
$$
\sum^{\infty}_{j=0}w^{(j)}t^{\gamma}(\log t)^{j}=\sum^{k-[\overline{m}]}_{j=0}
w^{(j)}t^{\gamma}(\log t)^{j}+\sum^{\infty}_{j=k-[\overline{m}]+1}w^{(j)}t^{\gamma}(\log t)^{j}.
$$
By \eqref{eq-ch4-def-w-int} and \eqref{eq-ch4-summation-regularity2-finite}, 
we have, for any $k\geq [\overline{m}]$ and any $r\in(0,1)$,
$$
\sum^{\infty}_{j=k-[\overline{m}]+1}w^{(j)}t^{\gamma}(\log t)^{j}\in C^{k}(\bar{G}_{r}),
$$
and we can obtain \eqref{eq-ch4-Decomposition-infinite-higher-non-int} easily.
\end{proof}

Under the additional assumption $\overline{m}> 1$, 
Theorem \ref{Thm-ch4-Linear-MainThm-curved-combined-smooth} 
follows from Theorem \ref{thrm-ch4-Decomposition-infinite-non-int} by renaming the functions $w^{(j)}$, for $j\ge 0$.

\section{Small H\"older Indices}\label{sec-Small-Holder-Indices}

In this section, we study the case that the H\"older index is smaller than the decimal part of 
the positive characteristic exponent. 
Specifically, let $\gamma$ be given by \eqref{eq-definition-gamma}. 
Throughout this section, we assume 
\begin{align}\label{eq-small-Holder-index}0<\alpha<\gamma<1\quad\text{on }B_1'.\end{align} 
We emphasize that $\gamma$ is a function on $B_1'$, instead of a constant. 

There is a new element in this section. 
In Lemma \ref{lemma-ch4-Linear-MainThm-non-integer-small} 
and Theorem \ref{thrm-ch4-Linear-MainThm-non-integer-small} below, 
we establish two different forms of expansions and different estimates of the remainders. 
The remainders have integer orders in one expansion and non-integer orders in another expansion. 
Different expansions are needed to study the regularity of solutions in the tangential directions 
and the normal direction in Theorem \ref{Thm-ch4-Linear-MainThm-decomposition-non-int-small}.

\begin{lemma}\label{lemma-ch4-Linear-MainThm-non-integer-small}
Suppose that $\underline{m}$, $\overline{m}$, and $\gamma$ are functions on $B'_{1}$ satisfying 
\eqref{eq-ch4-Assumption_m1}, \eqref{eq-ch4-Assumption_m2}, 
and \eqref{eq-definition-gamma}, 
and that $\ell$ is an integer and $\alpha\in(0,1)$ is a constant such that $\ell\geq [\overline{m}]+1$ and 
\eqref{eq-small-Holder-index} holds. 
Assume $a_{n n}, b_{n}, c \in C^{\ell+1, \alpha}(\bar{G}_{1})$ 
and $a_{i j}, b_{i}, f \in C^{\ell, \alpha}(\bar{G}_{1})$ for $i \neq n$, 
and let $u \in C(\bar{G}_{1}) \cap C^{2}(G_{1})$ 
be a solution of \eqref{eq-ch3-Equ} and \eqref{eq-ch3-Dirichlet}. 
Then, 
\begin{align*}
u&=\sum_{i=0}^{[\overline{m}]} c_{i} t^{i}+c_{[\overline{m}], 0} t^{[\overline{m}]+\gamma}+R'_{[\overline{m}]}\\
&=\sum_{i=0}^{[\overline{m}]+1} c_{i} t^{i}+c_{[\overline{m}], 0} t^{[\overline{m}]+\gamma}+R_{[\overline{m}]+1} 
\quad \text {in } G_{1},
\end{align*}
where $c_{0},c_{1},\cdots,c_{[\overline{m}]+1}$, and $c_{[\overline{m}], 0}$ are functions on $B'_{1}$, and $R'_{[\overline{m}]}$ and $R_{[\overline{m}]+1}$ are functions in $G_{1}$ satisfying, 
for any $x'_{0}\in B'_{1}$, 
any $r\in(0,1-|x'_{0}|)$, 
and any $\epsilon$ with $\alpha<\epsilon <\alpha+1-\gamma$ on $\bar B_{r}'(x_0')$, 
%for any $x'_{0}\in B'_{1}$ and any $\epsilon \in(\alpha, \alpha+1-\gamma(x'_{0}))$, 
%there exist an $r_{\epsilon}\in(0,1-|x'_{0}|)$ such that
\begin{align*}
&c_{i} \in C^{\ell-i, \alpha}(\bar B_{r}^{\prime}(x'_{0}))\quad\text{for }i=0,\cdots,[\overline{m}]+1,\\
&c_{[\overline{m}], 0} \in C^{\ell-[\overline{m}]-1, \epsilon}(\bar B_{r}^{\prime}(x'_{0})),
\end{align*}   
for any $\nu \leq [\overline{m}]$ and $\tau \leq \ell-1-[\overline{m}]$, %and any $r \in(0,r_{\epsilon})$,
$$
\partial_{t}^{\nu} D_{x^{\prime}}^{\tau} (t^{-\gamma}R'_{[\overline{m}]}), 
t \partial_{t} \partial_{t}^{\nu} D_{x^{\prime}}^{\tau} (t^{-\gamma}R'_{[\overline{m}]}), 
t^{2} \partial_{t}^{2} \partial_{t}^{\nu} D_{x^{\prime}}^{\tau} (t^{-\gamma}R'_{[\overline{m}]}) 
\in C^{\epsilon,\epsilon-\alpha}_{x',t}(\bar{G}_{r}(x_{0}')),   
$$
and
\begin{align*}
&\big|\partial_{t}^{\nu} D_{x^{\prime}}^{\tau} R'_{[\overline{m}]}\big|
+\big|t \partial_{t} \partial_{t}^{\nu} D_{x^{\prime}}^{\tau} R'_{[\overline{m}]}\big|
+\big|t^{2} \partial_{t}^{2} \partial_{t}^{\nu} D_{x^{\prime}}^{\tau} R'_{[\overline{m}]}\big| \\
&\qquad \leq C t^{[\overline{m}]+1-\nu}\big\{|u|_{L^{\infty}(G_{1})}+|f|_{C^{\ell, \alpha}(\bar{G}_{1})}\big\} 
\quad \text {in }\bar{G}_{r}(x_{0}'),
\end{align*} 
and, for any $\nu \leq [\overline{m}]+1$ and $\tau \leq \ell-[\overline{m}]-1$, %and any $r \in(0,r_{\epsilon})$,
$$
\partial_{t}^{\nu} D_{x^{\prime}}^{\tau} R_{[\overline{m}]+1}, 
t \partial_{t} \partial_{t}^{\nu} D_{x^{\prime}}^{\tau} R_{[\overline{m}]+1}, 
t^{2} \partial_{t}^{2} \partial_{t}^{\nu} D_{x^{\prime}}^{\tau} R_{[\overline{m}]+1} 
\in C^{\alpha}(\bar{G}_{r}(x_{0}')),   
$$
and
\begin{align*}
&\big|\partial_{t}^{\nu} D_{x^{\prime}}^{\tau} R_{[\overline{m}]+1}\big|
+\big|t \partial_{t} \partial_{t}^{\nu} D_{x^{\prime}}^{\tau} R_{[\overline{m}]+1}\big|
+\big|t^{2} \partial_{t}^{2} \partial_{t}^{\nu} D_{x^{\prime}}^{\tau} R_{[\overline{m}]+1}\big| \\
&\qquad \leq C t^{[\overline{m}]+1-\nu+\alpha}\big\{|u|_{L^{\infty}(G_{1})}
+|f|_{C^{\ell, \alpha}(\bar{G}_{1})}\big\} \quad \text {in }\bar{G}_{r}(x_{0}'),
\end{align*} 
for some positive constant $C$ depending only on $n$, $\lambda$, $\ell$, $\alpha$, $r$, $\epsilon$, 
$\underline{m}$, $\overline{m}$, the $C^{\ell+1, \alpha}$-norms of $a_{n n}, b_{n}, c$ in $\bar{G}_{1}$, 
and the $C^{\ell, \alpha}$-norms of $a_{i j}, b_{i}$ in $\bar{G}_{1}$ for $i \neq n.$ 
\end{lemma}

\begin{proof}
Note that $\overline{m},\underline{m},\gamma\in C^{\ell+1, \alpha}(B'_{1})$ 
by \eqref{eq-ch4-Assumption_m1}, \eqref{eq-ch4-Assumption_m2}, 
and \eqref{eq-definition-gamma}, and that $[\overline{m}]$ is an integer. 
Without loss of generality, we assume $x'_{0}=0$ %, $r_0=1$, and 
and take a constant $\epsilon_0$ such that 
$\alpha<\epsilon_0 \le\alpha+1-\gamma$ on $\bar B_{1}'$. 
Throughout the proof, $\epsilon$ is an arbitrary constant in $(\alpha,\epsilon_{0})$ 
and $r$ is an arbitrary constant in $(0,1)$.

By Theorem \ref{thrm-ch3-Linear-NormalEstimate-general}, 
we have, for any $\nu \leq[\overline{m}]$ and $\tau \leq \ell-\nu$,
\begin{equation*}%\label{eq-ch4-regularity-u-non-int-preli1-small}
\partial_{t}^{\nu} D_{x^{\prime}}^{\tau} u, t D \partial_{t}^{\nu} D_{x^{\prime}}^{\tau} u, 
t^{2} D^{2} \partial_{t}^{\nu} D_{x^{\prime}}^{\tau} u \in C^{\alpha}(\bar{G}_{r}).    
\end{equation*}
By %\eqref{eq-ch4-regularity-u-non-int-preli1-small} and 
a similar argument as in the proof of 
Lemma \ref{lemma-ch4-Linear-MainThm-non-integer-large}, 
we get, for any $\nu \leq[\overline{m}]+1$ and $\tau \leq \ell-\nu$,
\begin{equation}\label{eq-ch4-regularity-F1-non-int-small}
\partial_{t}^{\nu} D_{x^{\prime}}^{\tau} F \in C^{\alpha}(\bar{G}_{r}).    
\end{equation}
Set
$$
a_{i}=\frac{1}{i !} \partial_{t}^{i} F(\cdot, 0) \quad \text{for } i=0, \cdots,[\overline{m}]+1.
$$
By \eqref{eq-ch4-regularity-F1-non-int-small}, we have
\begin{equation}\label{eq-regularity-a-non-small} 
a_{i}\in C^{\ell-i,\alpha}(B_{1}')\quad\text{for any }i=0,\cdots,[\overline{m}]+1.
\end{equation}

We first consider 
\begin{equation}\label{eq-ch4-identity-F-non-int-overline-m-small-1}
F=\sum_{i=0}^{[\overline{m}]} a_{i} t^{i}+S'_{[\overline{m}]}.
\end{equation}
Then, for any $\nu\leq [\overline{m}],\tau\leq \ell-[\overline{m}]$ and any $\nu\leq [\overline{m}]+1,\tau\leq \ell-[\overline{m}]-1$,
\begin{align}\label{eq-regularity-S-m-non-small-1a}\begin{split}
&\partial_{t}^{\nu} D_{x^{\prime}}^{\tau}S'_{[\overline{m}]} \in C^{\alpha}(\bar{G}_{r}),\\
&|\partial_{t}^{\nu} D_{x^{\prime}}^{\tau}S'_{[\overline{m}]}|\leq Ct^{[\overline{m}]+1-\nu}\quad\text{in }\bar{G}_{r}.
\end{split}
\end{align}
Here, we use $\prime$ in $S'_{[\overline{m}]}$ to indicate that $S'_{[\overline{m}]}$ is the remainder 
at the integer order and distinguish it from the remainder $S_{[\overline{m}]+1}$ to be used at the
non-integer order. By substituting \eqref{eq-ch4-identity-F-non-int-overline-m-small-1} 
in \eqref{eq-ch4-LinearODE-Solution-l} and a straightforward computation, we have
\begin{equation}\label{eq-ch4-indentity-u-int-overline-m-non-small-1}
u=\sum_{i=0}^{[\overline{m}]} c_{i} t^{i}+c_{[\overline{m}], 0} t^{[\overline{m}]+\gamma}+R'_{[\overline{m}]},
\end{equation}
where
\begin{align}\label{eq-expression-coefficient-m-int}\begin{split}
c_{i}&= \frac{a_{i}}{(i-\underline{m})(i-\overline{m})}\quad \text {for } i=0, \cdots,[\overline{m}], \\
c_{[\overline{m}], 0}&=u(\cdot, r) r^{-\overline{m}} 
-\sum_{i=0}^{[\overline{m}]} \frac{a_{i} r^{i-\overline{m}}}{(i-\underline{m})(i-\overline{m})} \\
&\qquad+\frac{r^{\underline{m}-\overline{m}}}{\overline{m}-\underline{m}} \int_{0}^{r} s^{-1-\underline{m}} S'_{[\overline{m}]}d s-\frac{1}{\overline{m}-\underline{m}} \int_{0}^{r} s^{-1-\overline{m}} S'_{[\overline{m}]} d s,
\end{split}
\end{align}
and
\begin{equation}\label{eq-ch4-expression-remainder-m-int-non-small-1}
R'_{[\overline{m}]}=\frac{t^{\overline{m}}}{\overline{m}-\underline{m}} 
\int_{0}^{t} s^{-1-\overline{m}} S'_{[\overline{m}]} d s
-\frac{t^{\underline{m}}}{\overline{m}-\underline{m}} \int_{0}^{t} s^{-1-\underline{m}} S'_{[\overline{m}]} d s.
\end{equation}
By \eqref{eq-regularity-a-non-small}, we have
$$
c_{i}\in C^{\ell-i,\alpha}(B'_{1})\quad\text{for any }i=0,\cdots,[\overline{m}].
$$
By \eqref{eq-regularity-S-m-non-small-1a}, 
Lemma \ref{lemma-ch4-BasicHolderRegularity} with $a=-\underline{m}$ and $k=[\overline{m}]$, 
and Lemma \ref{lemma-ch4-singular-integral-non-int-lower-smaller}  with $a=\overline{m}$, we have
$$
c_{[\overline{m}], 0}\in C^{\ell-[\overline{m}]-1,\epsilon}(B'_{1}).
$$
We emphasize that this is an almost optimal regularity of $c_{[\overline{m}],0}$, 
with a minor loss of the H\"older index.
If we include the term $t^{[\overline{m}]+1}$ and use the remainder $S_{[\overline{m}]+1}$ in 
\eqref{eq-ch4-identity-F-non-int-overline-m-small-1}, we will not get the optimal regularity of $c_{[\overline{m}], 0}$.

With \eqref{eq-regularity-S-m-non-small-1a} and \eqref{eq-ch4-expression-remainder-m-int-non-small-1},
we have, for any nonnegative integers $\nu \leq [\overline{m}]$ and $\tau \leq \ell-1-[\overline{m}]$,
$$
\big|\partial_{t}^{\nu} D_{x^{\prime}}^{\tau} R'_{[\overline{m}]}\big|
+\big|t \partial_{t} \partial_{t}^{\nu} D_{x^{\prime}}^{\tau} R'_{[\overline{m}]}\big|
+\big|t^{2} \partial_{t}^{2} \partial_{t}^{\nu} D_{x^{\prime}}^{\tau} R'_{[\overline{m}]}\big| 
\leq C t^{[\overline{m}]-\nu+1}\quad \text {in } G_{1 / 2}.
$$
Hence, $ R'_{[\overline{m}]}$ in \eqref{eq-ch4-indentity-u-int-overline-m-non-small-1} 
is indeed the remainder of $u$ expanded up to $t^{[\overline{m}]+\gamma}$. 
Next, by \eqref{eq-ch4-expression-remainder-m-int-non-small-1}, we write
 \begin{equation}\label{eq-ch4-expression-remainder-m-int-non-small-102}
t^{-\gamma}R'_{[\overline{m}]}=t^{-\gamma}\underline{R}'_{[\overline{m}]}
+t^{-\gamma}\overline{R}'_{[\overline{m}]},
\end{equation}
where
%\begin{align}\label{eq-ch4-expression-remainder-m-int-non-small-101}\begin{split}
%&\overline{R}'_{[\overline{m}]}=\frac{t^{\overline{m}}}{\overline{m}-\underline{m}} 
%\int_{0}^{t} s^{-1-\overline{m}} S'_{[\overline{m}]} d s,\\
%&\underline{R}'_{[\overline{m}]}=-\frac{t^{\underline{m}}}{\overline{m}-\underline{m}} 
%\int_{0}^{t} s^{-1-\underline{m}} S'_{[\overline{m}]} d s.
%\end{split}\end{align}
%By \eqref{eq-ch4-expression-remainder-m-int-non-small-101}, we write 
\begin{align*}
t^{-\gamma}\underline{R}'_{[\overline{m}]}&=-\frac{t^{\underline{m}-\gamma}}{\overline{m}-\underline{m}}
 \int_{0}^{t} s^{-1-(\underline{m}-\gamma)} s^{-\gamma}S'_{[\overline{m}]} d s,\\
 t^{-\gamma}\overline{R}'_{[\overline{m}]}&=\frac{t^{\overline{m}-\gamma}}{\overline{m}-\underline{m}} 
\int_{0}^{t} s^{-1-(\overline{m}-\gamma)} s^{-\gamma}S'_{[\overline{m}]} d s.
\end{align*}
A simple computation yields
\begin{equation}\label{eq-ch4-expression-remainder-m-int-non-small-103}
t\partial_{t}(t^{-\gamma}R'_{[\overline{m}]})
=(\underline{m}-\gamma) t^{-\gamma}\underline{R}'_{[\overline{m}]}
+(\overline{m}-\gamma)t^{-\gamma}\overline{R}'_{[\overline{m}]},
\end{equation}
and
\begin{equation}\label{eq-ch4-expression-remainder-m-int-non-small-104}\begin{split}
t^{2}\partial^{2}_{t}(t^{-\gamma}R'_{[\overline{m}]})
&=t^{-\gamma}S'_{[\overline{m}]}
+(\underline{m}-\gamma)(\underline{m}-\gamma-1) t^{-\gamma}\underline{R}'_{[\overline{m}]}\\
&\qquad+(\overline{m}-\gamma)(\overline{m}-\gamma-1)t^{-\gamma}\overline{R}'_{[\overline{m}]}.
\end{split}\end{equation}
By \eqref{eq-regularity-S-m-non-small-1a} 
and Lemma \ref{lemma-ch4-regularity-power2}(i) with $k=[\overline{m}]+1$ and $l=\ell-[\overline{m}]-1$, 
we have, for any $\nu\leq [\overline{m}]$ and $\tau\leq \ell-1-[\overline{m}]$, 
\begin{align*}
&\partial_{t}^{\nu} D_{x^{\prime}}^{\tau}(t^{-\gamma}S'_{[\overline{m}]})
\in  C^{\epsilon,\epsilon-\alpha}_{x',t}(\bar{G}_{r}),\\
&|\partial_{t}^{\nu} D_{x^{\prime}}^{\tau}(t^{-\gamma}S'_{[\overline{m}]})|
\leq Ct^{[\overline{m}]-\nu+\epsilon-\alpha}\quad\text{in }\bar{G}_{r}.
\end{align*}
Note that $\underline{m}-\gamma<0$ and $\overline{m}-\gamma=[\overline{m}]$ by the definition of $\gamma$. 
By Lemma \ref{lemma-ch4-BasicHolderRegularity} with $a=\gamma-\underline{m}$ 
and Lemma \ref{lemma-ch4-singular-integral-int-lower} with $a=[\overline{m}]$, 
we get, for any $\nu\leq [\overline{m}]$ and $\tau\leq \ell-1-[\overline{m}]$,
$$
\partial_{t}^{\nu} D_{x^{\prime}}^{\tau} (t^{-\gamma}\underline{R}'_{[\overline{m}]}), \,
\partial_{t}^{\nu} D_{x^{\prime}}^{\tau} (t^{-\gamma}\overline{R}'_{[\overline{m}]})
\in C^{\epsilon,\epsilon-\alpha}_{x',t}(\bar{G}_{r}).   
$$
%By \eqref{eq-ch4-expression-remainder-m-int-non-small-102}, 
%\eqref{eq-ch4-expression-remainder-m-int-non-small-103}, 
%and \eqref{eq-ch4-expression-remainder-m-int-non-small-104},  
As a consequence, we obtain, for any $\nu\leq [\overline{m}]$ and $\tau\leq \ell-1-[\overline{m}]$,
$$
\partial_{t}^{\nu} D_{x^{\prime}}^{\tau} (t^{-\gamma}R'_{[\overline{m}]}), 
t \partial_{t} \partial_{t}^{\nu} D_{x^{\prime}}^{\tau} (t^{-\gamma}R'_{[\overline{m}]}), 
t^{2} \partial_{t}^{2} \partial_{t}^{\nu} D_{x^{\prime}}^{\tau} (t^{-\gamma}R'_{[\overline{m}]}) 
\in C^{\epsilon,\epsilon-\alpha}_{x',t}(\bar{G}_{r}).   
$$
This implies the desired H\"older regularity of $t^{-\gamma}R'_{[\overline{m}]}$. 
Although the H\"older regularity of $t^{-\gamma}R'_{[\overline{m}]}$ 
in $t$ is not needed in applications later on,
we need the H\"older regularity of $t^{-\gamma}S'_{[\overline{m}]}$ both in $x'$ and $t$ 
in order to apply Lemma \ref{lemma-ch4-singular-integral-int-lower}.

We now continue to expand $u$. Write
\begin{equation}\label{eq-ch4-identity-F-non-int-overline-m-small-201}
F=\sum_{i=0}^{[\overline{m}]+1} a_{i} t^{i}+S_{[\overline{m}]+1}.
\end{equation}
In other words,
\begin{equation}\label{eq-ch4-identity-F-non-int-overline-m-small-2}
S'_{[\overline{m}]}=a_{[\overline{m}]+1}t^{[\overline{m}]+1}+S_{[\overline{m}]+1},
\end{equation}
Then, for any $\nu\leq [\overline{m}]+1$ and $\tau\leq \ell-[\overline{m}]-1$,
\begin{align}\label{eq-regularity-S-m-non-small-2}\begin{split}
&\partial_{t}^{\nu} D_{x^{\prime}}^{\tau}S_{[\overline{m}]+1} \in C^{\alpha}(\bar{G}_{r}),\\
&|\partial_{t}^{\nu} D_{x^{\prime}}^{\tau}S_{[\overline{m}]+1}|\leq Ct^{[\overline{m}]+1-\nu+\alpha}
\quad\text{in }\bar{G}_{r}.
\end{split}\end{align}
By substituting \eqref{eq-ch4-identity-F-non-int-overline-m-small-2} 
in \eqref{eq-ch4-expression-remainder-m-int-non-small-1} 
and a straightforward computation, we have
$$
R_{[\overline{m}]}'=c_{[\overline{m}]+1}t^{[\overline{m}]+1}+R_{[\overline{m}]+1},
$$
where $c_{[\overline{m}]+1}$ is given by 
\begin{equation}\label{eq-ch4-expression-coefficient-m1-int-non-small-2}
c_{i}= \frac{a_{i}}{(i-\underline{m})(i-\overline{m})}\quad \text {for } i=[\overline{m}]+1, \\
\end{equation}
and
\begin{align}\label{eq-ch4-expression-remainder-m-int-non-small-2}%\begin{split}
R_{[\overline{m}]+1}
=\frac{t^{\overline{m}}}{\overline{m}-\underline{m}} \int_{0}^{t} s^{-1-\overline{m}} S_{[\overline{m}]+1} d s
-\frac{t^{\underline{m}}}{\overline{m}-\underline{m}} \int_{0}^{t} s^{-1-\underline{m}} S_{[\overline{m}]+1} d s.
%\end{split} 
\end{align}
With \eqref{eq-ch4-indentity-u-int-overline-m-non-small-1}, we obtain
$$
u=\sum_{i=0}^{[\overline{m}]+1}c_it^i+c_{[\overline{m}],0}t^{[\overline{m}]+\gamma}+R_{[\overline{m}]+1}.
$$
By \eqref{eq-regularity-a-non-small}, we have
$$
c_{i}\in C^{\ell-i,\alpha}(B'_{1})\quad\text{for }i=[\overline{m}]+1.
$$
Next, by \eqref{eq-ch4-expression-remainder-m-int-non-small-2}, we write 
$$
R_{[\overline{m}]+1}=\underline{R}_{[\overline{m}]+1}+\overline{R}_{[\overline{m}]+1},
$$
where
\begin{align*}
\overline{R}_{[\overline{m}]+1}&=\frac{t^{\overline{m}}}{\overline{m}-\underline{m}} 
\int_{0}^{t} s^{-1-\overline{m}} S_{[\overline{m}]+1} d s,\\ 
\underline{R}_{[\overline{m}]+1}
&=-\frac{t^{\underline{m}}}{\overline{m}-\underline{m}} \int_{0}^{t} s^{-1-\underline{m}} S_{[\overline{m}]+1} d s.   
\end{align*}
A simple computation yields
$$
t\partial_{t}R_{[\overline{m}]+1}=\underline{m} \underline{R}_{[\overline{m}]+1}
+\overline{m}\overline{R}_{[\overline{m}]+1},
$$
and
$$
t^{2}\partial^{2}_{t}R_{[\overline{m}]+1}=S_{[\overline{m}]+1}
+\underline{m}(\underline{m}-1) \underline{R}_{[\overline{m}]+1}
+\overline{m}(\overline{m}-1)\overline{R}_{[\overline{m}]+1}.
$$
By \eqref{eq-regularity-S-m-non-small-2}, 
Lemma \ref{lemma-ch4-BasicHolderRegularity} 
with $k=[\overline{m}]+1$, $a=-\underline{m}$, and $\beta=\alpha$, 
and Lemma \ref{lemma-ch4-singular-integral-non-int-higher} 
with $k=[\overline{m}]+1$, $a=\overline{m}$, and $\beta=\alpha$, 
we have, for any $\nu \leq [\overline{m}]+1$ and $\tau \leq \ell-[\overline{m}]-1$, 
$$
\partial_{t}^{\nu} D_{x^{\prime}}^{\tau} \underline{R}_{[\overline{m}]+1},
\partial_{t}^{\nu} D_{x^{\prime}}^{\tau} \overline{R}_{[\overline{m}]+1}\in C^{\alpha}(\bar{G}_{r}),
$$
and
$$
\big|\partial_{t}^{\nu} D_{x^{\prime}}^{\tau} \underline{R}_{[\overline{m}]+1}\big|
+\big|\partial_{t}^{\nu} D_{x^{\prime}}^{\tau} \overline{R}_{[\overline{m}]+1}\big|
\leq C t^{[\overline{m}]+1-\nu+\alpha}\quad \text {in } G_{1 / 2}.
$$
As a consequence, we obtain, for any $\nu \leq [\overline{m}]+1$ and $\tau \leq \ell-[\overline{m}]-1$, 
$$
\partial_{t}^{\nu} D_{x^{\prime}}^{\tau} R_{[\overline{m}]+1}, 
t \partial_{t} \partial_{t}^{\nu} D_{x^{\prime}}^{\tau} R_{[\overline{m}]+1}, 
t^{2} \partial_{t}^{2} \partial_{t}^{\nu} D_{x^{\prime}}^{\tau} R_{[\overline{m}]+1} \in C^{\alpha}(\bar{G}_{r}),
$$
and
\begin{align*}
&|\partial_{t}^{\nu} D_{x^{\prime}}^{\tau} R_{[\overline{m}]+1}|
+|t \partial_{t} \partial_{t}^{\nu} D_{x^{\prime}}^{\tau} R_{[\overline{m}]+1}|\
+|t^{2} \partial_{t}^{2} \partial_{t}^{\nu} D_{x^{\prime}}^{\tau} R_{[\overline{m}]+1}| \\
&\qquad\leq C t^{[\overline{m}]+1-\nu+\alpha}\quad \text {in } G_{1 / 2}.
\end{align*}
We conclude the desired result.
\end{proof}

Next, we study the higher regularity.

\begin{theorem}\label{thrm-ch4-Linear-MainThm-non-integer-small}
Suppose that $\underline{m}$, $\overline{m}$, and $\gamma$ are functions on $B'_{1}$ satisfying 
\eqref{eq-ch4-Assumption_m1}, \eqref{eq-ch4-Assumption_m2}, 
and \eqref{eq-definition-gamma}, 
and that $\ell$ is an integer and $\alpha\in(0,1)$ is a constant such that $\ell\geq [\overline{m}]+1$ 
and \eqref{eq-small-Holder-index} holds. 
Assume $a_{n n}, b_{n}, c \in C^{\ell+1, \alpha}(\bar{G}_{1})$ 
and $a_{i j}, b_{i}, f \in C^{\ell, \alpha}(\bar{G}_{1})$ for $i \neq n$, 
and let $u \in C(\bar{G}_{1}) \cap C^{2}(G_{1})$ be a solution 
of \eqref{eq-ch3-Equ} and \eqref{eq-ch3-Dirichlet}. %, with \eqref{eq-ch3-boundary-value}. 
Then, for any $k$ with $[\overline{m}]+1\leq k\leq \ell$,
\begin{align}\label{eq-ch4-LinearExpansion-non-int-k-small}\begin{split}
u&=\sum_{i=0}^{k-1} c_{i} t^{i}+\sum^{k-1}_{i=[\overline{m}]}
\sum^{i-[\overline{m}]}_{j=0}c_{i, j} t^{i+\gamma}(\log t)^{j}+R'_{k-1}\\
&=\sum_{i=0}^{k} c_{i} t^{i}+\sum^{k-1}_{i=[\overline{m}]}
\sum^{i-[\overline{m}]}_{j=0}c_{i, j} t^{i+\gamma}(\log t)^{j}+R_{k} \quad \text {in } G_{1},
\end{split}\end{align}
where $\{c_{i}\}_{i=0}^{\ell}$ and $\{c_{i, j}\}_{ [\overline{m}]\leq i\leq \ell-1,0\leq j\leq i-[\overline{m}]}$ 
are functions on $B'_{1}$, and $R'_{k-1}$ and $R_{k}$ are functions in $G_{1}$ satisfying, 
for any $x'_{0}\in B'_{1}$, 
any $r\in(0,1-|x'_{0}|)$, 
and any $\epsilon$ with $\alpha<\epsilon <\alpha+1-\gamma$ on $\bar B_{r}'(x_0')$, 
\begin{align}\label{eq-ch4-regularity-coefficients-non-int-small}\begin{split}
&c_{i} \in C^{\ell-i, \alpha}(\bar B_{r}^{\prime}(x'_{0}))\quad\text{for }i=0,\cdots,\ell,\\
&c_{i, j} \in C^{\ell-i-1, \epsilon}(\bar B_{r}^{\prime}(x'_{0}))
\quad\text{for }i=[\overline{m}],\cdots,\ell-1\text{ and }j=0,\cdots,i-[\overline{m}],
\end{split}\end{align}
for any $\nu \leq k-1$ and $\tau \leq \ell-k$, %and any $r \in(0,r_0)$,
\begin{align}\label{eq-ch4-LinearRegularity-non-int-m1-prime-small}\begin{split}
\partial_{t}^{\nu} D_{x^{\prime}}^{\tau} (t^{-\gamma}R'_{k-1}),
& t \partial_{t} \partial_{t}^{\nu} D_{x^{\prime}}^{\tau} (t^{-\gamma}R'_{k-1}),\\
&t^{2} \partial_{t}^{2} \partial_{t}^{\nu} D_{x^{\prime}}^{\tau} (t^{-\gamma}R'_{k-1}) 
\in C^{\epsilon,\epsilon-\alpha}_{x',t}(\bar{G}_{r}(x_{0}')), 
\end{split}\end{align}
and
\begin{align}\label{eq-ch4-LinearRegularity-non-int-m2-prime-small}\begin{split}
&\big|\partial_{t}^{\nu} D_{x^{\prime}}^{\tau} R'_{k-1}\big|
+\big|t \partial_{t} \partial_{t}^{\nu} D_{x^{\prime}}^{\tau} R'_{k-1}\big|
+\big|t^{2} \partial_{t}^{2} \partial_{t}^{\nu} D_{x^{\prime}}^{\tau} R'_{k-1}\big| \\
&\qquad \leq C t^{k-\nu}\big\{|u|_{L^{\infty}(G_{1})}
+|f|_{C^{\ell, \alpha}(\bar{G}_{1})}\big\} \quad \text {in }\bar{G}_{r}(x_{0}'),
\end{split}\end{align}
and, for any $\nu \leq k$ and $\tau \leq \ell-k$, %and any $r \in(0,r_0)$,
\begin{align}\label{eq-ch4-LinearRegularity-non-int-m1-small}
\partial_{t}^{\nu} D_{x^{\prime}}^{\tau} R_{k}, t \partial_{t} \partial_{t}^{\nu} D_{x^{\prime}}^{\tau} R_{k}, 
t^{2} \partial_{t}^{2} \partial_{t}^{\nu} D_{x^{\prime}}^{\tau} R_{k} \in C^{\alpha}(\bar{G}_{r}(x_{0}')),   
\end{align}
and
\begin{align}\label{eq-ch4-LinearRegularity-non-int-m2-small}\begin{split}
&\big|\partial_{t}^{\nu} D_{x^{\prime}}^{\tau} R_{k}\big|
+\big|t \partial_{t} \partial_{t}^{\nu} D_{x^{\prime}}^{\tau} R_{k}\big|
+\big|t^{2} \partial_{t}^{2} \partial_{t}^{\nu} D_{x^{\prime}}^{\tau} R_{k}\big| \\
&\qquad \leq C t^{k-\nu+\alpha}\big\{|u|_{L^{\infty}(G_{1})}+|f|_{C^{\ell, \alpha}(\bar{G}_{1})}\big\}
 \quad \text {in }\bar{G}_{r}(x_{0}'),
\end{split}\end{align}
for some positive constant $C$ depending only on $n$, $\lambda$, $\ell$, $\alpha$, $r$, $\epsilon$, 
$\underline{m}$, $\overline{m}$, the $C^{\ell+1, \alpha}$-norms of $a_{n n}, b_{n}, c$ in $\bar{G}_{1}$, 
and the $C^{\ell, \alpha}$-norms of $a_{i j}, b_{i}$ in $\bar{G}_{1}$ for $i \neq n.$ 
If, in addition, $c_{[\overline{m}], 0}=0$ on $B'_{r}(x_0')$, 
for some $x'_{0}\in B'_{1}$ and $r\in(0,1-|x'_{0}|)$, 
then $c_{i, j}=0$ on $B'_{r}(x_0')$, 
for any $i=[\overline{m}], \cdots, \ell-1$ and $j=0,\cdots,i-[\overline{m}]$, 
and $u \in C^{\ell,\alpha}(\bar{G}_{r}(x_{0}'))$.  
%and any $\epsilon$ with $\alpha<\epsilon <\alpha+1-\gamma$ on $\bar B_{r}'(x_0')$.
\end{theorem}

The proof of %Theorem \ref{thrm-ch4-Linear-MainThm-non-integer-large} and 
Theorem \ref{thrm-ch4-Linear-MainThm-non-integer-small} 
is similar to that of Theorem \ref{thrm-ch4-Linear-MainThm-non-integer-large}
and hence omitted. 
We obtain the regularity of $t^{-\gamma} R'_{k-1}$ 
by Lemma \ref{lemma-ch4-BasicHolderRegularity} 
and Lemma \ref{lemma-ch4-singular-integral-int-higher}
and the regularity of $R_k$ by Lemma \ref{lemma-ch4-BasicHolderRegularity} 
and Lemma \ref{lemma-ch4-singular-integral-non-int-higher}. 
We point out two differences, the pattern of the expansion 
and the regularity of the remainders. 

The second expansion in %\eqref{eq-ch4-LinearExpansion-non-int-k-large} and 
\eqref{eq-ch4-LinearExpansion-non-int-k-small} with $k=\ell$ has the following pattern: 
\begin{align*}&t^0, t^1, \cdots, t^{[\overline{m}]},t^{[\overline{m}]+\gamma},
t^{[\overline{m}]+1},t^{[\overline{m}]+1+\gamma}\log t,t^{[\overline{m}]+1+\gamma},\cdots,\\ 
&\qquad t^{\ell-1}, t^{\ell-1+\gamma}(\log t)^{\ell-1-[\overline{m}]}, \cdots, 
t^{\ell-1+\gamma}\log t, t^{\ell-1+\gamma}, t^{\ell}.\end{align*}
The coefficients $c_0, \cdots, c_{\ell}$ have an optimal regularity 
determined by the regularity of $a_{ij}, b_i, c$, and $f$, 
and there is a loss of regularity for $c_{i,j}$, 
for $[\overline{m}]\le i\le \ell-1$ and $0\le j\le i-[\overline{m}]$. 
%with $\ell_*=\ell$ in  \eqref{eq-ch4-LinearExpansion-non-int-k-large}
%and $\ell_*=\ell-1$ in  \eqref{eq-ch4-LinearExpansion-non-int-k-small}. 
The first term with non-integer power is given by $t^{[\overline{m}]+\gamma}$. 
If the first non-integer power does not appear, 
then there are no non-integer power or logarithmic terms in the expansion and the solutions are as regular as the 
coefficients and nonhomogeneous terms allow. 

We point out that $R_k$ in \eqref{eq-ch4-LinearExpansion-non-int-k-large}
and \eqref{eq-ch4-LinearExpansion-non-int-k-small} has the same decay order 
according to \eqref{eq-ch4-LinearRegularity-non-int-m2-large}
and \eqref{eq-ch4-LinearRegularity-non-int-m2-small}; namely, 
\begin{align*}
|R_{k}|
\le Ct^{k+\alpha}\big\{|u|_{L^\infty(G_1)}+|f|_{C^{\ell,\alpha}(\bar G_1)}\big\}
\quad\text{in }G_{1/2}.\end{align*}
Due to the different relations $\gamma<\alpha$ and $\alpha<\gamma$, 
$R_k$ appears in different locations in the expansion. 
Specifically, $R_k$ appears after the singular term $t^{k+\gamma}$ if $\gamma<\alpha$ and 
after the term $t^k$ if $\alpha<\gamma$. 
Moreover, $R_{k}$ is $C^\alpha$ in $t$ but $C^\epsilon$ in $x'$ 
for any $\epsilon\in (0,\alpha-\gamma)$
in Theorem \ref{thrm-ch4-Linear-MainThm-non-integer-large}
and 
$R_{k}$ is $C^\alpha$ in both $t$ and $x'$
in Theorem \ref{thrm-ch4-Linear-MainThm-non-integer-small}.
In other words, there is no loss of regularity in Theorem \ref{thrm-ch4-Linear-MainThm-non-integer-small}
but a loss of regularity in Theorem \ref{thrm-ch4-Linear-MainThm-non-integer-large}, 
with the loss only in the $x'$ direction.

\begin{theorem}\label{Thm-ch4-Linear-MainThm-decomposition-non-int-small}
Suppose that $\underline{m}$, $\overline{m}$, and $\gamma$ are functions on $B'_{1}$ satisfying 
\eqref{eq-ch4-Assumption_m1}, \eqref{eq-ch4-Assumption_m2}, 
and \eqref{eq-definition-gamma}, 
and that $\ell$ is an integer and $\alpha\in(0,1)$ is a constant such that $\ell\geq [\overline{m}]+1$ 
and \eqref{eq-small-Holder-index} holds. 
Assume $a_{n n}, b_{n}, c \in C^{\ell+1, \alpha}(\bar{G}_{1})$ 
and $a_{i j}, b_{i}, f \in C^{\ell, \alpha}(\bar{G}_{1})$ for $i \neq n$, 
and let $u \in C(\bar{G}_{1}) \cap C^{2}(G_{1})$ be a solution 
of \eqref{eq-ch3-Equ} and \eqref{eq-ch3-Dirichlet}. 
Then,
\begin{align}\label{eq-ch4-LinearExpansion-decomposition-non-small}
u=v+\sum^{\ell-1-[\overline{m}]}_{j=0}w^{(j)}t^{\gamma}(\log t)^{j} \quad\text{in } G_{1},
\end{align}
where $v$ and $w^{(0)},\cdots,w^{(\ell-1-[\overline{m}])}$ are functions in $G_{1}$ satisfying, 
for any $x'_{0}\in B'_{1}$, 
any $r\in(0,1-|x'_{0}|)$, 
and any $\epsilon$ with $\alpha<\epsilon <\alpha+1-\gamma$ on $\bar B_{r}'(x_0')$, 
\begin{align*}
&v \in C^{\ell, \alpha}(\bar{G}_{r}(x_{0}')),\\
&w^{(j)} \in C^{\ell-1, \epsilon}(\bar{G}_{r}(x_{0}'))\quad\text{for }j=0,\cdots,\ell-1-[\overline{m}],
\end{align*}
with $\partial_t^i w^{(j)}=0$ on $\Sigma_{1}$, for $j=0,\cdots,\ell-1-[\overline{m}]$ and $i=0,\cdots,[\overline{m}]+j-1$, 
and 
$$
|v|_{C^{\ell, \alpha}(\bar{G}_{r}(x_{0}'))}
+\sum^{\ell-1-[\overline{m}]}_{j=0}|w^{(j)}|_{C^{\ell-1, \epsilon}(\bar{G}_{r}(x_{0}'))} 
\leq C\big\{|u|_{L^{\infty}(G_1)}+|f|_{C^{\ell, \alpha}(\bar{G}_1)}\big\},
$$
where $C$ is a positive constant depending only on $n$, $\lambda$, $\ell$, $\alpha$, $\epsilon$, $r$, 
$\underline{m}$, $\overline{m}$, the $C^{\ell+1, \alpha}$-norms of $a_{n n}, b_n, c$ in $\bar{G}_1$, 
and the $C^{\ell, \alpha}$-norms of $a_{i j}, b_i$ in $\bar{G}_1$ for $i \neq n$.
If, in addition, $\partial_t^{[\overline{m}]} w^{(0)}=0$ on $\Sigma_{r}(x_0')$, 
for some $x'_{0}\in B'_{1}$ 
and $r\in(0,1-|x'_{0}|)$, 
then $u \in C^{\ell,\alpha}(\bar{G}_{r}(x_{0}'))$.  
%any $\epsilon\in(0, \alpha+1-\gamma(x_0))$, and any $r\in(0,r_{\epsilon})$.
\end{theorem} 

\begin{proof}
We will apply Theorem \ref{thrm-ch4-Linear-MainThm-non-integer-small} 
and Lemma \ref{lemma-ch4-extensions} to construct functions $v$ 
and $w^{(0)}$, $\cdots$, $w^{(\ell-[\overline{m}]-2)}$, 
and prove that there is a function $w^{(\ell-1-[\overline{m}])}$ with the desired regularity 
such that \eqref{eq-ch4-LinearExpansion-decomposition-non-small} holds.
%We will use notation in the proof of Theorem \ref{thrm-ch4-Linear-MainThm-non-integer-small}. 
Without loss of generality, we assume $x'_{0}=0$ %, $r_0=1$, and 
and take a constant $\epsilon_0$ such that 
$\alpha<\epsilon_0 \le\alpha+1-\gamma$ on $\bar B_{1}'$. 
Throughout the proof, $\epsilon$ is an arbitrary constant in $(\alpha,\epsilon_{0})$ 
and $r$ is an arbitrary constant in $(0,1)$. 

By Theorem \ref{thrm-ch4-Linear-MainThm-non-integer-small}, 
there exist functions $c_{i} \in C^{\ell-i, \alpha}(B'_{1})$, for $0 \leq i \leq$ $\ell$, 
and $c_{i, j} \in C^{\ell-1-i, \epsilon}(B'_{1})$, for $[\overline{m}] \leq i \leq \ell-1$ and $0\leq j\leq i-[\overline{m}]$, 
such that
$$
u=\sum_{i=0}^{\ell} c_{i} t^{i}+\sum_{i=[\overline{m}]}^{\ell-1}
\sum^{i-[\overline{m}]}_{j=0} c_{i, j} t^{i+\gamma}(\log t)^{j}+R_{\ell},
$$
where $R_{\ell}$ is a function in $G_{1}$ with the stated properties 
in Theorem \ref{thrm-ch4-Linear-MainThm-non-integer-small}  with $k=\ell$.
We write 
$$
u=\sum_{i=0}^{\ell} c_{i} t^{i}
+\sum_{j=0}^{\ell-1-[\overline{m}]}\sum^{\ell-1}_{i=[\overline{m}]+j} c_{i, j} t^{i}t^{\gamma}(\log t)^{j}+R_{\ell}.
$$
By Lemma \ref{lemma-ch4-extensions}, there exist functions $v \in C^{\ell, \alpha}(\bar{B}_{r})$ 
and $w^{(j)}\in  C^{\ell-1, \epsilon}(\bar{B}_{r}) $ for $j=0,\cdots,\ell-[\overline{m}]-2$, such that
\begin{equation}\label{eq-ch4-choice-v-small}
\partial_{t}^{i} v|_{t=0}=i ! c_{i}\quad\text {for } i=0, \cdots, \ell,
\end{equation}
and
\begin{equation}\label{eq-ch4-choice-w-j-small}
\partial_{t}^{i} w^{(j)}|_{t=0}=
\begin{cases}
0&\quad\text{for }i=0,\cdots,[\overline{m}]+j-1,\\
i ! c_{i,j}&\quad\text {for } i=[\overline{m}]+j, \cdots, \ell-1. 
\end{cases}    
\end{equation}
Introduce a function $w=w^{(\ell-1-[\overline{m}])}$ in $G_{1}$ such that
\begin{equation}\label{eq-identity-w-v-non-int-small}
u=v+\sum^{\ell-1-[\overline{m}]}_{j=0}w^{(j)}t^{\gamma}(\log t)^{j}.    
\end{equation}
We will prove $w \in C^{\ell-1, \epsilon}(\bar{G}_{r})$.

We first consider a positive integer $k$ with $k\leq [\overline{m}]$. 
We will prove $\partial^{\nu}_{t}D^{\tau}_{x'}w\in C^{\epsilon}(\bar{G}_{r})$ 
for any nonnegative integers $\nu\leq k-1$ and $\tau\leq \ell-k$. 
By Theorem \ref{thrm-ch3-Linear-NormalEstimate-general} and $k+\alpha<\overline{m}$, 
we get $\partial^{\nu}_{t}D^{\tau}_{x'}u\in C^{\alpha}(\bar{G}_{r})$ for any $\nu\leq k$ and $\tau\leq \ell-\nu$. 
By $v\in C^{\ell,\alpha}(\bar{G}_{r}),$ \eqref{eq-ch4-choice-v-small}, 
Lemma \ref{lemma-ch4-regularity-log}(ii), and Lemma \ref{lemma-ch4-regularity-power2}, 
we have, for any $\nu\leq k-1$ and $\tau\leq \ell-k$,
$$
\partial^{\nu}_{t}D^{\tau}_{x'}\big[t^{-\gamma}(\log t)^{-(\ell-1-[\overline{m}])}(u-v)\big]\in C^{\epsilon}(\bar{G}_{r}).
$$
Similarly, by $w^{(j)}\in  C^{\ell-1, \epsilon}(\bar{G}_{r}) $ for $j=0,\cdots,\ell-[\overline{m}]-2$, 
\eqref{eq-ch4-choice-w-j-small}, and Lemma \ref{lemma-ch4-regularity-log}(ii), we have
$$
\partial^{\nu}_{t}D^{\tau}_{x'}\big[(\log t)^{j-(\ell-1-[\overline{m}])}w^{(j)}\big]\in C^{\epsilon}(\bar{G}_{r}),
$$
for any $\nu\leq k-1$ and $\tau\leq \ell-k$, and $j=0,\cdots,\ell-[\overline{m}]-2$. 
By \eqref{eq-identity-w-v-non-int-small}, we obtain
$$
w=t^{-\gamma}(\log t)^{-(\ell-1-[\overline{m}])}(u-v)
-\sum^{\ell-[\overline{m}]-2}_{j=0}(\log t)^{j-(\ell-1-[\overline{m}])}w^{(j)},
$$
and hence $\partial^{\nu}_{t}D^{\tau}_{x'}w\in C^{\epsilon}(\bar{G}_{r})$ for any $\nu\leq k-1$ and $\tau\leq \ell-k$.

We next fix an integer $k$ with $[\overline{m}]+1 \leq k\leq \ell$. 
We will prove that $\partial_{t}^{\nu} D_{x'}^{\tau} w \in C^{\epsilon}(\bar{G}_{r})$, 
for any $\nu\leq k-1$ and $\tau \leq \ell-k$, in two steps.

{\it Step 1.} We prove that $\partial_{t}^{\nu} D_{x'}^{\tau} w \in C^{\epsilon}_{x'}(\bar{G}_{r})$, 
for any $\nu \leq k-1$ and $\tau \leq \ell-k$.

By Theorem \ref{thrm-ch4-Linear-MainThm-non-integer-small}, we have
\begin{align}\label{eq-ch4-definition-v-k-small-v1}\begin{split}
u&=\sum_{i=0}^{k-1} c_{i} t^{i}
+\sum_{i=[\overline{m}]}^{k-1}\sum^{i-[\overline{m}]}_{j=0} c_{i, j} t^{i+\gamma}(\log t)^{j}+R_{k-1}'\\
&=\sum_{i=0}^{k-1} c_{i} t^{i}
+\sum_{j=0}^{k-1-[\overline{m}]}\sum^{k-1}_{i=[\overline{m}]+j} c_{i, j} t^{i+\gamma}(\log t)^{j}+R_{k-1}',
\end{split}\end{align}
where $R_{k-1}'$ is a function in $G_{1}$ 
with the stated properties in Theorem \ref{thrm-ch4-Linear-MainThm-non-integer-small}. 
By the definition of $v$, we write
\begin{equation}\label{eq-ch4-definition-wj-k-small-v1}
v=\sum_{i=0}^{k-1} c_{i} t^{i}+S'_{k-1}(v), 
\end{equation}
where $S'_{k-1}(v)$ is a function in $\bar{G}_{r}$ such that, 
for any $\nu \leq k-1$, $ \tau \leq \ell-(k-1)$ and any $\nu \leq k$, $ \tau \leq \ell-k$,
\begin{equation}\label{eq-ch4-regularity-S-v-small}
\partial_{t}^{\nu}D_{x'}^{\tau}  S'_{k-1}(v) \in C^{\alpha}(\bar{G}_{r}).  
\end{equation}
Similarly, for $j=0,\cdots,\ell-[\overline{m}]-2$, we write
\begin{equation}\label{eq-ch4-definition-wj-k-small-wj}
w^{(j)}=\sum_{i=[\overline{m}]+j}^{k-1} c_{i,j} t^{i}+S'_{k-1}(w^{(j)}),
\end{equation}
where $S'_{k-1}(w^{(j)})$ is a function in $\bar{G}_{r}$ such that, for any $\nu \leq k-1$ and $ \tau \leq \ell-k$, 
\begin{equation}\label{eq-ch4-regularity-S-wj-small}
\partial_{t}^{\nu}D_{x'}^{\tau}  S'_{k-1}(w^{(j)}) \in C^{\epsilon}(\bar{G}_{r}).   
\end{equation}
Note that $w^{(j)}=S'_{k-1}(w^{(j)})$ for $j=k-[\overline{m}],\cdots,\ell-[\overline{m}]-2$. 
By substituting \eqref{eq-ch4-definition-v-k-small-v1}, \eqref{eq-ch4-definition-wj-k-small-v1}, 
and \eqref{eq-ch4-definition-wj-k-small-wj} in \eqref{eq-identity-w-v-non-int-small}, 
we get, for $[\overline{m}]+1\leq k\leq \ell-1$,
$$
w=(R_{k-1}'-S'_{k-1}(v))t^{-\gamma}(\log t)^{-(\ell-1-[\overline{m}])}
-\sum^{\ell-[\overline{m}]-2}_{j=0}S'_{k-1}(w^{(j)})(\log t)^{j-(\ell-1-[\overline{m}])},
$$
and, for $k=\ell$,
\begin{align*}
w&=c_{\ell-1,\ell-1-[\overline{m}]}t^{\ell-1}
+(R_{\ell-1}'-S'_{\ell-1}(v))t^{-\gamma}(\log t)^{-(\ell-1-[\overline{m}])}\\
&\qquad-\sum^{\ell-[\overline{m}]-2}_{j=0}S'_{\ell-1}(w^{(j)})(\log t)^{j-(\ell-1-[\overline{m}])}.
\end{align*}
Take any $\nu \leq k-1$ and $\tau \leq \ell-k$. 
By \eqref{eq-ch4-regularity-S-v-small}, \eqref{eq-ch4-regularity-S-wj-small}, 
Lemma \ref{lemma-ch4-regularity-log}(ii), and Lemma \ref{lemma-ch4-regularity-power2}(i), we have, 
$$
\partial_{t}^{\nu}D_{x'}^{\tau}\big[S'_{k-1}(v)t^{-\gamma}(\log t)^{-(\ell-1-[\overline{m}])}\big]
\in C^{\epsilon}_{x'}(\bar{G}_{r}),
$$
and for $j=0,\cdots,\ell-[\overline{m}]-2$,
$$
\partial_{t}^{\nu}D_{x'}^{\tau}\big[S'_{k-1}(w^{(j)})(\log t)^{j-(\ell-1-[\overline{m}])}\big]
\in C^{\epsilon}_{x'}(\bar{G}_{r}).
$$
By \eqref{eq-ch4-regularity-coefficients-non-int-small} and with $k=\ell$, 
we have $\partial_{t}^{\nu}D_{x'}^{\tau}(c_{\ell-1,\ell-1-[\overline{m}]}t^{\ell-1})\in C^{\epsilon}_{x'}(\bar{G}_{r})$. By \eqref{eq-ch4-LinearRegularity-non-int-m1-prime-small} 
and Lemma \ref{lemma-ch4-regularity-log}(ii), we get
$$
\partial_{t}^{\nu}D_{x'}^{\tau}\big[R'_{k-1}t^{-\gamma}(\log t)^{-(\ell-1-[\overline{m}])}\big]
\in C^{\epsilon}_{x'}(\bar{G}_{r}).
$$
Therefore, we obtain $\partial_{t}^{\nu}D_{x'}^{\tau}w\in C^{\epsilon}_{x'}(\bar{G}_{r})$. 
We note that the regularity of $\partial_{t}^{\nu}D_{x'}^{\tau}w$ in $x'$ is the desired result. 
%However, the regularity in $t$ is worse. We need to study the regularity in $t$ separately.

{\it Step 2.} We prove that $\partial_{t}^{\nu} D_{x'}^{\tau} w \in C^{\epsilon}_{t}(\bar{G}_{r})$, 
for any $\nu \leq k-1$ and $\tau \leq \ell-k$.

By Theorem \ref{thrm-ch4-Linear-MainThm-non-integer-small}, we have
\begin{align}\label{eq-ch4-definition-v-k-small-v2}\begin{split}
u&=\sum_{i=0}^{k} c_{i} t^{i}
+\sum_{i=[\overline{m}]}^{k-1}\sum^{i-[\overline{m}]}_{j=0} c_{i, j} t^{i+\gamma}(\log t)^{j}+R_{k}\\
&=\sum_{i=0}^{k} c_{i} t^{i}
+\sum_{j=0}^{k-1-[\overline{m}]}\sum^{k-1}_{i=[\overline{m}]+j} c_{i, j} t^{i+\gamma}(\log t)^{j}+R_{k},
\end{split}\end{align}
where $R_{k}$ is a function in $G_{1}$ with the stated properties
in Theorem \ref{thrm-ch4-Linear-MainThm-non-integer-small}. By the definition of $v$, we can write
\begin{equation}\label{eq-ch4-definition-wj-k-small-v2}
v=\sum_{i=0}^{k} c_{i} t^{i}+S_{k}(v), 
\end{equation}
where $S_{k}(v)$ is a function in $\bar{G}_{r}$ such that, for any $\nu \leq k$ and $ \tau \leq \ell-k$,
\begin{equation}\label{eq-ch4-regularity-S-v-small-v2}
\partial_{t}^{\nu}D_{x'}^{\tau}  S_{k}(v) \in C^{\alpha}(\bar{G}_{r}).
\end{equation}
Similarly, for $j=0,\cdots,\ell-[\overline{m}]-2$, we can write
\begin{equation}\label{eq-ch4-definition-wj-k-small-wj-w2}
w^{(j)}=\sum_{i=[\overline{m}]+j}^{k-1} c_{i,j} t^{i}+S_{k-1}(w^{(j)}),
\end{equation}
where $S_{k-1}(w^{(j)})$ is a function in $\bar{G}_{r}$ such that, for any $\nu \leq k-1$ and $ \tau \leq \ell-k$,
\begin{equation}\label{eq-ch4-regularity-S-wj-small-w2}
\partial_{t}^{\nu}D_{x'}^{\tau}  S_{k-1}(w^{(j)}) \in C^{\epsilon}(\bar{G}_{r}).
\end{equation}
Note that $w^{(j)}=S_{k-1}(w^{(j)})$ for $j=k-[\overline{m}],\cdots,\ell-[\overline{m}]-2$. 
By substituting \eqref{eq-ch4-definition-v-k-small-v2}, \eqref{eq-ch4-definition-wj-k-small-v2}, 
and \eqref{eq-ch4-definition-wj-k-small-wj-w2} in \eqref{eq-identity-w-v-non-int-small}, 
we get, for $[\overline{m}]+1\leq k\leq \ell-1$,
$$
w=(R_{k}-S_{k}(v))t^{-\gamma}(\log t)^{-(\ell-1-[\overline{m}])}
-\sum^{\ell-[\overline{m}]-2}_{j=0}S_{k-1}(w^{(j)})(\log t)^{j-(\ell-1-[\overline{m}])},
$$
and, for $k=\ell$,
\begin{align*}
w&=c_{\ell-1,\ell-1-[\overline{m}]}t^{\ell-1}+(R_{\ell}-S_{\ell}(v))t^{-\gamma}(\log t)^{-(\ell-1-[\overline{m}])}\\
&\qquad-\sum^{\ell-[\overline{m}]-2}_{j=0}S_{\ell-1}(w^{(j)})(\log t)^{j-(\ell-1-[\overline{m}])}.
\end{align*}
Take any $\nu \leq k-1$ and $\tau \leq \ell-k$. 
By \eqref{eq-ch4-regularity-S-v-small-v2}, \eqref{eq-ch4-regularity-S-wj-small-w2}, 
Lemma \ref{lemma-ch4-regularity-log}(ii), and Lemma \ref{lemma-ch4-regularity-power2}(ii), we have, 
$$
\partial_{t}^{\nu}D_{x'}^{\tau}\big[S_{k}(v)t^{-\gamma}(\log t)^{-(\ell-1-[\overline{m}])}\big]
\in C^{\epsilon}_{t}(\bar{G}_{r}),
$$
and for $j=0,\cdots,\ell-[\overline{m}]-2$, 
$$
\partial_{t}^{\nu}D_{x'}^{\tau}\big[S_{k-1}(w^{(j)})(\log t)^{j-(\ell-1-[\overline{m}])}\big]
\in C^{\epsilon}_{t}(\bar{G}_{r}).
$$
By \eqref{eq-ch4-LinearRegularity-non-int-m1-small}, 
Lemma \ref{lemma-ch4-regularity-log}(ii), and Lemma \ref{lemma-ch4-regularity-power2}(ii), we have
$$
\partial_{t}^{\nu}D_{x'}^{\tau}\big[R_{k}t^{-\gamma}(\log t)^{-(\ell-1-[\overline{m}])}\big]\in C^{\epsilon}_{t}(\bar{G}_{r}).
$$
Therefore, we obtain $\partial_{t}^{\nu}D_{x'}^{\tau}w\in C^{\epsilon}_{t}(\bar{G}_{r})$. 
\end{proof}

By using Theorem \ref{thrm-ch4-Linear-MainThm-non-integer-small}, 
we can prove Theorem \ref{thrm-ch4-Decomposition-infinite-non-int} 
without the additional assumption $[\overline{m}]\ge 1$.  
Then, Theorem \ref{Thm-ch4-Linear-MainThm-curved-combined-smooth}
follows  by renaming the functions $w^{(j)}$, for $j\ge 0$.

\section{Appendix A: H\"older Continuous Functions}\label{sec-Appen-CalculusL}

In this section, we prove several results concerning 
the H\"older continuity of functions in integral forms.
These results play important roles in earlier sections. In this section, we fix an $r\in(0,1)$.

Throughout this section, we will use the following estimates repeatedly. 
Let $\gamma_1$ and $\gamma_2$ be two positive constants such that 
$c_1\le \gamma_1, \gamma_2\le c_2$ for some constants $c_1, c_2>0$. 
Then, for any $s\in (0,1)$, 
$$|s^{\gamma_1}-s^{\gamma_2}|\le s^{c_1}|\log s|\,|\gamma_1-\gamma_2|,$$
and 
$$|s^{-\gamma_1}-s^{-\gamma_2}|\le s^{-c_2}|\log s|\,|\gamma_1-\gamma_2|.$$
Obviously, if $\gamma_1=\gamma_2$, these inequalities become trivial. 

\begin{lemma}\label{lemma-ch4-BasicHolderRegularity}
Let $k$ and $l$ be nonnegative integers, $\alpha,\beta\in(0,1)$ be constants, 
and  $a\in C^{l,\beta}(\bar{B}'_{r})$ and $f \in C(\bar{G}_{r})$ be functions, 
with $a\geq a_{0}$ in $\bar{B}'_{r}$ for some constant $a_{0}>0$. 
Define, for any $\left(x^{\prime}, t\right) \in G_{r}$,
$$
F(x^{\prime}, t)=\frac{1}{t^{a(x^{\prime})}} \int_{0}^{t} s^{a(x^{\prime})-1} f(x^{\prime}, s) d s.
$$
Suppose $\partial_{t}^{\nu} D_{x^{\prime}}^{\tau} f \in C^{\beta,\alpha}_{x',t}(\bar{G}_{r})$, 
for any $\tau \leq l$ and $\nu \leq k.$ 
Then, $\partial_{t}^{\nu} D_{x^{\prime}}^{\tau} F \in C^{\beta,\alpha}_{x',t}(\bar{G}_{r})$, 
for any $\tau \leq l$ and $\nu \leq k$. Moreover, if $\partial^{\nu}_{t}D^{\tau}_{x'}f(\cdot,0)=0$, 
for any $\tau \leq l$ and $\nu \leq k$, then, for any $\tau \leq l$ and $\nu \leq k$, 
and any $\left(x^{\prime}, t\right) \in G_{r}$,
$$
|\partial_{t}^{\nu} D_{x^{\prime}}^{\tau} F(x^{\prime}, t)| 
\leq C\sum^{\tau}_{i=0}[\partial_{t}^{k} D_{x^{\prime}}^{i} f]_{C^{\alpha}_{t}(\bar{G}_{r})} t^{k-\nu+\alpha},
$$
where $C$ is a positive constant depending only on $a_{0}$, $k$, $\alpha$, 
and the $C^{l}$-norm of $a$ in $\bar B'_r$. 
\end{lemma}

\begin{proof}
We first consider $l=0$. By a change of variables $s=t\rho$, we have, for any $(x',t)\in G_{r}$,
$$
F(x',t)=\int^{1}_{0}f(x',t\rho)\rho^{a(x')-1}d\rho.
$$
Then, 
$$
|F(x',t)|\leq |f|_{L^{\infty}(G_{r})}\int^{1}_{0}\rho^{a(x')-1}d\rho
=\frac{1}{a(x')}|f|_{L^{\infty}(G_{r})}\leq \frac{1}{a_{0}}|f|_{L^{\infty}(G_{r})}.
$$

Next, for any $(x'_{1},t),(x'_{2},t)\in G_{r}$, we write 
$$F(x'_{1},t)-F(x'_{2},t)=I_1+I_2,$$
where 
\begin{align*}
I_1&=\int^{1}_{0}\big[f(x_1',t\rho)-f(x_2',t\rho)\big]\rho^{a(x_1')-1}d\rho,\\
I_2&=\int^{1}_{0}f(x_2',t\rho)\big[\rho^{a(x_1')-1}-\rho^{a(x_2')-1}\big]d\rho.
\end{align*}
First, we have, for any $\rho\in(0,1]$,
$$
|f(x'_{1},t\rho)-f(x'_{2},t\rho)|
\leq [f]_{C^{\beta}_{x'}(\bar{G}_{r})}|x'_{1}-x'_{2}|^{\beta},
$$
and
\begin{align*}
\big|\rho^{a(x'_{1})-1}-\rho^{a(x'_{2})-1}\big|&\leq \rho^{a_{0}-1}|\log\rho||a(x'_{1})-a(x'_{2})|\\
&\leq  \rho^{a_{0}-1}|\log\rho|[a]_{C^{\beta}(\bar{B}'_{r})}|x'_{1}-x'_{2}|^{\beta}.
\end{align*}
Then,
\begin{align*}
|I_1|&\leq\int^{1}_{0}|f(x'_{1},t\rho)-f(x'_{2},t\rho)|\rho^{a(x'_{1})-1}d\rho\\
&\leq  [f]_{C^{\beta}_{x'}(\bar{G}_{r})}|x'_{1}-x'_{2}|^{\beta}\int^{1}_{0}\rho^{a_{0}-1}d\rho\\
&\leq  C[f]_{C^{\beta}_{x'}(\bar{G}_{r})}|x'_{1}-x'_{2}|^{\beta},\end{align*}
and 
\begin{align*}
|I_2|&\le \int^{1}_{0}|f(x'_{2},t\rho)||\rho^{a(x'_{1})-1}-\rho^{a(x'_{2})-1}|d\rho\\
&\le|f|_{L^{\infty}(G_{r})}[a]_{C^{\beta}(\bar{B}'_{r})}|x'_{1}-x'_{2}|^{\beta}\int^{1}_{0}\rho^{a_{0}-1}|\log \rho|d\rho\\
&\le C|f|_{L^{\infty}(G_{r})}|x'_{1}-x'_{2}|^{\beta}.
\end{align*}
Thus,
\begin{align*}
|F(x'_{1},t)-F(x'_{2},t)|\leq 
C\big\{[f]_{C^{\beta}_{x'}(\bar{G}_{r})}+|f|_{L^{\infty}(G_{r})}\big\}|x'_{1}-x'_{2}|^{\beta},
\end{align*}
where $C$ is a positive constant depending only on $a_{0}$
and the $C^{\beta}$-norm of $a$ in $\bar B'_r$. 
Hence, $F\in C^{\beta}_{x'}(\bar{G}_{r})$. 

Similarly, for any $(x',t_{1}),(x',t_{2})\in G_{r}$, we have, for any $\rho\in[0,1]$,
$$
|f(x',t_{1}\rho)-f(x',t_{2}\rho)|\leq [f]_{C^{\alpha}_{t}(\bar{G}_{r})}\rho^{\alpha}|t_{1}-t_{2}|^{\alpha}.
$$
Then,
\begin{align*}
|F(x',t_{1})-F(x',t_{2})|&\leq  [f]_{C^{\alpha}_{t}(\bar{G}_{r})}|t_{1}-t_{2}|^{\alpha}
\int^{1}_{0}\rho^{a(x')+\alpha-1}d\rho\\
&\leq \frac{1}{a_{0}+\alpha}[f]_{C^{\alpha}_{t}(\bar{G}_{r})} |t_{1}-t_{2}|^{\alpha}.
\end{align*}
Hence, $F\in C^{\alpha}_{t}(\bar{G}_{r})$. 

Take any nonnegative integer $\nu$ with $\nu\leq k$. Then, for any $(x',t)\in G_{r}$,
$$
\partial^{\nu}_{t}F(x',t)=\int^{1}_{0}\partial^{\nu}_{t}f(x',t\rho)\rho^{a(x')-1+\nu}d\rho.
$$
Hence, $\partial^{\nu}_{t}F\in C^{\beta,\alpha}_{x',t}(\bar{G}_{r})$. 
If $\partial^{\nu}_{t}f(x',0)=0$ for any $\nu \leq k$, then
\begin{align*}
|\partial_{t}^{\nu} F(x^{\prime}, t)| 
&\leq [\partial_{t}^{k} f]_{C^{\alpha}_{t}(\bar{G}_{r})}t^{k-\nu+\alpha}\int^{1}_{0}\rho^{k+a(x')-1+\alpha}d\rho\\
&\leq \frac{1}{k+a_{0}+\alpha}[\partial_{t}^{k} f]_{C^{\alpha}_{t}(\bar{G}_{r})}t^{k-\nu+\alpha}.
\end{align*}
This yields the desired result for $l=0$. 

If $l=1$, we have
$$
D_{x'}F(x',t)=\int^{1}_{0}D_{x'}f(x',t\rho)\rho^{a(x^{\prime})-1}d\rho
+\int^{1}_{0}f(x',t\rho)\rho^{a(x^{\prime})-1}\log \rho d\rho\cdot  D_{x'}a(x^{\prime}).
$$
We can proceed similarly as above. Therefore, we can obtain the desired result for any $l\geq 0$.
\end{proof}

%\textcolor{red}{
%We now make a remark. For the general case that $a$ is a function of $x'\in \bar B'_r$, 
%the $x'$-derivative of $\rho^{a-1}$ yields a singular factor $\log\rho$. 
%However, such a singular factor does not cause the loss of H\"older index 
%since it is coupled with an integrable monomial.  
%}

\begin{lemma}\label{lemma-ch4-singular-integral-non-int-lower-larger}
Let $l$ be a nonnegative integer, $\alpha\in(0,1)$ be a constant, 
and $a\in C^{l,\alpha}(\bar{B}'_{r})$ and $f\in C(\bar{G}_{r})$ be functions, 
with $a>0$ and $0< a-[a]\leq c<1$ in $\bar{B}'_{r}$ for some constant $c>0$. Define, for any $(x',t)\in G_{r}$,
$$
F(x',t)=t^{a(x^{\prime})}\int^{t}_{0}s^{-a(x^{\prime})-1}f(x',s)ds.
$$
Suppose $\alpha\in(c,1)$, $\partial^{\nu}_{t}D^{\tau}_{x'}f\in C^{\alpha}(\bar{G}_{r})$ 
and $\partial^{\nu}_{t}D^{\tau}_{x'}f(\cdot,0)=0$, for any $\tau\leq l$ and $\nu\leq [a]$. 
Then, $\partial^{\nu}_{t}D^{\tau}_{x'}F\in C^{\epsilon,\alpha}_{x',t}(\bar{G}_{r})$, 
for any $\tau\leq l$ and $\nu\leq [a]$, and any $\epsilon\in (0,\alpha-c)$. 
Moreover, for any $\tau \leq l$ and $\nu \leq [a]$, and any $(x^{\prime}, t) \in G_{r}$,
$$
|\partial_{t}^{\nu} D_{x^{\prime}}^{\tau} F(x^{\prime}, t)| 
\leq C\sum^{\tau}_{i=0}\big[\partial_{t}^{[a]} D_{x^{\prime}}^{i} f\big]_{C^{\alpha}_{t}(\bar{G}_{r})} t^{[a]-\nu+\alpha},
$$
where $C$ is a positive constant depending only on $c$, $\alpha$, 
and the $C^{l}$-norm of $a$ in $\bar B'_r$. 
\end{lemma}

\begin{proof}
%For simplicity, we assume that $a-[a]\equiv c_{1}\equiv c_{2}$. The %discussion of the general case is similar.
We only consider $l=0$.
First, we derive decay estimates of $F$ and its derivatives. By a change of variables $s=t \rho$, we have
$$
F(x^{\prime}, t)=\int_{0}^{1} \rho^{-a(x^{\prime})-1} f(x^{\prime}, t \rho) d \rho,
$$
and hence, for any $0\leq \nu \leq [a]$,
$$
\partial_{t}^{\nu} F(x^{\prime}, t)=\int_{0}^{1} \rho^{\nu-a(x')-1} \partial_{t}^{\nu} f(x^{\prime}, t \rho) d \rho.
$$
Now, we fix a $0\leq \nu \leq [a]$. By
$$
|\partial_{t}^{\nu} f(x^{\prime}, t)| \leq\big[\partial_{t}^{[a]} f\big]_{C^{\alpha}_{t}(\bar{G}_{r})} t^{[a]-\nu+\alpha},
$$
we have, for any $(x^{\prime}, t) \in G_{r}$,
\begin{align*}
|\partial_{t}^{\nu} F(x^{\prime}, t)| 
& \leq\big[\partial_{t}^{[a]} f\big]_{C^{\alpha}_{t}(\bar{G}_{r})} t^{[a]-\nu+\alpha} 
\int_{0}^{1} \rho^{[a]+\alpha-a(x^{\prime})-1} d \rho \\
&\leq\frac{1}{\alpha-c}\big[\partial_{t}^{[a]} f\big]_{C^{\alpha}_{t}(\bar{G}_{r})} t^{[a]-\nu+\alpha},
\end{align*}
where we used $[a]+\alpha-a\geq \alpha-c>0$. 

For the H\"older semi-norm of $\partial_{t}^{\nu} F$ in $t$, it suffices to consider the case $\nu = [a].$ 
For any $(x^{\prime}, t_{1}), (x^{\prime}, t_{2})\in G_{r}$, we have
$$
\big|\partial_{t}^{[a]} f(x^{\prime}, t_{1} \rho)-\partial_{t}^{[a]} f(x^{\prime}, t_{2} \rho)\big| 
\leq\big[\partial_{t}^{[a]} f\big]_{C^{\alpha}_{t}(\bar{G}_{r})} \rho^{\alpha}\left|t_{1}-t_{2}\right|^{\alpha}.
$$
Then,
\begin{align*}
\big|\partial_{t}^{[a]} F(x^{\prime}, t_{1})-\partial_{t}^{[a]} F(x^{\prime}, t_{2})\big| 
& \leq\big[\partial_{t}^{[a]} f\big]_{C^{\alpha}_{t}(\bar{G}_{r})}|t_{1}-t_{2}|^{\alpha} 
\int_{0}^{1} \rho^{[a]+\alpha-a(x^{\prime})-1} d \rho \\
&\leq \frac{1}{\alpha-c}\big[\partial_{t}^{[a]} f\big]_{C^{\alpha}_{t}(\bar{G}_{r})}|t_{1}-t_{2}|^{\alpha}.
\end{align*}
Hence, $\partial_{t}^{[a]} F\in C^{\alpha}_{t}(\bar{G}_{r})$.

Next, we consider the H\"{o}lder semi-norm of $\partial_{t}^{\nu} F$ in $x'.$ 
Take any $(x_{1}^{\prime}, t)$, $(x_{2}^{\prime}, t)\in G_{r}$. Fix a $\nu \leq [a]$
and write 
$$\partial_{t}^{\nu} F(x_1^{\prime}, t)-\partial_{t}^{\nu} F(x_2^{\prime}, t)=I_1+I_2,$$
where 
\begin{align*}
I_1&=\int_{0}^{1} \rho^{\nu-a(x_1')-1} \big[\partial_{t}^{\nu} f(x_1^{\prime}, t \rho)
-\partial_{t}^{\nu} f(x_2^{\prime}, t \rho)\big] d \rho,\\
I_2&=\int_{0}^{1} \partial_{t}^{\nu} f(x_2^{\prime}, t \rho)[\rho^{\nu-a(x_1')-1} -\rho^{\nu-a(x_2')-1}] d \rho.
\end{align*}
We first study $I_1$. Note that for $\nu\leq [a]-1$,
$$
\partial_{t}^{\nu} f(x_{1}^{\prime}, t \rho)-\partial_{t}^{\nu} f(x_{2}^{\prime}, t \rho)
=\rho t \int_{0}^{1}\left(\partial_{t}^{\nu+1} f(x_{1}^{\prime}, s t \rho)
-\partial_{t}^{\nu+1} f(x_{2}^{\prime}, s t \rho)\right) d s.
$$
By iterating this formula finitely many times, we have
$$
|\partial_{t}^{\nu} f(x_{1}^{\prime}, t \rho)-\partial_{t}^{\nu} f(x_{2}^{\prime}, t \rho)| 
\leq(\rho t)^{[a]-\nu} \sup _{s \in[0,1]}\big|\partial_{t}^{[a]} f(x_{1}^{\prime}, s t \rho)
-\partial_{t}^{[a]} f(x_{2}^{\prime}, s t \rho)\big|,
$$
and hence
\begin{align*}
|I_1|&\leq\int^{1}_{0}\rho^{\nu-a(x'_{1})-1}|\partial_{t}^{\nu} f(x_{1}^{\prime}, t \rho)
-\partial_{t}^{\nu} f(x_{2}^{\prime}, t \rho)| d\rho\\
& \leq t^{[a]-\nu} \int_{0}^{1} \rho^{-c-1} \sup _{s \in[0,1]}\big|\partial_{t}^{[a]} f(x_{1}^{\prime}, s t \rho)
-\partial_{t}^{[a]} f(x_{2}^{\prime}, s t \rho)\big| d \rho.
\end{align*}
For any $\delta\in(0,1)$, by writing $|I-I\!\!I|=|I-I\!\!I|^{\delta}\cdot|I-I\!\!I|^{1-\delta}$, we have
\begin{align*}
&|\partial_{t}^{[a]} f (x_{1}^{\prime}, t \rho)-\partial_{t}^{[a]} f(x_{2}^{\prime}, t \rho)|\\
&\qquad \leq(|\partial_{t}^{[a]} f(x_{1}^{\prime}, t \rho)|
+|\partial_{t}^{[a]} f(x_{2}^{\prime}, t \rho)|)^{\delta}
|\partial_{t}^{[a]} f(x_{1}^{\prime}, t \rho)-\partial_{t}^{[a]} f(x_{2}^{\prime}, t \rho)|^{1-\delta} \\
&\qquad \leq 2[\partial_{t}^{[a]} f]^{\delta}_{C^{\alpha}_{t}(\bar{G}_{r})}
[\partial_{t}^{[a]} f]^{1-\delta}_{C^{\alpha}_{x'}(\bar{G}_{r})} 
t^{\delta \alpha} \rho^{\delta \alpha}\left|x_{1}^{\prime}-x_{2}^{\prime}\right|^{(1-\delta) \alpha}.
\end{align*}
Hence,
\begin{align*}
|I_1|
&\leq 2[\partial_{t}^{[a]} f]^{\delta}_{C^{\alpha}_{t}(\bar{G}_{r})}
[\partial_{t}^{[a]} f]^{1-\delta}_{C^{\alpha}_{x'}(\bar{G}_{r})} 
r^{\delta \alpha}|x_{1}^{\prime}-x_{2}^{\prime}|^{(1-\delta) \alpha} \int_{0}^{1} \rho^{-1-c+\delta \alpha} d \rho \\
&\leq C[\partial_{t}^{[a]} f]^{\delta}_{C^{\alpha}_{t}(\bar{G}_{r})}
[\partial_{t}^{[a]} f]^{1-\delta}_{C^{\alpha}_{x'}(\bar{G}_{r})} 
r^{\delta \alpha}|x_{1}^{\prime}-x_{2}^{\prime}|^{(1-\delta) \alpha},
\end{align*}
if $-c+\delta\alpha>0$. For $I_2$, we get 
\begin{align*}
|I_2|
&\le 
\int^{1}_{0}|\partial_{t}^{\nu} f(x_{2}^{\prime}, t \rho)||\rho^{\nu-a(x'_{1})-1}-\rho^{\nu-a(x'_{2})-1}|d\rho\\
&\le\int^{1}_{0}\rho^{\nu-[a]-c-1}|\log\rho||\partial_{t}^{\nu} f(x_{2}^{\prime}, t \rho)||a(x'_{1})-a(x'_{2})|d\rho.
\end{align*}
By 
$$
|\partial_{t}^{\nu} f(x^{\prime}_{2}, t\rho)| 
\leq\big[\partial_{t}^{[a]} f\big]_{C^{\alpha}_{t}(\bar{G}_{r})} (t\rho)^{[a]-\nu+\alpha},
$$
we have
\begin{align*}
|I_2|
&\leq [a]_{C^{\alpha}(\bar{B}'_{r})}\big[\partial_{t}^{[a]} f\big]_{C^{\alpha}_{t}(\bar{G}_{r})}r^{\alpha}
|x'_{1}-x'_{2}|^{\alpha}\int^{1}_{0}\rho^{-1-c+\alpha}|\log\rho|d\rho\\
&\leq C[a]_{C^{\alpha}(\bar{B}'_{r})}\big[\partial_{t}^{[a]} f\big]_{C^{\alpha}_{t}(\bar{G}_{r})}r^{\alpha}
|x'_{1}-x'_{2}|^{\alpha}.
\end{align*}
Therefore, 
$$|\partial_{t}^{\nu} F(x_1^{\prime}, t)-\partial_{t}^{\nu} F(x_2^{\prime}, t)|
\le CK|x_{1}^{\prime}-x_{2}^{\prime}|^{(1-\delta) \alpha},$$
if $-c+\delta\alpha>0$, where
$$
K=[\partial_{t}^{[a]} f]^{\delta}_{C^{\alpha}_{t}(\bar{G}_{r})}[\partial_{t}^{[a]} f]^{1-\delta}_{C^{\alpha}_{x'}(\bar{G}_{r})}
+\big[\partial_{t}^{[a]} f\big]_{C^{\alpha}_{t}(\bar{G}_{r})},
$$
and $C$ is a positive constant depending only on $r$, $\delta$, $c$, $\alpha$, 
and the $C^{\alpha}$-norm of $a$ in $\bar B'_r$. 
For any $\epsilon \in(0, \alpha-c),$ take $\delta$ such that $(1-\delta)\alpha=\epsilon$. 
Then, $0<\delta<1$ and $-c+\delta\alpha>0$, and hence $\partial_{t}^{\nu} F \in C^{\epsilon}_{x'}(\bar{G}_{r})$. 
This holds for any $\nu \leq [a]$. 
\end{proof} 

\begin{lemma}\label{lemma-ch4-singular-integral-non-int-lower-smaller}
Let $l$ be a nonnegative integer, $\alpha\in(0,1)$ be a constant, 
and $a\in C^{l}(\bar{B}'_{r})$ and $f\in C(\bar{G}_{r})$ be functions, 
with $a>0$ and $0< a-[a]\leq c<1$ in $\bar{B}'_{r}$ for some constant $c>0$. Define, for any $(x',t)\in G_{r}$,
$$
F(x',t)=t^{a(x^{\prime})}\int^{t}_{0}s^{-a(x^{\prime})-1}f(x',s)ds.
$$
Suppose $\alpha\in(0,c)$, $l\geq 1$, $\partial^{\nu}_{t}D^{\tau}_{x'}f\in C^{\alpha}(\bar{G}_{r})$, 
for any $\tau\leq l$, $\nu\leq [a]$ and any $\tau\leq l-1$, $\nu\leq [a]+1$, 
and $\partial^{\nu}_{t}D^{\tau}_{x'}f(\cdot,0)=0$, for any $\tau\leq l$ and $\nu\leq [a]$. 
Then, $\partial^{\nu}_{t}D^{\tau}_{x'}F\in C^{\epsilon}_{x'}(\bar{G}_{r})$, for any $\tau\leq l-1$ and $\nu\leq [a]$, 
and any $\epsilon\in (0,\alpha+1-c)$, and $\partial^{\nu}_{t}D^{\tau}_{x'}F\in C^{\alpha}_{t}(\bar{G}_{r})$, 
for any $\tau\leq l-1$ and $\nu\leq [a]+1$. 
Moreover, for any $\tau \leq l-1$ and $\nu \leq [a]+1$, and any $(x^{\prime}, t) \in G_{r}$,
$$
|\partial_{t}^{\nu} D_{x^{\prime}}^{\tau} F(x^{\prime}, t)| 
\leq C\sum^{\tau}_{i=0}\big|\partial_{t}^{[a]+1}D^{i}_{x'} f\big|_{L^{\infty}(G_{r})} t^{[a]-\nu+1},
$$
where $C$ is a positive constant depending only on $c$, $\alpha$,
and the $C^{l-1}$-norm of $a$ in $\bar B'_r$. 
\end{lemma}

\begin{proof} 
We only consider $l=1$.
First, we derive an estimate of $F$. By a change of variables $s=t \rho$, we have
$$
F(x^{\prime}, t)=\int_{0}^{1} \rho^{-a(x^{\prime})-1} f(x^{\prime}, t \rho) d \rho,
$$
and hence, for any $0\leq \nu \leq [a]+1$,
$$
\partial_{t}^{\nu} F(x^{\prime}, t)=\int_{0}^{1} \rho^{\nu-a(x^{\prime})-1} \partial_{t}^{\nu} f(x^{\prime}, t \rho) d \rho.
$$
Now, we fix a $ \nu \leq [a]+1$. By
$$
|\partial_{t}^{\nu} f(x^{\prime}, t)| \leq|\partial_{t}^{[a]+1} f|_{L^{\infty}(G_{r})} t^{[a]-\nu+1},
$$
we have, for any $(x^{\prime}, t) \in G_{r}$,
\begin{align*}
|\partial_{t}^{\nu} F(x^{\prime}, t)| 
& \leq|\partial_{t}^{[a]+1} f|_{L^{\infty}(G_{r})} t^{[a]-\nu+1} \int_{0}^{1} \rho^{[a]-a(x^{\prime})} d \rho \\
&\leq \frac{1}{1-c}|\partial_{t}^{[a]+1} f|_{L^{\infty}(G_{r})} t^{[a]-\nu+1}.
\end{align*}

For the H\"older semi-norm of $\partial_{t}^{\nu} F$ in $t$, it suffices to consider the case $\nu = [a]+1.$ 
For any $(x^{\prime}, t_{1}), (x^{\prime}, t_{2})\in G_{r}$, we have
$$
|\partial_{t}^{[a]+1} f(x^{\prime}, t_{1} \rho)-\partial_{t}^{[a]+1} f(x^{\prime}, t_{2} \rho)| 
\leq[\partial_{t}^{[a]+1} f]_{C^{\alpha}_{t}(\bar{G}_{r})} \rho^{\alpha}|t_{1}-t_{2}|^{\alpha}.
$$
Note that $[a]+\alpha-a> -1$, so we have
\begin{align*}
|\partial_{t}^{[a]+1} F(x^{\prime}, t_{1})-\partial_{t}^{[a]+1} F(x^{\prime}, t_{2})| 
& \leq[\partial_{t}^{[a]+1} f]_{C^{\alpha}_{t}(\bar{G}_{r})}|t_{1}-t_{2}|^{\alpha} 
\int_{0}^{1} \rho^{[a]+\alpha-a(x^{\prime})} d \rho \\
&\leq\frac{1}{\alpha-c+1}[\partial_{t}^{[a]+1} f]_{C^{\alpha}_{t}(\bar{G}_{r})}|t_{1}-t_{2}|^{\alpha}.
\end{align*}
Hence, $\partial_{t}^{[a]+1} F \in C^{\alpha}_{t}(\bar{G}_{r})$.

Next, we consider the H\"{o}lder semi-norm of $\partial_{t}^{\nu} F$ in $x^{\prime}.$ 
Take any $(x_{1}^{\prime}, t)$, $(x_{2}^{\prime}, t)\in G_{r}$. 
Fix a $\nu \leq [a]$ and write 
$$\partial_{t}^{\nu} F(x_1^{\prime}, t)-\partial_{t}^{\nu} F(x_2^{\prime}, t)=I_1+I_2,$$
where 
\begin{align*}
I_1&=\int_{0}^{1} \rho^{\nu-a(x_1')-1} \big[\partial_{t}^{\nu} f(x_1^{\prime}, t \rho)
-\partial_{t}^{\nu} f(x_2^{\prime}, t \rho)\big] d \rho,\\
I_2&=\int_{0}^{1} \partial_{t}^{\nu} f(x_2^{\prime}, t \rho)[\rho^{\nu-a(x_1')-1} -\rho^{\nu-a(x_2')-1}] d \rho.
\end{align*}
We first study $I_1$. 
Note
$$
\partial_{t}^{\nu} f(x_{1}^{\prime}, t \rho)-\partial_{t}^{\nu} f(x_{2}^{\prime}, t \rho)
=\rho t \int_{0}^{1}\big(\partial_{t}^{\nu+1} f(x_{1}^{\prime}, s t \rho)-\partial_{t}^{\nu+1} f(x_{2}^{\prime}, s t \rho)\big) d s.
$$
By iterating this formula finitely many times, we have
\begin{align*}
|\partial_{t}^{\nu} f(x_{1}^{\prime}, t \rho)-\partial_{t}^{\nu} f(x_{2}^{\prime}, t \rho)|
\leq(\rho t)^{[a]-\nu} \sup _{s \in[0,1]}\big|\partial_{t}^{[a]} f(x_{1}^{\prime}, s t \rho)
-\partial_{t}^{[a]} f(x_{2}^{\prime}, s t \rho)\big|,
\end{align*}
and hence
\begin{align*}
|I_1|&\leq\int^{1}_{0}\rho^{\nu-a(x'_{1})-1}
|\partial_{t}^{\nu} f(x_{1}^{\prime}, t \rho)-\partial_{t}^{\nu} f(x_{2}^{\prime}, t \rho)| d\rho\\
& \leq t^{[a]-\nu} \int_{0}^{1} \rho^{-c-1} \sup _{s \in[0,1]}
\big|\partial_{t}^{[a]} f(x_{1}^{\prime}, s t \rho)-\partial_{t}^{[a]} f(x_{2}^{\prime}, s t \rho)\big| d \rho.
\end{align*}
First, we have
\begin{align*}
|\partial_{t}^{[a]} f(x_{1}^{\prime}, t \rho)-\partial_{t}^{[a]} f(x_{2}^{\prime}, t \rho)|
&\leq  \rho t \sup _{s \in[0,1]}|\partial_{t}^{[a]+1} f(x_{1}^{\prime}, s t \rho)-\partial_{t}^{[a]+1} f(x_{2}^{\prime}, s t \rho)|\\
&\leq \rho t [\partial_{t}^{[a]+1} f]_{C^{\alpha}_{x'}(\bar{G}_{r})} |x_{1}^{\prime}-x_{2}^{\prime}|^{\alpha}.
\end{align*}
Next, we also get
\begin{align*}
|\partial_{t}^{[a]} f(x_{1}^{\prime}, t \rho)-\partial_{t}^{[a]} f(x_{2}^{\prime}, t \rho)|
&=\Big|\int^{1}_{0}\partial_{t}^{[a]}D_{x'} f(s x_{1}'+(1-s) x_{2}', t \rho)\cdot (x_{1}'-x_{2}')ds\Big|\\
&\leq [\partial_{t}^{[a]}D_{x'} f]_{C^{\alpha}_{t}(\bar{G}_{r})}t^{\alpha}\rho^{\alpha}|x_{1}^{\prime}-x_{2}^{\prime}|.
\end{align*}
For any $\delta\in(0,1)$, by writing $|I-I\!\!I|=|I-I\!\!I|^{\delta}\cdot|I-I\!\!I|^{1-\delta}$, we have
\begin{align*}
&\big| \partial_{t}^{[a]} f(x_{1}^{\prime}, t \rho)-\partial_{t}^{[a]} f(x_{2}^{\prime}, t \rho) \big| \\
&\qquad \leq [\partial_{t}^{[a]+1} f]_{C^{\alpha}_{x'}(\bar{G}_{r})}^{\delta} 
t^{\delta} \rho^{\delta}|x_{1}^{\prime}-x_{2}^{\prime}|^{\alpha\delta}\\
&\qquad\qquad\cdot[\partial_{t}^{[a]}D_{x'} f]^{1-\delta}_{C^{\alpha}_{t}(\bar{G}_{r})}
t^{(1-\delta)\alpha}\rho^{(1-\delta)\alpha}\left|x_{1}^{\prime}-x_{2}^{\prime}\right|^{1-\delta}\\
&\qquad\leq CK_1t^{\delta+(1-\delta)\alpha}\rho^{\delta+(1-\delta)\alpha}
|x_{1}^{\prime}-x_{2}^{\prime}|^{1+\alpha\delta-\delta},
\end{align*}
where 
$$K_1=[\partial_{t}^{[a]+1} f]_{C^{\alpha}_{x'}(\bar{G}_{r})}^{\delta}
[\partial_{t}^{[a]}D_{x'} f]^{1-\delta}_{C^{\alpha}_{t}(\bar{G}_{r})}.$$
Thus,
\begin{align*}
|I_1|& \leq CK_1r^{\delta +(1-\delta)\alpha}
|x_{1}^{\prime}-x_{2}^{\prime}|^{1+\alpha\delta-\delta} \int_{0}^{1} \rho^{-1-c+\delta+(1-\delta) \alpha} d \rho \\
& \leq CK_1r^{\delta +(1-\delta)\alpha}|x_{1}^{\prime}-x_{2}^{\prime}|^{1+\alpha\delta-\delta},
\end{align*}
if $-c+\delta+(1-\delta) \alpha> 0$. For $I_2$, we have 
\begin{align*}
|I_2|
&\le\int^{1}_{0}|\partial_{t}^{\nu} f(x_{2}^{\prime}, t \rho)||\rho^{\nu-a(x'_{1})-1}-\rho^{\nu-a(x'_{2})-1}|d\rho\\
&\le\int^{1}_{0}\rho^{\nu-[a]-c-1}|\log\rho||\partial_{t}^{\nu} f(x_{2}^{\prime}, t \rho)||a(x'_{1})-a(x'_{2})|d\rho.
\end{align*}
By
$$
|\partial_{t}^{\nu} f(x^{\prime}_{2}, t\rho)| \leq\big|\partial_{t}^{[a]+1} f\big|_{L^{\infty}(G_{r})}(t\rho)^{[a]+1-\nu},
$$
we get
\begin{align*}
|I_2| 
&\le  \big|\partial_{t}^{[a]+1} f\big|_{L^{\infty}(G_{r})}
|D_{x'}a|_{L^{\infty}(B'_{r})}r^{[a]+1-\nu}|x'_{1}-x'_{2}|\int^{1}_{0}\rho^{-c}|\log\rho|d\rho\\
&\leq C \big|\partial_{t}^{[a]+1} f\big|_{L^{\infty}(G_{r})}
r^{[a]+1-\nu}|x'_{1}-x'_{2}|. 
\end{align*}
In summary, we obtain 
$$|\partial_{t}^{\nu} F(x_1^{\prime}, t)-\partial_{t}^{\nu} F(x_2^{\prime}, t)|
\le CK|x_{1}^{\prime}-x_{2}^{\prime}|^{1+\alpha\delta-\delta},$$
if $-c+\delta+(1-\delta) \alpha> 0$, where 
$$
K=[\partial_{t}^{[a]+1} f]_{C^{\alpha}_{x'}(\bar{G}_{r})}^{\delta}
[\partial_{t}^{[a]}D_{x'} f]^{1-\delta}_{C^{\alpha}_{t}(\bar{G}_{r})}+\big|\partial_{t}^{[a]+1} f\big|_{L^{\infty}(G_{r})},
$$
and $C$ is a positive constant depending only on $r$, $\delta$, $c$, $\alpha$, 
and the $L^{\infty}$-norm of $D_{x'}a$ in $B'_r$.  
For any $\epsilon \in(\alpha, \alpha-c+1)$, take $\delta$ such that $1+\alpha\delta-\delta=\epsilon$. Then,
$$
\delta=\frac{1-\epsilon}{1-\alpha}\in(0,1),
$$
and $-c+\delta+(1-\delta) \alpha> 0$, and hence $\partial_{t}^{\nu} F \in C^{\epsilon}_{x'}(\bar{G}_{r})$. 
This holds for any $\nu \leq [a]$.  
\end{proof}

\begin{lemma}\label{lemma-ch4-singular-integral-non-int-higher}
Let $k$ and $l$ be nonnegative integers, $\alpha,\beta\in(0,1)$ be constants, 
and $a\in C^{l,\beta}(\bar{B}'_{r})$ and $f\in C(\bar{G}_{r})$ be functions, with $a>0$ in $\bar{B}'_{r}$. 
Suppose $k\geq [a]+1$. Define, for any $(x',t)\in G_{r}$,
$$
F(x',t)=t^{a(x')}\int^{t}_{0}s^{-a(x')-1}f(x',s)ds.
$$
Suppose $\partial^{\nu}_{t}D^{\tau}_{x'}f\in C^{\beta,\alpha}_{x',t}(\bar{G}_{r})$ 
and $\partial^{\nu}_{t}D^{\tau}_{x'}f(\cdot,0)=0$, for any $\tau\leq l$ and $\nu\leq k$. 
Then, $\partial^{\nu}_{t}D^{\tau}_{x'}F\in C^{\beta,\alpha}_{x',t}(\bar{G}_{r})$, for any $\tau\leq l$ and $\nu\leq k$. 
Moreover, for any $\tau \leq l$ and $\nu \leq k$, and any $(x^{\prime}, t) \in G_{r}$,
$$
|\partial_{t}^{\nu} D_{x^{\prime}}^{\tau} F(x^{\prime}, t)| 
\leq C\sum^{\tau}_{i=0}[\partial_{t}^{k} D_{x^{\prime}}^{i} f]_{C^{\alpha}_{t}(\bar{G}_{r})} t^{k-\nu+\alpha},
$$
where $C$ is a positive constant depending only on $|a|_{C^{l}(\bar{B}'_{r})}$, $k$, $l$, and $\alpha$. 
\end{lemma}

The proof is similar to that of Lemma \ref{lemma-ch4-singular-integral-non-int-lower-larger} and is omitted.

\begin{lemma}\label{lemma-ch4-singular-integral-int-lower}
Let $l$ and $a$ be nonnegative integers, $\alpha,\beta\in(0,1)$ be constants, and $f\in C(\bar{G}_{r})$ be a function. 
Define, for any $(x',t)\in G_{r}$,
$$
F(x',t)=t^{a}\int^{t}_{0}s^{-a-1}f(x',s)ds.
$$
Suppose $\partial^{\nu}_{t}D^{\tau}_{x'}f\in C^{\beta,\alpha}_{x',t}(\bar{G}_{r})$ 
and $\partial^{\nu}_{t}D^{\tau}_{x'}f(\cdot,0)=0$, for any $\tau\leq l$ and $\nu\leq a$. 
Then, $\partial^{\nu}_{t}D^{\tau}_{x'}F\in C^{\epsilon,\alpha}_{x',t}(\bar{G}_{r})$, 
for any $\tau\leq l$ and $\nu\leq a$, and any $\epsilon\in (0,\beta)$. 
Moreover, for any $\tau \leq l$ and $\nu \leq a$, and any $(x^{\prime}, t) \in G_{r}$,
$$
|\partial_{t}^{\nu} D_{x^{\prime}}^{\tau} F(x^{\prime}, t)| 
\leq C[\partial_{t}^{a} D_{x^{\prime}}^{\tau} f]_{C^{\alpha}_{t}(\bar{G}_{r})} t^{a-\nu+\alpha},
$$
where $C$ is a positive constant depending only on $a$ and $\alpha$.  
\end{lemma}

The proof is similar to that of Lemma \ref{lemma-ch4-singular-integral-non-int-lower-larger} and is omitted.

\begin{lemma}\label{lemma-ch4-singular-integral-int-higher} 
Let $k$, $l$, and $a$ be nonnegative integers with $k\ge a+1$, 
$\alpha, \beta\in (0,1)$ be constants, and $f\in C(\bar G_r)$ be a function. 
Define, 
for any $(x',t)\in G_r$, 
$$F(x', t)=t^{a}\int_0^ts^{-a-1}f(x',s)ds.$$
Suppose 
$\partial_t^\nu D_{x'}^\tau  f \in C^{\beta,\alpha}_{x',t}(\bar{G}_r)$ and 
$\partial_t^\nu D_{x'}^\tau f(\cdot, 0)=0$ for any 
$\tau\le l$ and $\nu\le k$. 
Then, $\partial_t^\nu D_{x'}^\tau F\in C^{\beta,\alpha}_{x',t}(\bar{G}_r),$ for any 
$\tau\le l$ and $\nu\le k$. 
Moreover, for any 
$\tau\le l$, $\nu\le k$, and any $(x',t)\in G_r$,  
$$|\partial_t^\nu D_{x'}^\tau F(x',t)|
\le C[\partial_t^k D_{x'}^\tau f]_{C^\alpha_t(\bar G_r)}t^{k-\nu+\alpha},$$ 
where $C$ is a positive constant depending only on $a$, $k$, and $\alpha$. 
\end{lemma}

The proof is similar to that of Lemma \ref{lemma-ch4-singular-integral-non-int-lower-larger} and is omitted.

\begin{lemma}\label{lemma-ch4-regularity-log}
Let $k$ and $l$ be nonnegative integers, $\alpha\in(0,1)$ be a constant, and $f\in C(\bar{G}_{r})$ be a function. 
Suppose $\partial^{\nu}_{t}D^{\tau}_{x'}f\in C^{\alpha}(\bar{G}_{r})$ and $\partial^{\nu}_{t}D^{\tau}_{x'}f(\cdot,0)=0$, 
for any $\tau\leq l$ and $\nu\leq k$. Then, 

$\mathrm{(i)}$ $\partial^{\nu}_{t}D^{\tau}_{x'}[f\log t]\in C^{\epsilon}(\bar{G}_{r})$, 
for any $\tau\leq l$ and $\nu\leq k$, 
and any $\epsilon\in(0,\alpha)$.

$\mathrm{(ii)}$ $\partial^{\nu}_{t}D^{\tau}_{x'}[f(\log t)^{-1}]\in C^{\alpha}(\bar{G}_{r})$, 
for any $\tau\leq l$ and $\nu\leq k$.  
\end{lemma}

The proof is standard and hence omitted. 

\begin{lemma}\label{lemma-ch4-regularity-power1}
Let $k$ and $l$ be nonnegative integers, $\alpha\in(0,1)$ be a constant, 
and  $\gamma\in C^{l,\alpha}(\bar{B}'_{r})$ and $f\in C(\bar{G}_{r})$ be functions, 
with $0<\gamma\leq c<1$ in $\bar{B}'_{r}$, for some positive constant $c<\alpha$. 
Suppose $\partial^{\nu}_{t}D^{\tau}_{x'}f\in C^{\alpha}(\bar{G}_{r})$ and $\partial^{\nu}_{t}D^{\tau}_{x'}f(\cdot,0)=0$, 
for any $\tau\leq l$ and $\nu\leq k$. 
Then, $\partial^{\nu}_{t}D^{\tau}_{x'}(t^{-\gamma}f)\in C^{\epsilon}(\bar{G}_{r})$, 
for any $\tau\leq l$ and $\nu\leq k$, and $\epsilon\in (0,\alpha-c)$. 
%and any $\epsilon\in(0,\alpha-c)$.
\end{lemma}

%{\color{red} (Check! The H\"older index is $\alpha-c$.)}

\begin{proof}
%For simplicity, we assume that $\gamma\equiv c$. 
%The discussion of the general case is similar.   
We only consider several special cases. 

{\it Case 1: $k=l=0$.} First, we have, for any $(x',t)\in \bar{G}_{r}$,
$$
|t^{-\gamma(x')}f(x',t)|\leq [f]_{C^{\alpha}(\bar{G}_{r})}t^{\alpha-\gamma(x')}
\leq [f]_{C^{\alpha}(\bar{G}_{r})}t^{\alpha-c}.
$$

Next, we estimate the H\"{o}lder semi-norm of $t^{-\gamma}f$ in $x'$. For any $(x_{1}',t)$, $(x_{2}',t)\in G_{r}$, we write 
$$t^{-\gamma(x_1')}f(x_1',t)-t^{-\gamma(x_2')}f(x_2',t)=I_1+I_2,$$
where 
\begin{align*} 
I_1&=t^{-\gamma(x_1')}[f(x_1',t)-f(x_2',t)],\\
I_2&=f(x_2',t)[t^{-\gamma(x_1')}-t^{-\gamma(x_2')}].
\end{align*}
We first analyze $I_1$. Note that 
$$
|f(x_{1}',t)-f(x_{2}',t)|\leq 2[f]_{C^{\alpha}_{t}(\bar{G}_{r})}t^{\alpha},
$$
and 
$$
|f(x_{1}',t)-f(x_{2}',t)|\leq [f]_{C^{\alpha}_{x'}(\bar{G}_{r})}|x_{1}'-x_{2}'|^{\alpha}.
$$
For a constant $\delta\in(0,1)$ to be determined, we have 
\begin{align*}
|f(x_{1}',t)-f(x_{2}',t)|&= |f(x_{1}',t)-f(x_{2}',t)|^{\delta}|f(x_{1}',t)-f(x_{2}',t)|^{1-\delta}   \\
&\leq 2[f]_{C^{\alpha}_{t}(\bar{G}_{r})}^{\delta}t^{\alpha\delta}
\cdot[f]_{C^{\alpha}_{x'}(\bar{G}_{r})}^{1-\delta}|x_{1}'-x_{2}'|^{\alpha(1-\delta)}\\
&\leq2[f]_{C^{\alpha}(\bar{G}_{r})}t^{\alpha\delta}|x_{1}'-x_{2}'|^{\alpha-\alpha\delta},
\end{align*}
and hence
$$
|I_1|\leq 2[f]_{C^{\alpha}(\bar{G}_{r})}
t^{\alpha\delta-\gamma(x_1')}|x_{1}'-x_{2}'|^{\alpha-\alpha\delta}
\leq 2[f]_{C^{\alpha}(\bar{G}_{r})}t^{\alpha\delta-c}|x_{1}'-x_{2}'|^{\alpha-\alpha\delta}.
$$
By taking $\delta={c}/{\alpha}\in(0,1)$, we obtain 
$$
|I_1|\leq 2[f]_{C^{\alpha}(\bar{G}_{r})}|x_{1}'-x_{2}'|^{\alpha-c}.
$$
Next, we analyze $I_2$. For any $0<t\le r\le 1$, we have 
$$|t^{-\gamma(x_1')}-t^{-\gamma(x_2')}|\le t^{-c}|\log t|\,
|\gamma(x_1')-\gamma(x_2')|\le Ct^{-c}|\log t|\,
|x_1'-x_2'|^\alpha,$$
and hence 
$$
|I_2|\leq C[f]_{C^{\alpha}(\bar{G}_{r})}t^{\alpha-c}|\log t|\,
|x_1'-x_2'|^\alpha
\leq C[f]_{C^{\alpha}(\bar{G}_{r})}
|x_1'-x_2'|^{\alpha-c},$$
by $c<\alpha$. 
In conclusion, we obtain 
\begin{align*}
|t^{-\gamma(x_1')}f(x_1',t)-t^{-\gamma(x_2')}f(x_2',t)|\leq  
C[f]_{C^{\alpha}(\bar{G}_{r})}|x_{1}'-x_{2}'|^{\alpha-c}.
\end{align*}
Thus, $t^{-\gamma}f\in C^{\alpha-c}_{x'}(\bar{G}_{r})$.

Last, we estimate the H\"{o}lder semi-norm of $t^{-\gamma}f$ in $t$. 
Take any  $(x',t_{1})$, $(x',t_{2})\in G_{r}$ with $t_{1}\leq t_{2}$. We write
$$
t^{-\gamma(x')}_{1}f(x',t_{1})-t^{-\gamma(x')}_{2}f(x',t_{2})=J_{1}+J_{2},
$$
where
\begin{align*}
J_{1}&= t^{-\gamma(x')}_{2}[f(x',t_{1})-f(x',t_{2})],\\
J_{2}&=[t^{-\gamma(x')}_{1}-t^{-\gamma(x')}_{2}]f(x',t_{1}).
\end{align*}
First, we get
\begin{align*}
|J_{1}|&\leq  [f]_{C^{\alpha}_{t}(\bar{G}_{r})}t^{-\gamma(x')}_{2}(t_{2}-t_{1})^{\alpha}\\
&=  [f]_{C^{\alpha}_{t}(\bar{G}_{r})}\Big(\frac{t_2-t_{1}}{t_{2}}\Big)^{\gamma(x')}(t_{2}-t_{1})^{\alpha-\gamma(x')}\\
&\leq  [f]_{C^{\alpha}_{t}(\bar{G}_{r})}(t_{2}-t_{1})^{\alpha-\gamma(x')}.
\end{align*}
Second, we have
\begin{align*}
|J_{2}|&\leq [f]_{C^{\alpha}_{t}(\bar{G}_{r})}t^{\alpha}_{1}
(t^{-\gamma(x')}_{1}-t^{-\gamma(x')}_{2})\\
&= [f]_{C^{\alpha}_{t}(\bar{G}_{r})}t^{\alpha-\gamma(x')}_{1}\Big(1-\Big(\frac{t_{1}}{t_{2}}\Big)^{\gamma(x')}\Big)\\
&\leq [f]_{C^{\alpha}_{t}(\bar{G}_{r})}t^{\alpha-\gamma(x')}_{1}\Big(1-\frac{t_{1}}{t_{2}}\Big)\\
&=[f]_{C^{\alpha}_{t}(\bar{G}_{r})}\Big(\frac{t_{2}-t_{1}}{t_{2}}\Big)^{1-(\alpha-\gamma(x'))}
\Big(\frac{t_{1}}{t_{2}}\Big)^{\alpha-\gamma(x')}(t_{2}-t_{1})^{\alpha-\gamma(x')}\\
&\leq [f]_{C^{\alpha}_{t}(\bar{G}_{r})}(t_{2}-t_{1})^{\alpha-\gamma(x')}.
\end{align*}
In conclusion, we obtain
$$
|t^{-\gamma(x')}_{1}f(x',t_{1})-t^{-\gamma(x')}_{2}f(x',t_{2})|
\leq 2[f]_{C^{\alpha}_{t}(\bar{G}_{r})}(t_{2}-t_{1})^{\alpha-c}.
$$
Thus, $t^{-\gamma}f\in C^{\alpha-c}_{t}(\bar{G}_{r})$.

{\it Case 2: $k=1$ and $l=0$.} Note that 
$$\partial_t(t^{-\gamma}f)=t^{-\gamma}\partial_tf-\gamma t^{-\gamma}\,\frac{f}{t}.$$
By what we proved in Case 1, we conclude $\partial_t(t^{-\gamma}f)\in C^{\alpha-c}(\bar{G}_{r})$.

{\it Case 3: $k=0$ and $l=1$.} Note that 
$$D_{x'}(t^{-\gamma}f)=t^{-\gamma}D_{x'}f-t^{-\gamma}\log t\, fD_{x'}\gamma.$$
By what we proved in Case 1, we have 
$t^{-\gamma}D_{x'}f, t^{-\gamma}f\in C^{\alpha-c}(\bar{G}_{r})$.
By Lemma \ref{lemma-ch4-regularity-log}(i), we obtain 
$D_{x'}(t^{-\gamma}f)\in C^{\epsilon}(\bar{G}_{r})$ for any $\epsilon\in (0,\alpha-c)$. 
\end{proof}

\begin{lemma}\label{lemma-ch4-regularity-power2}
Let $k$ and $l$ be nonnegative integers with $k\geq 1$, $\alpha\in(0,1)$ be a constant, 
and $\gamma\in C^{l,\alpha}(\bar{B}'_{r})$ and $f\in C(\bar{G}_{r})$ be functions, 
with $0<\gamma\leq c<1$ in $\bar{B}'_{r}$, for some positive constant $c>\alpha$. 

$\mathrm{(i)}$ Suppose $\partial^{\nu}_{t}D^{\tau}_{x'}f\in C^{\alpha}(\bar{G}_{r})$ 
and $\partial^{\nu}_{t}D^{\tau}_{x'}f(\cdot,0)=0$, for any $\tau\leq l+1$ and $\nu\leq k-1$, and 
$\partial^{\nu}_{t}D^{\tau}_{x'}f\in C^{\alpha}(\bar{G}_{r})$, for any $\tau\leq l$ and $\nu\leq k$. 
Then, $\partial^{\nu}_{t}D^{\tau}_{x'}(t^{-\gamma}f)\in C^{\epsilon,\delta}_{x',t}(\bar{G}_{r})$, 
for any $\tau\leq l$ and $\nu\leq k-1$, any $\epsilon\in(0,1+\alpha-c)$, and any $\delta\in(0,1-c)$.

$\mathrm{(ii)}$ Suppose $\partial^{\nu}_{t}D^{\tau}_{x'}f\in C^{\alpha}(\bar{G}_{r})$ 
and $\partial^{\nu}_{t}D^{\tau}_{x'}f(\cdot,0)=0$, for any $\tau\leq l$ and $\nu\leq k$. 
Then, $\partial^{\nu}_{t}D^{\tau}_{x'}(t^{-\gamma}f)\in C^{\epsilon}_{t}(\bar{G}_{r})$, 
for any $\tau\leq l$ and $\nu\leq k-1$, and any $\epsilon\in(0,1+\alpha-c)$.
\end{lemma}

The proof is similar to that of Lemma \ref{lemma-ch4-regularity-power1} and is omitted.

%\begin{remark}
%\textcolor{red}{In the assumptions and conclusions of Lemma 4.1 and Lemma 4.3, 
%we can replace $C^{\beta,\alpha}_{x',t}(\bar{G}_{r})$ with $C^{\beta}_{x'}(\bar{G}_{r})$ 
%or $C^{\alpha}_{t}(\bar{G}_{r})$ for regularity.}
%\end{remark}  

\end{document}